\documentclass[11pt]{article}
\usepackage{amsmath,amssymb,amsthm,amscd}
\usepackage[all]{xy}

\newtheorem{theorem}{Theorem}[section]
\newtheorem{lemma}[theorem]{Lemma}
\newtheorem{proposition}[theorem]{Proposition}
\newtheorem{corollary}[theorem]{Corollary}

\theoremstyle{definition}
\newtheorem{definition}[theorem]{Definition}
\newtheorem{remark}[theorem]{Remark}

\numberwithin{equation}{section}

\renewcommand{\phi}{\varphi}

\newcommand{\ep}{\varepsilon}

\newcommand{\Coker}{\operatorname{Coker}}
\newcommand{\Homeo}{\operatorname{Homeo}}
\newcommand{\id}{\operatorname{id}}
\newcommand{\height}{\operatorname{height}}
\newcommand{\Ima}{\operatorname{Im}}
\newcommand{\Ker}{\operatorname{Ker}}
\newcommand{\rank}{\operatorname{rank}}
\newcommand{\supp}{\operatorname{supp}}

\newcommand{\V}{\mathcal{V}}
\newcommand{\E}{\mathcal{E}}
\newcommand{\B}{\mathcal{B}}

\newcommand{\N}{\mathbb{N}}
\newcommand{\Z}{\mathbb{Z}}

\newcommand{\R}{\mathbb{R}}

\newcommand{\F}{\mathbb{F}}

\title{Topological full groups of one-sided shifts of finite type}
\author{Hiroki Matui \\
Graduate School of Science \\
Chiba University \\
Inage-ku, Chiba 263-8522, Japan}
\date{}

\begin{document}
\maketitle

\begin{abstract}
We explore the topological full group $[[G]]$ of 
an essentially principal \'etale groupoid $G$ on a Cantor set. 
When $G$ is minimal, 
we show that $[[G]]$ (and its certain normal subgroup) is a complete invariant 
for the isomorphism class of the \'etale groupoid $G$. 
Furthermore, when $G$ is either almost finite or purely infinite, 
the commutator subgroup $D([[G]])$ is shown to be simple. 
The \'etale groupoid $G$ 
arising from a one-sided irreducible shift of finite type is 
a typical example of a purely infinite minimal groupoid. 
For such $G$, 
$[[G]]$ is thought of as a generalization of the Higman-Thompson group. 
We prove that $[[G]]$ is of type F$_\infty$, 
and so in particular it is finitely presented. 
This gives us a new infinite family of 
finitely presented infinite simple groups. 
Also, the abelianization of $[[G]]$ is calculated and described 
in terms of the homology groups of $G$. 
\end{abstract}

\tableofcontents

%%%%%%%%%%%%%%%%%%%%%%%%%%%%%%%%%%%%%%%%%%%%%%%%%%%%%%%%%%%%
\section{Introduction}

We study various properties of topological full groups of 
topological dynamical systems on Cantor sets. 
H. Dye \cite{D59AJM,D63AJM} introduced the notion of 
full groups for ergodic measure preserving actions of countable groups and 
proved that the full group is a complete invariant of orbit equivalence. 
The study of full groups in the setting of topological dynamics 
was initiated by T. Giordano, I. F. Putnam and C. F. Skau \cite{GPS99Israel}. 
For a minimal action $\phi:\Z\curvearrowright X$ on a Cantor set $X$, 
they defined several types of full groups and showed that 
these groups completely determine the orbit equivalence class, 
the strong orbit equivalence class and the flip conjugacy class of $\phi$, 
respectively. 

The notion of topological full groups was later generalized to the setting of 
essentially principal \'etale groupoids $G$ on Cantor sets in \cite{M12PLMS}. 
\'Etale groupoids (called $r$-discrete groupoids in \cite{Rtext}) provide us 
a natural framework for unified treatment of 
various topological dynamical systems. 
The topological full group $[[G]]$ of $G$ is a subgroup of $\Homeo(G^{(0)})$ 
consisting of all homeomorphisms of $G^{(0)}$ 
whose graph is `contained' in the groupoid $G$ as a compact open subset 
(see Definition \ref{defoftfg}). 
From an action $\phi$ of a discrete group $\Gamma$ on a Cantor set $X$, 
we can construct the \'etale groupoid $G_\phi$, 
which is called the transformation groupoid. 
The topological full group $[[G_\phi]]$ of $G_\phi$ is 
the group of $\alpha\in\Homeo(X)$ 
for which there exists a continuous map $c:X\to\Gamma$ such that 
$\alpha(x)=\phi_{c(x)}(x)$ for all $x\in X$. 
Another important class of \'etale groupoids is the AF groupoids. 
The terminology AF comes from $C^*$-algebra theory and 
means approximately finite. 
AF groupoids have played a crucial role 
in the study of orbit equivalence for minimal $\Z^N$-actions 
(\cite{GMPS}). 
The \'etale groupoids $G$ 
arising from one-sided irreducible shifts of finite type $(X,\sigma)$ 
have been also studied by many people. 
The groupoid $C^*$-algebras $C^*_r(G)$ are called Cuntz-Krieger algebras 
(\cite{CK80Invent}). 
V. V. Nekrashevych \cite{Ne04JOP} observed that 
$[[G]]$ is naturally isomorphic to the Higman-Thompson group $V_{n,1}$ 
when $(X,\sigma)$ is the full shift over $n$ symbols 
(see Remark \ref{HigThomp}). 
Thus, the groups $[[G]]$ for shifts of finite type are 
regarded as a generalization of $V_{n,1}$. 
For two \'etale groupoids $G_1$ and $G_2$ 
arising from one-sided irreducible shifts of finite type, 
K. Matsumoto \cite{Matsu1205} has recently proved that 
the topological full groups $[[G_1]]$ and $[[G_2]]$ are isomorphic as groups 
if and only if $G_1$ and $G_2$ are isomorphic as \'etale groupoids. 
This is analogous to the above mentioned result for minimal $\Z$-actions. 
Furthermore, Matsumoto \cite{MatsuProcAMS} gave a sufficient condition 
on the subshifts under which the corresponding groupoids are isomorphic 
(see Theorem \ref{classify}). 
Clearly the group $[[G]]$ embeds 
in the Leavitt path algebra associated with $(X,\sigma)$. 
This may indicate possible connections to the study of Leavitt path algebras. 

In the present paper, 
we first extend the results of \cite{GPS99Israel,Matsu1205}. 
More precisely, 
for given essentially principal minimal \'etale groupoids $G_1$ and $G_2$ 
on Cantor sets and 
normal subgroups $\Gamma_i\subset[[G_i]]$ 
containing the commutator subgroups $D([[G_i]])$, 
we will show that any isomorphism between $\Gamma_1$ and $\Gamma_2$ 
is realized by a homeomorphism between the unit spaces $G_i^{(0)}$ 
(Theorem \ref{spatial} and Proposition \ref{tfg=F}). 
This means that if $\Gamma_1$ and $\Gamma_2$ are isomorphic as groups, 
then $G_1$ and $G_2$ are isomorphic as \'etale groupoids (Theorem \ref{iso}). 
In particular, $[[G_i]]$ (or $D([[G_i]])$) are isomorphic to each other 
if and only if $G_i$ are isomorphic to each other (Theorem \ref{iso2}). 
In other words, we can say that the topological full group $[[G]]$ 
(and its commutator subgroup $D([[G]])$) `remembers' the \'etale groupoid $G$. 
We remark that similar results are also obtained 
by S. Bezuglyi and K. Medynets \cite{BM08Coll,Me11BLMS}. 

In Section 4 we study simplicity properties of $[[G]]$. 
The notion of almost finite groupoids was introduced in \cite{M12PLMS}. 
Almost finiteness is a weak version of approximate finiteness. 
The transformation groupoids of free actions of $\Z^N$ are 
known to be almost finite (\cite[Lemma 6.3]{M12PLMS}). 
We show that if $G$ is almost finite and minimal, 
then the commutator subgroup $D([[G]])$ is simple (Theorem \ref{simple1}). 
This is a generalization of 
\cite[Theorem 4.9]{M06IJM} and \cite[Theorem 3.4 (2)]{BM08Coll} 
for minimal $\Z$-actions. 
We also present a new class of \'etale groupoids, 
namely purely infinite groupoids. 
The definition of pure infiniteness is inspired 
by the work of M. R\o rdam and A. Sierakowski \cite{RS12ETDS}
(see Definition \ref{pi}). 
Transformation groupoids arising from $n$-filling actions 
in the sense of P. Jolissaint and G. Robertson \cite{JR00JFA} are 
purely infinite and minimal (Remark \ref{n-filling}). 
In contrast to almost finite groupoids, 
purely infinite groupoids admit no invariant probability measures. 
It will be proven that if $G$ is purely infinite and minimal, 
then the commutator subgroup $D([[G]])$ is simple (Theorem \ref{simple2}). 
We also observe that 
the topological full group $[[G]]$ for a purely infinite groupoid $G$ 
contains the free product $\Z_2*\Z_3$ (Theorem \ref{pi>Z2*Z3}), 
and so is not amenable. 

In Section 5 we prove that for any purely infinite groupoid $G$, 
the index map $I:[[G]]\to H_1(G)$ is surjective, 
where $H_1(G)$ stands for the homology group of $G$. 
The index map $I$ is a homomorphism introduced in \cite[Section 7]{M12PLMS}. 
We denote by $[[G]]_0$ the kernel of $I$. 
For almost finite groupoids, 
the surjectivity of $I$ is already known (\cite[Theorem 7.5]{M12PLMS}). 
Therefore, if $G$ is almost finite or purely infinite, 
then the abelianization of $[[G]]$ at least has 
$H_1(G)\cong[[G]]/[[G]]_0$ as its quotient. 

Section 6 is devoted to the study of the \'etale groupoids $G$ 
arising from one-sided irreducible shifts of finite type $(X,\sigma)$. 
Shifts of finite type (also called topological Markov shifts) form 
the most prominent class of symbolic dynamical systems. 
A good introduction to symbolic dynamics can be found 
in the standard textbook \cite{LM} by D. Lind and B. Marcus. 
From $(X,\sigma)$, we can construct the \'etale groupoid $G$ by 
\[
G=\{(x,n,y)\in X\times\Z\times X\mid
\exists k,l\in\N,\ n=k{-}l,\ \sigma^k(x)=\sigma^l(y)\}. 
\]
By \cite[Theorem 4.14]{M12PLMS}, the homology groups $H_n(G)$ of $G$ are 
\[
H_n(G)=\begin{cases}\Coker(\id-M^t)&n=0\\
\Ker(\id-M^t)&n=1\\0&n\geq2, \end{cases}
\]
where $M$ is the $k\times k$ matrix with entries in $\Z_+$ 
representing $(X,\sigma)$ and 
is thought of as a homomorphism from $\Z^k$ to $\Z^k$. 
Notice that $H_0(G)$ is a finitely generated abelian group and 
$H_1(G)$ is isomorphic to the torsion-free part of $H_0(G)$. 
For any non-empty clopen set $Y\subset X$, the reduction $G|Y$ of $G$ to $Y$ 
(see Section 2 for the definition) is again an \'etale groupoid. 

For the topological full group $[[G|Y]]$ 
associated to an irreducible shift of finite type, we have the following. 
\begin{itemize}
\item $G|Y$ is purely infinite and minimal (Theorem \ref{SFT>pim}). 
Hence the commutator subgroup $D([[G|Y]])$ is simple by Theorem \ref{simple2}. 
\item By Theorem \ref{iso2}, $[[G|Y]]$ (or $[[G|Y]]_0$ or $D([[G|Y]])$) is 
a complete invariant of $G|Y$. 
\item By K. Matsumoto's theorem (Theorem \ref{classify}), 
for two groupoids $G_1|Y_1$ and $G_2|Y_2$ as above, 
if there exists an isomorphism $\phi:H_0(G_1)\to H_0(G_2)$ 
such that $\phi([1_{Y_1}]_{G_1})=[1_{Y_2}]_{G_2}$ and 
$\det(\id-M_1^t)=\det(\id-M_2^t)$, then 
$G_1|Y_1$ is isomorphic to $G_2|Y_2$ as an \'etale groupoid. 
\item $[[G|Y]]$ is not amenable (see Theorem \ref{pi>Z2*Z3}), 
but has the Haagerup property (Theorem \ref{Haagerup}). 
This is essentially due to B. Hughes's theorem \cite[Theorem 1.1]{H09GGD}. 
\item $[[G|Y]]$ is of type F$_\infty$ (Theorem \ref{finite}). 
Thus, $[[G|Y]]$ is finitely presented and is of type FP$_\infty$. 
This is a generalization of K. S. Brown's result \cite{Br87JPAA}, 
which says that the Higman-Thompson group $V_{n,r}$ is of type F$_\infty$ 
(see also Remark \ref{HigThomp}). 
\item The abelianization $[[G|Y]]/D([[G|Y]])$ is isomorphic to 
$(H_0(G)\otimes\Z_2)\oplus H_1(G)$ (Corollary \ref{abel}). 
In particular, $[[G|Y]]$ is simple if and only if $H_0(G)$ is $2$-divisible. 
\item If $H_1(G)=0$, then 
$[[G|Y]]_0$ and $D([[G|Y]])$ are of type F$_\infty$ 
(Corollary \ref{Dfinite} (2)). 
\item $[[G|Y]]_0$ and $D([[G|Y]])$ are finitely generated 
(Theorem \ref{fg} and Corollary \ref{Dfinite} (1)). 
\item Examples are given in Section 6.7. 
Among others, 
the boundary action $\phi_k$ of the free group $\F_k$ is discussed 
(Section 6.7.5). 
The associated transformation groupoid $G_{\phi_k}$ is naturally isomorphic 
to an \'etale groupoid of an irreducible shift of finite type. 
Hence $[[G_{\phi_k}]]$ is of type F$_\infty$ and 
the commutator subgroup $D([[G_{\phi_k}]])$ is simple. 
The abelianization of $[[G_{\phi_k}]]$ is $\Z^k\oplus(\Z_2)^k$ 
when $k$ is even and is $\Z^k\oplus(\Z_2)^{k+1}$ when $k$ is odd. 
If $k\neq l$, then 
$D([[G_{\phi_k}]])$ is not isomorphic to $D([[G_{\phi_l}]])$. 
\end{itemize}
We have to explain some terminologies in group theory used above. 
For a natural number $n\in\N$, a group $\Gamma$ is said to be of type F$_n$ 
if it admits a $K(\Gamma,1)$ with finite $n$-skeleton, and 
$\Gamma$ is said to be of type F$_\infty$ 
if it admits a $K(\Gamma,1)$ with finite $n$-skeleton in all dimensions $n$ 
(see \cite[Section 7.2]{Geotext}). 
A group is of type F$_1$ if and only if it is finitely generated, 
and a group is of type F$_2$ if and only if it is finitely presented 
(\cite[Proposition 7.2.1]{Geotext}). 
The properties F$_n$ and F$_\infty$ are called 
topological finiteness properties of $\Gamma$. 
A group $\Gamma$ is said to be of type FP$_n$ (resp. FP$_\infty$) 
if the trivial $\Z\Gamma$-module $\Z$ admits a projective resolution 
which is finitely generated in dimensions not greater than $n$ 
(resp. in all dimensions), cf. \cite[Section 8.2]{Geotext}. 
The properties FP$_n$ and FP$_\infty$ are called 
homological finiteness properties of $\Gamma$. 
It is clear that F$_n$ implies FP$_n$. 
It is known that $\Gamma$ is of type F$_\infty$ 
if and only if it is of type FP$_\infty$ and is finitely presented 
(see the proof of \cite[Theorem VIII.7.1]{BrGTM}). 

We would like to mention various known results 
for the topological full group $[[G_\phi]]$ of 
a minimal action $\phi:\Z\curvearrowright X$ on a Cantor set $X$. 
While this is not directly relevant to the present article, 
the reader may find interesting similarities 
to the case of shifts of finite type. 
By \cite[Corollary 4.4]{GPS99Israel} and \cite[Theorem 5.13]{BM08Coll}, 
$[[G_\phi]]$ (or $[[G_\phi]]_0$ or $D([[G_\phi]])$) is 
a complete invariant of $G_\phi$ (see also Theorem \ref{iso}). 
By Theorem \ref{simple1}, the commutator subgroup $D([[G_\phi]])$ is simple 
(this was first proved in \cite[Theorem 4.9]{M06IJM}). 
The abelianization $[[G_\phi]]/D([[G_\phi]])$ is known to be 
isomorphic to $(H_0(G_\phi)\otimes\Z_2)\oplus\Z$ 
(see \cite[Theorem 4.8]{M06IJM}). 
Note that $H_1(G_\phi)$ is always $\Z$ in this setting. 
By \cite[Theorem 5.4]{M06IJM}, $D([[G_\phi]])$ is finitely generated 
if and only if $\phi$ is expansive 
(i.e. $\phi$ is conjugate to a two-sided minimal subshift). 
However, $[[G_\phi]]$, $[[G_\phi]]_0$ and $D([[G_\phi]])$ are never 
finitely presented (\cite[Theorem 5.7, Corollary 5.8]{M06IJM}). 
When $D([[G_\phi]])$ is finitely generated, it is known that 
$D([[G_\phi]])$ has exponential growth (\cite[Corollary 2.5]{M1111}). 
R. Grigorchuk and K. Medynets \cite{GM} proved that 
$[[G_\phi]]$ is locally embeddable into finite groups, 
and conjectured that $[[G_\phi]]$ is amenable. 
K. Juschenko and N. Monod \cite{JM1204} recently confirmed this conjecture. 
This provided the first examples of 
finitely generated simple amenable infinite groups.

%%%%%%%%%%%%%%%%%%%%%%%%%%%%%%%%%%%%%%%%%%%%%%%%%%%%%%%%%%%%
\section{Preliminaries}

The cardinality of a set $A$ is written $\#A$ and 
the characteristic function of $A$ is written $1_A$. 
The finite cyclic group of order $n$ is denoted by $\Z_n$. 
We say that a subset of a topological space is clopen 
if it is both closed and open. 
A topological space is said to be totally disconnected 
if its topology is generated by clopen subsets. 
By a Cantor set, 
we mean a compact, metrizable, totally disconnected space 
with no isolated points. 
It is known that any two such spaces are homeomorphic. 
The homeomorphism group of a topological space $X$ is written $\Homeo(X)$. 
For $\alpha\in\Homeo(X)$, we let $\supp(\alpha)$ denote 
the closure of $\{x\in X\mid\alpha(x)\neq x\}$. 
The commutator of $\alpha,\beta\in\Homeo(X)$ is 
$[\alpha,\beta]=\alpha\beta\alpha^{-1}\beta^{-1}$. 
The commutator subgroup of a group $\Gamma$ is denoted by $D(\Gamma)$. 

In this article, by an \'etale groupoid 
we mean a second countable locally compact Hausdorff groupoid 
such that the range map is a local homeomorphism. 
We refer the reader to \cite{Rtext,R08Irish} 
for background material on \'etale groupoids. 
For an \'etale groupoid $G$, 
we let $G^{(0)}$ denote the unit space and 
let $s$ and $r$ denote the source and range maps. 
For $x\in G^{(0)}$, 
$G(x)=r(Gx)$ is called the $G$-orbit of $x$. 
When every $G$-orbit is dense in $G^{(0)}$, 
$G$ is said to be minimal. 
For an open subset $Y\subset G^{(0)}$, 
the reduction of $G$ to $Y$ is $r^{-1}(Y)\cap s^{-1}(Y)$ and 
denoted by $G|Y$. 
The reduction $G|Y$ is an \'etale subgroupoid of $G$ in an obvious way. 
The isotropy bundle is $G'=\{g\in G\mid r(g)=s(g)\}$. 
We say that $G$ is principal if $G'=G^{(0)}$. 
When the interior of $G'$ is $G^{(0)}$, 
we say that $G$ is essentially principal. 
A subset $U\subset G$ is called a $G$-set if $r|U,s|U$ are injective. 
When $U,V$ are $G$-sets, 
\[
U^{-1}=\{g\in G\mid g^{-1}\in U\}
\]
and 
\[
UV=\{gg'\in G\mid g\in U,\ g'\in V,\ s(g)=r(g')\}
\]
are also $G$-sets. 
For an open $G$-set $U$, 
we let $\pi_U$ denote the homeomorphism $r\circ(s|U)^{-1}$ 
from $s(U)$ to $r(U)$. 
For any compact open $G$-set $U$, 
$G|r(U)$ is naturally isomorphic to $G|s(U)$. 
A probability measure $\mu$ on $G^{(0)}$ is said to be $G$-invariant 
if $\mu(r(U))=\mu(s(U))$ holds for every open $G$-set $U$. 
The set of all $G$-invariant measures is denoted by $M(G)$. 
For an \'etale groupoid $G$, 
we denote the reduced groupoid $C^*$-algebra of $G$ by $C^*_r(G)$ and 
identify $C_0(G^{(0)})$ with a subalgebra of $C^*_r(G)$. 
J. Renault's theorem \cite[Theorem 4.11]{R08Irish} tells us that 
two essentially principal \'etale groupoids $G_1$ and $G_2$ are isomorphic 
if and only if there exists an isomorphism $\phi:C^*_r(G_1)\to C^*_r(G_2)$ 
such that $\phi(C_0(G_1^{(0)}))=C_0(G_2^{(0)})$ 
(see also \cite[Theorem 5.1]{M12PLMS}). 

We would like to recall 
the notion of topological full groups for \'etale groupoids. 

\begin{definition}[{\cite[Definition 2.3]{M12PLMS}}]\label{defoftfg}
Let $G$ be an essentially principal \'etale groupoid 
whose unit space $G^{(0)}$ is compact. 
\begin{enumerate}
\item The set of all $\alpha\in\Homeo(G^{(0)})$ such that 
for every $x\in G^{(0)}$ there exists $g\in G$ 
satisfying $r(g)=x$ and $s(g)=\alpha(x)$ 
is called the full group of $G$ and denoted by $[G]$. 
\item The set of all $\alpha\in\Homeo(G^{(0)})$ for which 
there exists a compact open $G$-set $U$ satisfying $\alpha=\pi_U$ 
is called the topological full group of $G$ and denoted by $[[G]]$. 
\end{enumerate}
Obviously $[G]$ is a subgroup of $\Homeo(G^{(0)})$ and 
$[[G]]$ is a subgroup of $[G]$. 
\end{definition}

For $\alpha\in[[G]]$ the compact open $G$-set $U$ as above uniquely exists, 
because $G$ is essentially principal. 
Obviously $[[G]]$ is a subgroup of $\Homeo(G^{(0)})$. 
Since $G$ is second countable, it has countably many compact open subsets, 
and so $[[G]]$ is at most countable. 
By \cite[Proposition 5.6]{M12PLMS}, 
we have the following short exact sequence: 
\[
1\longrightarrow U(C(G^{(0)}))\longrightarrow
N(C(G^{(0)}),C^*_r(G))\stackrel{\sigma}{\longrightarrow}[[G]]\longrightarrow1, 
\]
where $U(C(G^{(0)}))$ denotes the group of unitaries in $C(G^{(0)})$ and 
$N(C(G^{(0)}),C^*_r(G))$ denotes the group of unitaries in $C^*_r(G)$ 
which normalize $C(G^{(0)})$. 
In addition, the homomorphism $\sigma$ has a right inverse. 

\begin{lemma}\label{suppclp}
Let $G$ be an essentially principal \'etale groupoid 
whose unit space $G^{(0)}$ is compact. 
Then $\supp(\alpha)$ is clopen for any $\alpha\in[[G]]$. 
\end{lemma}
\begin{proof}
There exists a compact open $G$-set $U$ such that $\alpha=\pi_U$. 
The map $s|U:U\to G^{(0)}$ is a homeomorphism. 
It suffices to show that $\supp(\alpha)=s(U\setminus G^{(0)})$. 
Since $G^{(0)}$ is open and $s(g)=g=r(g)$ for $g\in G^{(0)}$, 
it is evident that $\supp(\alpha)$ is contained in $s(U\setminus G^{(0)})$. 
It is also clear that $s(U\setminus G')$ is contained in $\supp(\alpha)$. 
Therefore $s(U\setminus G^{(0)})$ is contained in $\supp(\alpha)$, 
because $G$ is essentially principal. 
\end{proof}

In \cite[Section 7]{M12PLMS}, 
we introduced the index map $I:[[G]]\to H_1(G)$, 
where $H_1(G)$ is the homology group of $G$ (see \cite[Section 3]{M12PLMS}). 
For $\alpha\in[[G]]$, 
let $U\subset G$ be the compact open $G$-set such that $\alpha=\pi_U$. 
Then the characteristic function $1_U\in C_c(G,\Z)$ is a cycle, 
and $I(\alpha)$ is the equivalence class of $1_U$ in $H_1(G)$. 
The index map $I$ is a homomorphism. 
We let $[[G]]_0$ denote the kernel of the index map $I$. 
Evidently $D([[G]])$ is contained in $[[G]]_0$. 
The main objective of the present paper is 
to study various properties of 
the groups $D([[G]])\subset[[G]]_0\subset[[G]]$.

%%%%%%%%%%%%%%%%%%%%%%%%%%%%%%%%%%%%%%%%%%%%%%%%%%%%%%%%%%%%
\section{A spatial realization theorem}

In this section, we will prove that 
any isomorphism between (certain normal subgroups of) topological full groups 
is realized by a homeomorphism of the underlying spaces 
(Theorem \ref{spatial}, Proposition \ref{tfg=F}). 
In particular, 
these various groups (as abstract groups) are complete invariants 
for \'etale groupoids (Theorem \ref{iso2}). 
Our proof of Theorem \ref{spatial} is 
essentially the same as that of \cite[Theorem 384D]{Fr}, 
and a similar argument can be found in \cite[Section 5]{BM08Coll}. 

\begin{definition}\label{F}
Let $X$ be a Cantor set and 
let $\Gamma\subset\Homeo(X)$ be a subgroup. 
We say that $\Gamma$ is of class F 
if the following conditions are satisfied. 
\begin{enumerate}
\item[(F0)]
For any $\alpha\in\Gamma$ satisfying $\alpha^2=1$, $\supp(\alpha)$ is clopen. 
\item[(F1)]
For any $x\in X$ and any clopen neighborhood $A\subset X$ of $x$, 
there exists $\alpha\in\Gamma\setminus\{1\}$ 
such that $x\in\supp(\alpha)$, $\supp(\alpha)\subset A$ and $\alpha^2=1$. 
\item[(F2)]
For any $\alpha\in\Gamma\setminus\{1\}$ satisfying $\alpha^2=1$ 
and any non-empty clopen set $A\subset\supp(\alpha)$, 
there exists $\beta\in\Gamma\setminus\{1\}$ such that 
$\supp(\beta)\subset A\cup\alpha(A)$ and $\beta(x)=\alpha(x)$ 
for every $x\in\supp(\beta)$. 
\item[(F3)]
For any non-empty clopen set $A\subset X$, there exists $\alpha\in\Gamma$ 
such that $\supp(\alpha)\subset A$ and $\alpha^2\neq1$. 
\end{enumerate}
\end{definition}

For $A\subset X$, we write 
\[
\Gamma(A)=\{\alpha\in\Gamma\mid\supp(\alpha)\subset A\}. 
\]
A closed set $A$ is said to be regular 
if it is equal to the closure of its interior. 

\begin{lemma}\label{localsubgrp}
Let $X$ be a Cantor set and 
let $\Gamma\subset\Homeo(X)$ be a subgroup of class F. 
For two regular closed sets $A,B\subset X$, 
$\Gamma(A)\subset\Gamma(B)$ if and only if $A\subset B$. 
\end{lemma}
\begin{proof}
It is clear that $A\subset B$ implies $\Gamma(A)\subset\Gamma(B)$. 
Suppose that $A\setminus B$ is not empty. 
Since $A$ is regular, $A\setminus B$ has non-empty interior. 
It follows from (F1) (or (F3)) that 
there exists $\alpha\in\Gamma\setminus\{1\}$ 
such that $\supp(\alpha)\subset A\setminus B$. 
Hence $\Gamma(A)$ is not contained in $\Gamma(B)$. 
\end{proof}

Let $X$ be a Cantor set and 
let $\Gamma\subset\Homeo(X)$ be a subgroup of class F. 
For $\tau\in\Gamma\setminus\{1\}$ satisfying $\tau^2=1$, 
we define $C_\tau$, $U_\tau$, $S_\tau$ and $W_\tau$ as follows. 
First, let $C_\tau$ be the centralizer of $\tau$, that is, 
\[
C_\tau=\{\alpha\in\Gamma\mid[\alpha,\tau]=1\}. 
\]
Define a subset $U_\tau\subset C_\tau$ by 
\[
U_\tau=\{\sigma\in C_\tau\mid\sigma^2=1,\quad 
[\sigma,\alpha\sigma\alpha^{-1}]=1\quad\forall\alpha\in C_\tau\}. 
\]
We define $S_\tau$ by 
\[
S_\tau=\{\alpha^2\in\Gamma\mid\alpha\in\Gamma,\ 
[\alpha,\sigma]=1\quad\forall\sigma\in U_\tau\}. 
\]
Then, we let $W_\tau$ be the centralizer of $S_\tau$, that is, 
\[
W_\tau=\{\alpha\in\Gamma\mid[\alpha,\beta]=1\quad\forall\beta\in S_\tau\}. 
\]
For these subsets, we can prove the following. 

\begin{lemma}\label{W}
Let the notation be as above. 
\begin{enumerate}
\item For any $\sigma\in U_\tau$ and $x\notin\supp(\tau)$, 
we have $\sigma(x)=x$. 
\item For any non-empty clopen set $A\subset\supp(\tau)$, 
there exists $\sigma\in U_\tau$ such that 
$\supp(\sigma)\subset A\cup\tau(A)$ and $\sigma(x)=\tau(x)$ 
for every $x\in\supp(\sigma)$. 
\item For any non-empty clopen set $A$ such that $A\cap\supp(\tau)=\emptyset$, 
there exists $\alpha\in S_\tau\setminus\{1\}$ 
such that $\supp(\alpha)\subset A$. 
\item For any $\alpha\in S_\tau$ and $x\in\supp(\tau)$, $\alpha(x)=x$. 
\item We have $W_\tau=\Gamma(\supp(\tau))$. 
\end{enumerate}
\end{lemma}
\begin{proof}
(1)
Let $\sigma\in U_\tau$ and $x\notin\supp(\tau)$. 
Assume $\sigma(x)\neq x$. 
There exists a clopen neighborhood $A\subset X$ of $x$ 
such that $A\cap\sigma(A)=\emptyset$ and $A\cap\supp(\tau)=\emptyset$. 
By (F3), we can find $\alpha\in\Gamma$ 
such that $\supp(\alpha)\subset A$ and $\alpha^2\neq1$. 
Clearly $\alpha$ commutes with $\tau$, and so $\alpha$ is in $C_\tau$. 
There exists $y\in A$ such that 
$\alpha^{-1}(y),y,\alpha(y)$ are mutually distinct. 
Then we have 
\[
(\sigma\alpha\sigma\alpha^{-1})(y)=\alpha^{-1}(y)
\]
and 
\[
(\alpha\sigma\alpha^{-1}\sigma)(y)=\alpha(y), 
\]
which contradict $[\sigma,\alpha\sigma\alpha^{-1}]=1$. 

(2)
Let $A\subset\supp(\tau)$ be a non-empty clopen set. 
By (F2), there exists $\sigma\in\Gamma$ such that 
$\supp(\sigma)\subset A\cup\tau(A)$ and $\sigma(x)=\tau(x)$ 
for every $x\in\supp(\sigma)$. 
It suffices to show that $\sigma$ is in $U_\tau$. 
It is easy to see $\sigma\in C_\tau$ and $\sigma^2=1$. 
Let $\alpha\in C_\tau$. 
Then, for any $x\in\supp(\alpha\sigma\alpha^{-1})=\alpha(\supp(\sigma))$, 
one has $(\alpha\sigma\alpha^{-1})(x)=(\alpha\tau\alpha^{-1})(x)=\tau(x)$. 
It follows that $\sigma$ commutes with $\alpha\sigma\alpha^{-1}$, 
and hence $\sigma$ belongs to $U_\tau$. 

(3)
Let $A$ be a non-empty clopen set such that $A\cap\supp(\tau)=\emptyset$. 
By (F3), there exists $\alpha\in\Gamma$ 
such that $\supp(\alpha)\subset A$ and $\alpha^2\neq1$. 
It follows from (1) that $\alpha$ commutes with any elements in $U_\tau$. 
Therefore $\alpha^2$ is in $S_\tau\setminus\{1\}$. 

(4)
Suppose that $\alpha\in\Gamma$ commutes with any elements in $U_\tau$. 
It suffices to show that if $x\neq\tau(x)$, 
then $\alpha(x)$ is equal to either $x$ or $\tau(x)$. 
Assume that $x,\tau(x),\alpha(x)$ are mutually distinct. 
There exists a clopen neighborhood $A\subset\supp(\tau)$ of $x$ 
such that $A,\tau(A),\alpha(A)$ are mutually disjoint. 
By (2), we can find $\sigma\in U_\tau$ such that 
$\supp(\sigma)\subset A\cup\tau(A)$ and $\sigma(y)=\tau(y)$ 
for every $y\in\supp(\sigma)$. 
Then $(\sigma\alpha)(y)=\alpha(y)$ is not equal to $(\alpha\sigma)(y)$ 
for $y\in\supp(\sigma)$, which is a contradiction. 

(5)
From (4), we easily obtain $\Gamma(\supp(\tau))\subset W_\tau$. 
Let us show $W_\tau\subset\Gamma(\supp(\tau))$. 
Let $\alpha\in W_\tau$ and $x\notin\supp(\tau)$. 
Assume $\alpha(x)\neq x$. 
There exists a non-empty clopen neighborhood $A$ of $x$ 
such that $A\cap\supp(\tau)=\emptyset$ and $A\cap\alpha(A)=\emptyset$. 
By (3), there exists $\beta\in S_\tau\setminus\{1\}$ 
such that $\supp(\beta)\subset A$. 
Then $\beta$ is not equal to $\alpha\beta\alpha^{-1}$, 
which is a contradiction. 
Therefore $\alpha(x)=x$, and whence $\alpha\in\Gamma(\supp(\tau))$. 
\end{proof}

\begin{lemma}
For $i=1,2$, 
let $X_i$ be a Cantor set and 
let $\Gamma_i\subset\Homeo(X_i)$ be a subgroup of class F. 
Suppose that $\Phi:\Gamma_1\to\Gamma_2$ is an isomorphism. 
Let $\tau,\sigma\in\Gamma_1$ be such that $\tau^2=\sigma^2=1$. 
\begin{enumerate}
\item We have $\supp(\tau)\subset\supp(\sigma)$ 
if and only if $\supp(\Phi(\tau))\subset\supp(\Phi(\sigma))$. 
\item We have $\supp(\tau)\cap\supp(\sigma)=\emptyset$ 
if and only if $\supp(\Phi(\tau))\cap\supp(\Phi(\sigma))=\emptyset$. 
\end{enumerate}
\end{lemma}
\begin{proof}
First, we note that $\Phi(W_\tau)$ equals $W_{\Phi(\tau)}$, 
because the definition of $W_\tau$ is purely algebraic. 

(1)
From Lemma \ref{localsubgrp} and Lemma \ref{W} (5), we get 
\[
\supp(\tau)\subset\supp(\sigma)
\Rightarrow\Gamma_1(\supp(\tau))\subset\Gamma_1(\supp(\sigma))
\Rightarrow W_\tau\subset W_{\sigma}
\Rightarrow\Phi(W_\tau)\subset\Phi(W_{\sigma}), 
\]
which means $W_{\Phi(\tau)}\subset W_{\Phi(\sigma)}$. 
Using Lemma \ref{localsubgrp} and Lemma \ref{W} (5) again, we have 
\[
W_{\Phi(\tau)}\subset W_{\Phi(\sigma)}
\Rightarrow\Gamma_2(\supp(\Phi(\tau)))\subset\Gamma_2(\supp(\Phi(\sigma)))
\Rightarrow\supp(\Phi(\tau))\subset\supp(\Phi(\sigma)). 
\]

(2)
Suppose that $\supp(\tau)\cap\supp(\sigma)$ is not empty. 
By (F0), $\supp(\tau)\cap\supp(\sigma)$ has non-empty interior. 
It follows from (F1) that there exists $\rho\in\Gamma\setminus\{1\}$ 
such that $\rho^2=1$ and $\supp(\rho)\subset\supp(\tau)\cap\supp(\sigma)$. 
From (1), 
we get $\supp(\Phi(\rho))\subset\supp(\Phi(\tau))\cap\supp(\Phi(\sigma))$. 
Thus $\supp(\Phi(\tau))\cap\supp(\Phi(\sigma))$ is not empty. 
\end{proof}

The following theorem is a generalization of 
\cite[Theorem 4.2]{GPS99Israel} and \cite[Theorem 7.2]{Matsu1205}. 
K. Medynets \cite{Me11BLMS} also obtained a similar result 
for (topological) full groups of transformation groupoids 
(see \cite[Remark 3]{Me11BLMS}). 

\begin{theorem}\label{spatial}
For $i=1,2$, 
let $X_i$ be a Cantor set and 
let $\Gamma_i\subset\Homeo(X_i)$ be a subgroup of class F. 
Suppose that $\Phi:\Gamma_1\to\Gamma_2$ is an isomorphism. 
Then there exists a homeomorphism $\phi:X_1\to X_2$ such that 
$\Phi(\alpha)=\phi\circ\alpha\circ\phi^{-1}$ for all $\alpha\in\Gamma_1$. 
\end{theorem}
\begin{proof}
For $x\in X_i$, we let 
\[
T(x)=\{\tau\in\Gamma_i\mid x\in\supp(\tau),\ \tau^2=1\}. 
\]
We first claim that for every $x\in X_1$, the set 
\[
P(x)=\bigcap_{\tau\in T(x)}\supp(\Phi(\tau))\subset X_2
\]
is a singleton. 
Let $\tau_1,\tau_2,\dots,\tau_n\in T(x)$. 
By (F0), $\supp(\tau_i)$ is clopen, and so 
there exists a clopen neighborhood $A$ of $x$ 
such that $A\subset\supp(\tau_i)$ for all $i=1,2,\dots,n$. 
By (F1), we can find $\tau\in T(x)\cap\Gamma_1(A)$. 
It follows from the lemma above that 
$\supp(\Phi(\tau))$ is contained in $\supp(\Phi(\tau_i))$ 
for every $i=1,2,\dots,n$. 
Hence the sets $\supp(\Phi(\tau_i))$ have non-empty intersection. 
From the compactness of $X_2$, we can conclude that $P(x)$ is not empty. 
Assume that $P(x)$ contains two distinct points $y,y'\in X_2$. 
By (F1), there exists $\sigma\in T(y)$ such that $y'\notin\supp(\sigma)$. 
If $x$ does not belong to $\supp(\Phi^{-1}(\sigma))$, 
then by (F1) there exists $\tau\in T(x)$ 
such that $\supp(\tau)\cap\supp(\Phi^{-1}(\sigma))=\emptyset$. 
By the lemma above, one has $\supp(\Phi(\tau))\cap\supp(\sigma)=\emptyset$. 
Since $y$ is in $P(x)$ and $P(x)$ is contained in $\supp(\Phi(\tau))$, 
$\supp(\Phi(\tau))$ intersects with $\supp(\sigma)$, 
which is a contradiction. 
Hence $x$ is in $\supp(\Phi^{-1}(\sigma))$, that is, 
$\Phi^{-1}(\sigma)$ is in $T(x)$. 
Then $P(x)$ is contained in $\supp(\sigma)$, 
and so $y'$ is not in $P(x)$, 
which is again a contradiction. 

Now we can define a map $\phi:X_1\to X_2$ by $\{\phi(x)\}=P(x)$. 
It is easy to see that 
$\phi(\supp(\tau))\subset\supp(\Phi(\tau))$ holds 
for any $\tau\in\Gamma_1$ such that $\tau^2=1$. 
We would like to show that $\phi$ is continuous. 
Let $A$ be an open neighborhood of $\phi(x)$. 
By the definition of $P(x)$, 
there exist $\tau_1,\tau_2,\dots,\tau_n\in T(x)$ such that 
the intersection of the sets $\supp(\Phi(\tau_i))$ is contained in $A$. 
In the same way as above, using (F0) and (F1), 
we can find $\tau\in T(x)$ satisfying $\supp(\tau)\subset\supp(\tau_i)$. 
Since $\supp(\tau)$ is a clopen neighborhood of $x$ by (F0) and 
\[
\phi(\supp(\tau))\subset\supp(\Phi(\tau))
\subset\bigcap_{i=1}^n\supp(\Phi(\tau_i))\subset A, 
\]
$\phi$ is continuous at $x$. 

In exactly the same way, 
we can construct a continuous map $\phi':X_2\to X_1$ such that 
\[
\{\phi'(y)\}=\bigcap_{\sigma\in T(y)}\supp(\Phi^{-1}(\sigma))
\quad\forall y\in X_2. 
\]
For any $x\in X_1$, by (F1), we have 
\[
\{\phi'(\phi(x))\}
=\bigcap_{\sigma\in T(\phi(x))}\supp(\Phi^{-1}(\sigma))
\subset\bigcap_{\tau\in T(x)}\supp(\Phi^{-1}(\Phi(\tau)))=\{x\}, 
\]
which means $\phi'\circ\phi=\id$. 
Similarly $\phi\circ\phi'=\id$. 
Thus $\phi$ is a homeomorphism and $\phi^{-1}=\phi'$. 

Let $\alpha\in\Gamma_1$. 
We would like to show $\Phi(\alpha)=\phi\circ\alpha\circ\phi^{-1}$. 
For any $\tau\in\Gamma_1$ satisfying $\tau^2=1$, one has 
\[
(\Phi(\alpha)\circ\phi)(\supp(\tau))=\Phi(\alpha)(\supp(\Phi(\tau))
=\supp(\Phi(\alpha\tau\alpha^{-1}))
\]
and 
\[
(\phi\circ\alpha)(\supp(\tau))=\phi(\supp(\alpha\tau\alpha^{-1}))
=\supp(\Phi(\alpha\tau\alpha^{-1})). 
\]
Since $\{x\}=\bigcap_{\tau\in T(x)}\supp(\tau)$ for any $x\in X_1$, 
we can conclude $\Phi(\alpha)\circ\phi=\phi\circ\alpha$. 
\end{proof}

\begin{proposition}\label{tfg=F}
Let $G$ be an essentially principal \'etale groupoid 
whose unit space is a Cantor set. 
Suppose that $G$ is minimal. 
Then any subgroup $\Gamma\subset[[G]]$ containing $D([[G]])$ is of class F. 
\end{proposition}
\begin{proof}
Let $\Gamma\subset[[G]]$ be a subgroup such that $D([[G]])\subset\Gamma$. 
We verify the conditions of Definition \ref{F} for $\Gamma$. 
Condition (F0) immediately follows from Lemma \ref{suppclp}. 

To show (F1), 
let $A\subset G^{(0)}$ be a clopen neighborhood of $x\in G^{(0)}$. 
Since $G$ is minimal, 
there exists a compact open $G$-set $U\subset G$ such that 
$x\in s(U)$, $s(U)\cup r(U)\subset A$ and $s(U)\cap r(U)=\emptyset$. 
Define $\alpha\in[[G]]$ by 
\[
\alpha(x)=\begin{cases}\pi_U(x)&x\in s(U)\\\pi_U^{-1}(x)&x\in r(U)\\
x&\text{otherwise. }\end{cases}
\]
We can also find a compact open $G$-set $V\subset G$ such that 
$x\in s(V)$, $s(V)\cup r(V)\subset s(U)$ and $s(V)\cap r(V)=\emptyset$, 
because $G$ is minimal. 
Define $\beta\in[[G]]$ by 
\[
\beta(x)=\begin{cases}\pi_V(x)&x\in s(V)\\\pi_V^{-1}(x)&x\in r(V)\\
x&\text{otherwise. }\end{cases}
\]
Then $[\alpha,\beta]=\alpha\beta\alpha\beta\in D[[G]]$ 
satisfies the requirement. 

In order to prove (F2), 
we take $\alpha\in[[G]]\setminus\{1\}$ satisfying $\alpha^2=1$ and 
a non-empty clopen set $A\subset\supp(\alpha)$. 
Since $G$ is minimal, 
we can find a non-empty compact open $G$-set $V$ such that 
$s(V)\cup r(V)\subset A$ and 
$s(V),r(V),\alpha(s(V)),\alpha(r(V))$ are mutually disjoint. 
Using $V$, we define $\beta\in[[G]]$ as above. 
Define $\tilde\alpha,\tilde\beta\in[[G]]$ by 
\[
\tilde\alpha(x)=\begin{cases}\alpha(x)&x\in s(V)\cup\alpha(s(V))\\
x&\text{otherwise}\end{cases}
\]
and $\tilde\beta=[\alpha,\beta]=\alpha\beta\alpha\beta$. 
Then $[\tilde\alpha,\tilde\beta]\in D([[G]])$ satisfies the requirement. 

Let us consider (F3). 
Since $G$ is minimal, there exist non-empty compact open $G$-sets $U_1,U_2$ 
satisfying $s(U_i)\cup r(U_i)\subset A$, $s(U_i)\cap r(U_i)=\emptyset$, 
$r(U_1)\cap s(U_2)=\emptyset$ and $s(U_1)=r(U_2)$. 
For $i=1,2$, we define $\alpha_i\in[[G]]$ by 
\[
\alpha_i(x)=\begin{cases}\pi_{U_i}(x)&x\in s(U_i)\\
\pi_{U_i}^{-1}(x)&x\in r(U_i)\\
x&\text{otherwise. }\end{cases}
\]
Then $[\alpha_1,\alpha_2]\in D([[G]])$ is of order three. 
\end{proof}

\begin{lemma}
Let $G$ be an essentially principal \'etale groupoid 
whose unit space is a Cantor set. 
Suppose that every $G$-orbit contains at least three points. 
For any $g\in G$, there exists a compact open $G$-set $U$ such that 
$\pi_U\in D([[G]])$ and $g\in U$. 
\end{lemma}
\begin{proof}
First, let us assume $s(g)\neq r(g)$. 
Since the $G$-orbit of $s(g)$ contains at least three points, 
there exist compact open $G$-sets $U$ and $V$ such that the following hold. 
\begin{itemize}
\item $s(U)=r(V)$. 
\item $r(U)$, $s(U)$ and $s(V)$ are mutually disjoint. 
\item $g\in U$. 
\end{itemize}
Then 
\[
W=U\cup V\cup(UV)^{-1}\cup(G^{(0)}\setminus(r(U)\cup s(U)\cup s(V)))
\]
is a desired compact open $G$-set. 
Next, assume $s(g)=r(g)$. 
We can find $g_1,g_2\in G$ such that 
$s(g_i)\neq r(g_i)$ for $i=1,2$ and $g=g_1g_2$. 
By the argument above, the proof is complete. 
\end{proof}

\begin{proposition}\label{spatial>iso}
For $i=1,2$, let $G_i$ be an essentially principal \'etale groupoid 
whose unit space is a Cantor set. 
Suppose that every $G_i$-orbit contains at least three points. 
For each $i=1,2$, let $\Gamma_i$ be a subgroup of $[[G_i]]$ 
such that $D([[G_i]])\subset\Gamma_i$. 
If there exists a homeomorphism $\phi:G_1^{(0)}\to G_2^{(0)}$ 
such that $\alpha\mapsto\phi\circ\alpha\circ\phi^{-1}$ gives 
an isomorphism from $\Gamma_1$ to $\Gamma_2$, then 
$G_1$ and $G_2$ are isomorphic as \'etale groupoids. 
\end{proposition}
\begin{proof}
For $g\in G_1$, by the lemma above, 
we can find a compact open $G_1$-set $U$ 
such that $\pi_U\in D([[G_1]])\subset\Gamma_1$ and $g\in U$. 
By assumption, $\phi\circ\pi_U\circ\phi^{-1}$ belongs to $\Gamma_2$. 
Thus, there exists a compact open $G_2$-set $V$ 
such that $\pi_V=\phi\circ\pi_U\circ\phi^{-1}$. 
It follows that there exists $g'\in G_2$ 
such that $g'\in V$ and $s(g')=\phi(s(g))$. 
It is routine to check that 
the map $g\mapsto g'$ is a well-defined isomorphism from $G_1$ to $G_2$. 
\end{proof}

The following theorem is a generalization of 
several results obtained in \cite[Section 4]{GPS99Israel} and 
\cite[Corollary 7.3]{Matsu1205}. 

\begin{theorem}\label{iso}
For $i=1,2$, let $G_i$ be an essentially principal \'etale groupoid 
whose unit space is a Cantor set. 
Suppose that $G_i$ is minimal. 
For each $i=1,2$, let $\Gamma_i$ be a subgroup of $[[G_i]]$ 
such that $D([[G_i]])\subset\Gamma_i$. 
If $\Gamma_1$ and $\Gamma_2$ are isomorphic as discrete groups, then 
$G_1$ and $G_2$ are isomorphic as \'etale groupoids. 
\end{theorem}
\begin{proof}
This readily follows from Theorem \ref{spatial}, Proposition \ref{tfg=F} 
and Proposition \ref{spatial>iso}. 
\end{proof}

As an immediate consequence of the theorem above, we obtain the following. 

\begin{theorem}\label{iso2}
For $i=1,2$, let $G_i$ be an essentially principal \'etale groupoid 
whose unit space is a Cantor set 
and suppose that $G_i$ is minimal. 
The following conditions are equivalent. 
\begin{enumerate}
\item $G_1$ and $G_2$ are isomorphic as \'etale groupoids. 
\item $[[G_1]]$ and $[[G_2]]$ are isomorphic as discrete groups. 
\item $[[G_1]]_0$ and $[[G_2]]_0$ are isomorphic as discrete groups. 
\item $D([[G_1]])$ and $D([[G_2]])$ are isomorphic as discrete groups. 
\end{enumerate}
\end{theorem}

%%%%%%%%%%%%%%%%%%%%%%%%%%%%%%%%%%%%%%%%%%%%%%%%%%%%%%%%%%%%
\section{Simplicity of commutator subgroups}

%%%%%%%%%%%%%%%%%%%%%%%%%%%%%%%%%%%%%%%%%%%%%%%%%%%%%%%%%%%%
\subsection{Almost finite groupoids}

Throughout this subsection, 
we let $G$ be an essentially principal \'etale groupoid 
whose unit space is a Cantor set. 
In this subsection, 
we would like to show that the commutator subgroup $D([[G]])$ of 
the topological full group $[[G]]$ is simple 
when $G$ is almost finite and minimal (Theorem \ref{simple1}). 
This is a generalization of 
\cite[Theorem 4.9]{M06IJM} and \cite[Theorem 3.4]{BM08Coll}. 
Let us recall the definition of almost finite groupoids. 

\begin{definition}[{\cite[Definition 6.2]{M12PLMS}}]
Let $G$ be an \'etale groupoid whose unit space is a Cantor set. 
\begin{enumerate}
\item We say that $K\subset G$ is an elementary subgroupoid 
if $K$ is a compact open principal subgroupoid of $G$ 
such that $K^{(0)}=G^{(0)}$. 
\item We say that $G$ is almost finite 
if for any compact subset $C\subset G$ and $\ep>0$ 
there exists an elementary subgroupoid $K\subset G$ such that 
\[
\frac{\#(CKx\setminus Kx)}{\#(Kx)}<\ep
\]
for all $x\in G^{(0)}$. 
We also remark that $\#(K(x))$ equals $\#(Kx)$, 
because $K$ is principal. 
\end{enumerate}
\end{definition}

In \cite[Lemma 6.3]{M12PLMS}, 
it was shown that 
when $\phi:\Z^N\curvearrowright X$ is a free action of $\Z^N$ 
on a Cantor set $X$, 
the transformation groupoid $G_\phi$ is almost finite 
(see \cite[Definition 2.1]{M12PLMS} for the definition of $G_\phi$). 

A homeomorphism $\alpha$ of a Cantor set $X$ is said to be elementary 
if it is of finite order and 
$\{x\in X\mid\alpha^k(x)=x\}$ is clopen for every $k\in\N$ 
(\cite[Definition 7.6 (1)]{M12PLMS}). 

\begin{lemma}
Suppose that $G$ is almost finite. 
For any $\ep>0$ and $\alpha\in[[G]]$, 
there exists an elementary homeomorphism $\alpha_0\in[[G]]$ such that 
$\mu(\supp(\alpha_0\alpha))\leq\ep$ for any $\mu\in M(G)$. 
\end{lemma}
\begin{proof}
This statement is almost the same as \cite[Lemma 7.10]{M12PLMS} 
and is shown by a small modification of its proof. 
\end{proof}

\begin{lemma}
Suppose that every $G$-orbit is infinite. 
For any elementary homeomorphism $\alpha\in[[G]]$ and $\ep>0$, 
there exist elementary homeomorphisms 
$\alpha_1,\alpha_2,\dots,\alpha_n\in[[G]]$ such that 
$\alpha=\alpha_1\alpha_2\dots\alpha_n$ and 
$\mu(\supp(\alpha_i))\leq\ep$ for any $i=1,2,\dots,n$ and $\mu\in M(G)$. 
\end{lemma}
\begin{proof}
For $x\in G^{(0)}$, let $r(x)=\min\{k\in\N\mid\alpha^k(x)=x\}$. 
Since every $G$-orbit is infinite, 
there exists a clopen neighborhood $U_x$ of $x$ 
such that the following hold. 
\begin{itemize}
\item $r(y)=r(x)$ for every $y\in U_x$, 
\item $\mu(U_x)\leq\ep/r(x)$ for any $\mu\in M(G)$. 
\item $U_x,\alpha(U_x),\dots,\alpha^{r(x)-1}(U_x)$ are mutually disjoint. 
\end{itemize}
Set $V_x=U_x\cup\alpha(U_x)\cup\dots\cup\alpha^{r(x)-1}(U_x)$. 
Then $V_x$ is $\alpha$-invariant and $\mu(V_x)=r(x)\mu(U_x)\leq\ep$. 
We can find finitely many such clopen sets $V_x$, 
say $V_1,V_2,\dots,V_n$, so that they cover $G^{(0)}$. 
Define $W_i$ inductively 
by $W_1=V_1$ and $W_i=V_i\setminus(W_1\cup W_2\cup\dots\cup W_{i-1})$. 
Put 
\[
\alpha_i(x)
=\begin{cases}\alpha(x)&x\in W_i\\x&\text{otherwise. }
\end{cases}
\]
It is easy to verify that 
$\alpha_1,\alpha_2,\dots,\alpha_n$ satisfy the requirement. 
\end{proof}

\begin{lemma}\label{decomp}
Suppose that $G$ is almost finite and every $G$-orbit is infinite. 
For any $\alpha\in[[G]]$ and $\ep>0$, 
there exist $\alpha_1,\alpha_2,\dots,\alpha_n\in[[G]]$ such that 
$\alpha=\alpha_1\alpha_2\dots\alpha_n$ and 
$\mu(\supp(\alpha_i))\leq\ep$ for any $i=1,2,\dots,n$ and $\mu\in M(G)$. 
\end{lemma}
\begin{proof}
This follows immediately from the lemmas above. 
\end{proof}

\begin{lemma}\label{tauexist}
Suppose that $G$ is almost finite and minimal. 
Let $N$ be a non-trivial subgroup of $[[G]]$ normalized by $D([[G]])$. 
Then there exists $\tau\in N\setminus\{1\}$ 
such that $\supp(\tau)\neq G^{(0)}$. 
\end{lemma}
\begin{proof}
Take $\sigma\in N\setminus\{1\}$. 
There exists a non-empty clopen set $A\subset G^{(0)}$ 
such that $A\cap\sigma(A)=\emptyset$ and $A\cup\sigma(A)\neq G^{(0)}$. 
Since $G$ is minimal, there exists $\alpha\in D([[G]])\setminus\{1\}$ 
such that $\supp(\alpha)\subset A$. 
Then $\tau=\alpha\sigma\alpha^{-1}\sigma^{-1}$ is in $N\setminus\{1\}$ 
and $\supp(\tau)\subset A\cup\sigma(A)\neq G^{(0)}$. 
\end{proof}

\begin{lemma}\label{normalize}
Suppose that $G$ is almost finite and minimal. 
Let $N$ be a non-trivial subgroup of $[[G]]$ normalized by $D([[G]])$. 
If $\tau\in N$ satisfies $\supp(\tau)\neq G^{(0)}$, 
then $\alpha\tau\alpha^{-1}\in N$ for all $\alpha\in[[G]]$. 
\end{lemma}
\begin{proof}
By \cite[Lemma 6.8]{M12PLMS}, 
there exists $c>0$ such that $\mu(\supp(\tau))<1-c$ for all $\mu\in M(G)$. 
Let $\alpha\in[[G]]$. 
By virtue of Lemma \ref{decomp}, 
it suffices to show $\alpha\tau\alpha^{-1}\in N$ 
when $\mu(\supp(\alpha))\leq c$ for any $\mu\in M(G)$. 
It follows from \cite[Lemma 6.7]{M12PLMS} that 
there exists $\sigma\in[[G]]$ such that 
$\sigma(\supp(\alpha))\cap\supp(\tau)=\emptyset$. 
Therefore $\sigma\alpha^{-1}\sigma^{-1}$ commutes with $\tau$. 
Since $[\alpha,\sigma]=\alpha\sigma\alpha^{-1}\sigma^{-1}$ is in $D([[G]])$ 
and $N$ is normalized by $D([[G]])$, we get 
\[
\alpha\tau\alpha^{-1}
=\alpha(\sigma\alpha^{-1}\sigma^{-1})\tau(\sigma\alpha\sigma^{-1})\alpha^{-1}
\in N, 
\]
which completes the proof. 
\end{proof}

The proof of the following theorem is inspired 
by that of \cite[Theorem 3.4]{BM08Coll}. 

\begin{theorem}\label{simple1}
Suppose that $G$ is almost finite and minimal. 
Then any non-trivial subgroup of $[[G]]$ 
normalized by the commutator subgroup $D([[G]])$ contains $D([[G]])$. 
In particular, $D([[G]])$ is simple. 
\end{theorem}
\begin{proof}
Let $N\subset[[G]]$ be a non-trivial subgroup normalized by $D([[G]])$. 
By Lemma \ref{tauexist}, 
there exists $\tau\in N\setminus\{1\}$ 
such that $\supp(\tau)\neq G^{(0)}$. 
Let $A\subset G^{(0)}$ be a non-empty clopen set 
such that $A\cap\tau(A)=\emptyset$. 
By \cite[Lemma 6.8]{M12PLMS}, 
there exists $c>0$ such that $\mu(A)>c$ for any $\mu\in M(G)$. 
We call $\alpha\in[[G]]$ small 
when $\mu(\supp(\alpha))<c/2$ for any $\mu\in M(G)$. 
Notice that if $\alpha\in[[G]]$ is small, then 
\[
\mu(\supp([\alpha,\sigma]))=\mu(\supp(\alpha\sigma\alpha^{-1}\sigma^{-1}))
\leq\mu(\supp(\alpha)\cup\sigma(\supp(\alpha)))
<c<1
\]
holds for any $\sigma\in[[G]]$ and $\mu\in M(G)$, 
and hence $\supp([\alpha,\sigma])\neq G^{(0)}$. 

First, we claim that 
the commutator $[\alpha,\beta]$ belongs to $N$ 
for any small $\alpha,\beta\in[[G]]$. 
By \cite[Lemma 6.7]{M12PLMS}, we can find $\sigma\in[[G]]$ 
such that $\sigma(\supp(\alpha)\cup\supp(\beta))\subset A$. 
From Lemma \ref{normalize}, 
we get $\tilde\tau=\sigma^{-1}\tau\sigma\in N$. 
Since $\tilde\tau(\supp(\alpha))$ and $\supp(\beta)$ are disjoint, 
$\tilde\tau\alpha\tilde\tau^{-1}$ commutes with $\beta$. 
Hence 
\[
[\alpha,\beta]=\alpha\beta\alpha^{-1}\beta^{-1}
=\alpha(\tilde\tau\alpha^{-1}\tilde\tau^{-1})\beta
(\tilde\tau\alpha\tilde\tau^{-1})\alpha^{-1}\beta^{-1}
=[[\alpha,\tilde\tau],\beta]. 
\]
As $\supp(\tilde\tau)=\sigma^{-1}(\supp(\tau))\neq G^{(0)}$, 
by using Lemma \ref{normalize} again, 
we also have $\alpha\tilde\tau\alpha^{-1}\in N$, 
and hence $[\alpha,\tilde\tau]\in N$. 
As mentioned above, $\supp([\alpha,\tilde\tau])\neq G^{(0)}$ 
because $\alpha$ is small. 
Therefore, by Lemma \ref{normalize}, 
we can conclude $[\alpha,\beta]\in N$. 

Next, we claim that 
the commutator $[\alpha,\beta]$ belongs to $N$ 
if $\beta$ is small. 
By Lemma \ref{decomp}, 
there exist $\alpha_1,\alpha_2,\dots,\alpha_n\in[[G]]$ such that 
$\alpha=\alpha_1\alpha_2\dots\alpha_n$ and 
$\alpha_i$ is small for all $i=1,2,\dots,n$. 
Suppose that $[\alpha_1\alpha_2\dots\alpha_{k-1},\beta]\in N$ is known. 
By the claim above, $[\alpha_k,\beta]\in N$. 
Besides, we can apply Lemma \ref{normalize} for $[\alpha_k,\beta]$, 
because $\supp([\alpha_k,\beta])\neq G^{(0)}$. 
Then 
\[
[\alpha_1\alpha_2\dots\alpha_k,\beta]
=\alpha_1\alpha_2\dots\alpha_{k-1}[\alpha_k,\beta]
(\alpha_1\alpha_2\dots\alpha_{k-1})^{-1}
[\alpha_1\alpha_2\dots\alpha_{k-1},\beta]
\]
is also in $N$. 
By induction, we obtain $[\alpha,\beta]\in N$. 

Finally, we would like to show that 
$[\alpha,\beta]$ belongs to $N$ for any $\alpha,\beta\in[[G]]$. 
By Lemma \ref{decomp}, 
there exist $\beta_1,\beta_2,\dots,\beta_m\in[[G]]$ such that 
$\beta=\beta_1\beta_2\dots\beta_m$ and 
$\beta_i$ is small for all $i=1,2,\dots,m$. 
Suppose that $[\alpha,\beta_1\beta_2\dots\beta_{k-1}]\in N$ is known. 
By the claim above, $[\alpha,\beta_k]\in N$. 
Besides, we can apply Lemma \ref{normalize} for $[\alpha,\beta_k]$, 
because $\supp([\alpha,\beta_k])\neq G^{(0)}$. 
Then 
\[
[\alpha,\beta_1\beta_2\dots\beta_k]
=[\alpha,\beta_1\beta_2\dots\beta_{k-1}]
\beta_1\beta_2\dots\beta_{k-1}[\alpha,\beta_k]
(\beta_1\beta_2\dots\beta_{k-1})^{-1}
\]
is also in $N$. 
By induction, we obtain $[\alpha,\beta]\in N$. 
Hence the commutator subgroup $D([[G]])$ is contained in $N$. 
\end{proof}

\begin{remark}
In general, we do not know what the abelianization $[[G]]/D([[G]])$ is. 
In \cite[Theorem 7.5]{M12PLMS}, it was shown that 
the index map $I:[[G]]\to H_1(G)$ is surjective when $G$ is almost finite. 
Thus $[[G]]/D([[G]])$ at least has $H_1(G)\cong[[G]]/[[G]]_0$ as its quotient. 
For a minimal action $\phi:\Z\curvearrowright X$ of $\Z$ on a Cantor set $X$, 
it is known that $[[G_\phi]]_0/D([[G_\phi]]])$ is isomorphic to 
$H_0(G_\phi)\otimes\Z_2$ 
(see \cite[Theorem 4.8]{M06IJM} and the remark following it). 
\end{remark}

%%%%%%%%%%%%%%%%%%%%%%%%%%%%%%%%%%%%%%%%%%%%%%%%%%%%%%%%%%%%
\subsection{Purely infinite groupoids}

Throughout this subsection, 
we let $G$ be an essentially principal \'etale groupoid 
whose unit space is a Cantor set. 
In this subsection, 
we would like to show that the commutator subgroup $D([[G]])$ of 
the topological full group $[[G]]$ is simple 
when $G$ is purely infinite and minimal (Theorem \ref{simple2}). 
We first introduce the notion of purely infinite groupoids. 
This definition is inspired 
by the work of M. R\o rdam and A. Sierakowski \cite{RS12ETDS}. 

\begin{definition}\label{pi}
\begin{enumerate}
\item A clopen set $A\subset G^{(0)}$ is said to be properly infinite 
if there exist compact open $G$-sets $U,V\subset G$ 
such that $s(U)=s(V)=A$, $r(U)\cup r(V)\subset A$ 
and $r(U)\cap r(V)=\emptyset$. 
\item We say that $G$ is purely infinite 
if every clopen set $A\subset G^{(0)}$ is properly infinite. 
\end{enumerate}
\end{definition}

It is easy to see that if $G$ is purely infinite, then 
for any clopen set $Y\subset G^{(0)}$ 
the reduction $G|Y$ is also purely infinite. 
It is also clear that if $G^{(0)}$ is properly infinite, then 
there does not exist a $G$-invariant probability measure. 
When $G$ is purely infinite, 
in essentially the same way as \cite[Theorem 4.1]{RS12ETDS}, 
one can prove that 
the reduced groupoid $C^*$-algebra $C^*_r(G)$ is purely infinite. 

The following is a generalization of \cite[Theorem 2.5]{MatsuDCDS}. 

\begin{proposition}\label{pi>Z2*Z3}
Suppose that $G^{(0)}$ is properly infinite. 
The group $[[G]]$ contains a subgroup isomorphic to 
the free product $\Z_2*\Z_3$. 
In particular, $[[G]]$ is not amenable. 
\end{proposition}
\begin{proof}
Since $G^{(0)}$ is properly infinite, 
there exist compact open $G$-sets $U,V\subset G$ 
such that $s(U)=s(V)=G^{(0)}$ and $r(U)\cap r(V)=\emptyset$. 
Let $A=r(U)$ and $B=r(V)$. 
Define $\alpha,\beta\in[[G]]$ by 
\[
\alpha(x)=\begin{cases}\pi_V(\pi_U^{-1}(x))&x\in A\\
\pi_U(\pi_V^{-1}(x))&x\in B\\ x&\text{otherwise, }\end{cases}
\beta(x)=\begin{cases}
\pi_V(\pi_U^{-1}(\pi_U^{-1}(x)))&x\in r(UU)\\
\pi_U(x)&x\in B\\
\pi_U(\pi_U(\pi_V^{-1}((\pi_U^{-1}(x)))))&x\in r(UV)\\
x&\text{otherwise. }\end{cases}
\]
Evidently $\alpha^2=1$, $\beta^3=1$, $\alpha(A)=B$, $\beta(B)\subset A$ and 
$\beta^2(B)\subset A$. 
The table-tennis lemma (see \cite[II.B.24]{Htext}, for instance) implies that 
the subgroup generated by $\alpha$ and $\beta$ is isomorphic to 
the free product $\Z_2*\Z_3$. 
\end{proof}

\begin{proposition}\label{pim}
The following conditions are mutually equivalent. 
\begin{enumerate}
\item $G$ is purely infinite and minimal. 
\item For any clopen sets $A,B\subset G^{(0)}$ with $B\neq\emptyset$, 
there exists a compact open $G$-set $U\subset G$ 
such that $s(U)=A$ and $r(U)\subset B$. 
\item For any clopen sets $A,B\subset G^{(0)}$ 
with $A\neq G^{(0)}$ and $B\neq\emptyset$, 
there exists $\alpha\in[[G]]$ such that $\alpha(A)\subset B$. 
\end{enumerate}
\end{proposition}
\begin{proof}
(1)$\Rightarrow$(2)
There exist compact open $G$-sets $U,V\subset G$ 
such that $s(U)=s(V)=B$, $r(U)\cup r(V)\subset B$ 
and $r(U)\cap r(V)=\emptyset$, 
because $B$ is properly infinite. 
Define $V_1,V_2,\dots$ inductively by $V_1=U$, $V_{n+1}=VV_n$. 
Then $V_n$ are compact open $G$-sets 
which satisfy $s(V_n)=B$, $r(V_n)\subset B$ 
and $r(V_n)\cap r(V_m)=\emptyset$ for $n\neq m$. 
Since $G$ is minimal, 
we can find compact open $G$-sets $W_1,W_2,\dots,W_n$ such that 
$r(W_i)\subset B$, $A=s(W_1)\cup s(W_2)\cup\dots\cup s(W_n)$ 
and $s(W_i)\cap s(W_j)=\emptyset$ for $i\neq j$. 
Define a compact open $G$-set $W\subset G$ by 
\[
W=\bigcup_{i=1}^nV_iW_i. 
\]
Clearly we have $s(W)=A$ and $r(W)\subset B$. 

(2)$\Rightarrow$(3)
First, we assume that $B\setminus A$ is not empty. 
By (2), there exists a compact open $G$-set $U\subset G$ 
such that $s(U)=A$ and $r(U)\subset B\setminus A$. 
Then 
\[
V=U\cup U^{-1}\cup(G^{(0)}\setminus(s(U)\cup r(U)))
\]
is a compact open $G$-set satisfying $s(V)=r(V)=G^{(0)}$. 
Hence $\pi_V$ belongs to $[[G]]$. 
It is easy to see $\pi_V(A)\subset B$. 
Next, we assume that $B\setminus A$ is empty, thus $B\subset A$. 
By the argument above, we can find $\alpha_1,\alpha_2\in[[G]]$ 
such that $\alpha_1(A)\subset G^{(0)}\setminus A$ and 
$\alpha_2(G^{(0)}\setminus A)\subset B$. 
Then $\alpha=\alpha_2\alpha_1$ meets the requirement. 

(3)$\Rightarrow$(1)
Clearly (3) implies minimality of $G$. 
Let $A\subset G^{(0)}$ be a non-empty clopen subset. 
Choose mutually disjoint non-empty clopen sets $B_0,B_1,B_2\subset A$. 
By (3), there exist $\alpha_1,\alpha_2\in[[G]]$ 
satisfying $\alpha_1(A\setminus B_0)\subset B_1$ and 
$\alpha_2(B_0\cup B_1)\subset B_2$. 
Let $U_1,U_2$ be compact open $G$-sets 
such that $\alpha_1=\pi_{U_1}$ and $\alpha_2=\pi_{U_2}$. 
Define compact open $G$-sets $V_1$ and $V_2$ by 
\[
V_1=B_0\cup U_1(A\setminus B_0)\quad\text{and}\quad 
V_2=U_2V_1. 
\]
It is not so hard to see 
$s(V_1)=s(V_2)=A$, 
$r(V_1)\subset B_0\cup B_1$ and $r(V_2)\subset B_2$. 
The proof is complete. 
\end{proof}

\begin{remark}\label{n-filling}
Let $X$ be a Cantor set and 
let $\phi:\Gamma\curvearrowright X$ be an action of 
a countable discrete group $\Gamma$ by homeomorphisms. 
The action $\phi$ is said to be $n$-filling 
in the sense of P. Jolissaint and G. Robertson \cite{JR00JFA} 
if for any non-empty clopen sets $A_1,A_2,\dots,A_n$ of $X$, 
there exist $g_1,g_2,\dots,g_n\in\Gamma$ such that 
$\phi_{g_1}(A_1)\cup\dots\cup\phi_{g_n}(A_n)=X$. 
We can see that if $\phi$ is $n$-filling, then 
the transformation groupoid $G_\phi=\Gamma\times X$ is 
purely infinite and minimal. 
Indeed, one can verify condition (2) of the proposition above as follows. 
We identify the unit space $G^{(0)}$ with $X$. 
Let $A,B\subset X$ be non-empty clopen sets. 
We would like to construct a compact open $G_\phi$-set $U\subset G_\phi$ 
such that $s(U)=A$ and $r(U)\subset B$. 
Let $B_1,B_2,\dots,B_n\subset B$ be mutually disjoint non-empty clopen sets. 
Since $\phi$ is $n$-filling, 
there exist $g_1,g_2,\dots,g_n\in\Gamma$ such that 
$\phi_{g_1}(B_1)\cup\dots\cup\phi_{g_n}(B_n)=X$. 
Define $C_1,C_2,\dots,C_n$ inductively by 
\[
C_1=B_1\cap\phi_{g_1}^{-1}(A)\quad\text{and}\quad 
C_{k+1}=B_{k+1}\cap
\phi_{g_{k+1}}^{-1}(A\setminus(\phi_{g_1}(C_1)\cup\dots\cup\phi_{g_k}(C_k))). 
\]
Then $C_k$ are mutually disjoint and 
$A$ equals the disjoint union $\phi_{g_1}(C_1)\cup\dots\cup\phi_{g_n}(C_n)$. 
It follows that $U=\bigcup\{g_k^{-1}\}\times\phi_{g_k}(C_k)$ is 
a compact open $G$-set with $s(U)=A$ and $r(U)\subset B$. 
\end{remark}

\begin{lemma}\label{smallsupp}
Suppose that $G$ is purely infinite and minimal. 
\begin{enumerate}
\item For any clopen set $A,B\subset G^{(0)}$ 
with $A\neq G^{(0)}$ and $B\neq\emptyset$, there exists $\alpha\in[[G]]$ 
such that $\alpha(A)\subset B$ and $A\cup\supp(\alpha)\neq G^{(0)}$. 
\item For any clopen set $A,B\subset G^{(0)}$ 
with $A\neq G^{(0)}$ and $B\neq\emptyset$, 
there exists $\alpha\in D([[G]])$ such that $\alpha(A)\subset B$. 
\end{enumerate}
\end{lemma}
\begin{proof}
(1)
Take a clopen set $C$ so that 
$C\neq G^{(0)}$, $A\subset C$, $C\setminus A\neq\emptyset$ 
and $B\cap C\neq\emptyset$. 
The reduction $G|C=r^{-1}(C)\cap s^{-1}(C)$ is purely infinite and minimal. 
By Proposition \ref{pim} (3), 
there exists $\alpha\in[[G|C]]$ such that $\alpha(A)\subset B\cap C$. 
Letting $\alpha(x)=x$ for $x\notin C$, 
we may regard $\alpha$ as an element of $[[G]]$ 
and obtain $A\cup\supp(\alpha)\subset C\neq G^{(0)}$. 

(2) 
By (1), we can find $\alpha\in[[G]]$ such that 
$\alpha(A)\subset B$ and $A\cup\supp(\alpha)\neq G^{(0)}$. 
By Proposition \ref{pim} (3), 
there exists $\sigma\in[[G]]$ such that 
\[
\sigma(A\cup\supp(\alpha))\subset G^{(0)}\setminus(A\cup\supp(\alpha)). 
\]
Then we have $(\alpha\sigma\alpha^{-1}\sigma^{-1})(A)=\alpha(A)\subset B$. 
\end{proof}

\begin{lemma}\label{decomp2}
Suppose that $G$ is purely infinite and minimal. 
Let $N\subset[[G]]$ be a non-trivial subgroup normalized by $D([[G]])$. 
For any $\tau\in N$, there exist $\tau_1,\tau_2\in N$ 
such that $\supp(\tau_1)\neq G^{(0)}$, $\supp(\tau_2)\neq G^{(0)}$ 
and $\tau=\tau_1\tau_2$. 
\end{lemma}
\begin{proof}
It suffices to show the lemma when $\supp(\tau)=G^{(0)}$. 
We can find a non-empty clopen set $A\subset G^{(0)}$ 
such that $A\cap\tau(A)=\emptyset$ and $A\cup\tau(A)\neq G^{(0)}$. 
Choose a non-empty clopen set $B$ so that 
$B\cap A=\emptyset$, $B\cap\tau(A)=\emptyset$ and 
$A\cup\tau(A)\cup B\neq G^{(0)}$. 
There exists $\sigma_0\in[[G]]$ such that $\sigma_0(\tau(A))\subset B$, 
because $G$ is purely infinite and minimal. 
Define $\sigma_1,\sigma_2\in[[G]]$ by 
\[
\sigma_1(x)=\begin{cases}\tau(x)&x\in A\\
\tau^{-1}(x)&x\in\tau(A)\\ x&\text{otherwise}\end{cases}
\quad\text{and}\quad 
\sigma_2(x)=\begin{cases}\sigma_0(x)&x\in\tau(A)\\
\sigma_0^{-1}(x)&x\in\sigma_0(\tau(A))\\ x&\text{otherwise. }\end{cases}
\]
Then $\sigma=[\sigma_2,\sigma_1]\in D([[G]])$ satisfies 
$\supp(\sigma)\subset A\cup\tau(A)\cup B$ and 
$\sigma(x)=\tau(x)$ for $x\in A$. 

By Lemma \ref{smallsupp} (2), there exists $\alpha\in D([[G]])$ 
such that $\alpha(A\cup\tau(A)\cup B)\subset A$. 
Since $N$ is normalized by $D([[G]])$ and $\tau$ is in $N$, we have 
\[
\tau_1=\alpha^{-1}[\alpha\sigma\alpha^{-1},\tau]\alpha\in N. 
\]
In addition, as 
\[
\supp(\tau\alpha\sigma\alpha^{-1}\tau^{-1})=\tau(\alpha(\supp(\sigma)))
\subset\tau(\alpha(A\cup\tau(A)\cup B))\subset\tau(A)
\]
and $A\cap\tau(A)=\emptyset$, for any $x\in A$ we get 
\[
\tau_1(x)
=(\sigma\alpha^{-1}\tau\alpha\sigma^{-1}\alpha^{-1}\tau^{-1}\alpha)(x)
=(\sigma\alpha^{-1}\alpha)(x)=\sigma(x)=\tau(x). 
\]
Also, one has $\supp(\tau_1)\subset\alpha^{-1}(A\cup\tau(A))\neq G^{(0)}$. 
Let $\tau_2=\tau_1^{-1}\tau$. 
Then $\tau=\tau_1\tau_2$ is a desired decomposition. 
\end{proof}

\begin{lemma}\label{normalize2}
Suppose that $G$ is purely infinite and minimal. 
Let $N\subset[[G]]$ be a non-trivial subgroup normalized by $D([[G]])$. 
For any $\tau\in N$ and $\alpha\in[[G]]$, 
we have $\alpha\tau\alpha^{-1}\in N$. 
\end{lemma}
\begin{proof}
First, we assume $\supp(\alpha)\neq G^{(0)}$ and $\supp(\tau)\neq G^{(0)}$. 
By Proposition \ref{pim} (3), there exists $\sigma\in[[G]]$ 
such that $\sigma(\supp(\alpha))\cap\supp(\tau)=\emptyset$. 
Then 
\[
\alpha\tau\alpha^{-1}=[\alpha,\sigma]\tau[\alpha,\sigma]^{-1}\in N, 
\]
because $\sigma\alpha\sigma^{-1}$ commutes with $\tau$. 

Let us consider the general case. 
Let $\tau\in N$ and $\alpha\in[[G]]$. 
By Lemma \ref{decomp2}, there exist $\tau_1,\tau_2\in N$ 
such that $\supp(\tau_1)\neq G^{(0)}$, $\supp(\tau_2)\neq G^{(0)}$ 
and $\tau=\tau_1\tau_2$. 
Similarly, there exist $\alpha_1,\alpha_2\in[[G]]$ 
such that $\supp(\alpha_1)\neq G^{(0)}$, $\supp(\alpha_2)\neq G^{(0)}$ 
and $\alpha=\alpha_1\alpha_2$. 
(It is easy to find such a decomposition for an element in $[[G]]$.) 
It follows from the proof above that 
\[
\alpha\tau\alpha^{-1}=
(\alpha_1\alpha_2\tau_1\alpha_2^{-1}\alpha_1^{-1})
(\alpha_1\alpha_2\tau_2\alpha_2^{-1}\alpha_1^{-1})\in N. 
\]
\end{proof}

We are now ready to prove the main theorem of this subsection. 

\begin{theorem}\label{simple2}
Suppose that $G$ is purely infinite and minimal. 
Then any non-trivial subgroup of $[[G]]$ 
normalized by the commutator subgroup $D([[G]])$ contains $D([[G]])$. 
In particular, $D([[G]])$ is simple. 
\end{theorem}
\begin{proof}
Let $N$ be a non-trivial subgroup of $[[G]]$ normalized by $D([[G]])$. 
Let $\tau\in N\setminus\{1\}$. 
There exists a non-empty clopen set $A\subset G^{(0)}$ 
such that $A\cap\tau(A)=\emptyset$. 

We would like to show that 
$[\alpha,\beta]$ is in $N$ for any $\alpha,\beta\in[[G]]$. 
First, we assume 
$\supp(\alpha)\neq G^{(0)}$ and $\supp(\beta)\neq G^{(0)}$. 
By Lemma \ref{smallsupp} (1), we can find $\gamma\in[[G]]$ 
such that $\gamma(\supp(\alpha))\cap\supp(\beta)=\emptyset$ 
and $\supp(\alpha)\cup\supp(\gamma)\neq G^{(0)}$. 
There exists $\sigma\in[[G]]$ 
such that $\sigma(\supp(\alpha)\cup\supp(\gamma))\subset A$. 
By Lemma \ref{normalize2}, $\tilde\tau=\sigma^{-1}\tau\sigma$ is in $N$. 
It is easy to see 
\[
(\supp(\alpha)\cup\supp(\gamma))
\cap\tilde\tau(\supp(\alpha)\cup\supp(\gamma))=\emptyset, 
\]
and so $(\tilde\tau\gamma\tilde\tau^{-1})(x)=x$ 
for any $x\in\supp(\alpha)$. 
Hence $\tilde\gamma=[\gamma,\tilde\tau]$ satisfies 
$\tilde\gamma(\supp(\alpha))\cap\supp(\beta)=\emptyset$. 
It follows that 
$\tilde\gamma\alpha\tilde\gamma^{-1}$ commutes with $\beta$. 
Moreover, by Lemma \ref{normalize2}, $\tilde\gamma$ is in $N$, too. 
Therefore 
\[
[\alpha,\beta]=\alpha\beta\alpha^{-1}\beta^{-1}
=\alpha(\tilde\gamma\alpha\tilde\gamma^{-1})\beta
(\tilde\gamma\alpha^{-1}\tilde\gamma^{-1})\alpha^{-1}\beta^{-1}
=[[\alpha,\tilde\gamma],\beta]
\]
is in $N$. 

Next, we would like to show $[\alpha,\beta]\in N$ 
when $\supp(\beta)\neq G^{(0)}$. 
We can find $\alpha_1,\alpha_2\in[[G]]$ such that 
$\supp(\alpha_1)\neq G^{(0)}$, $\supp(\alpha_2)\neq G^{(0)}$ 
and $\alpha=\alpha_1\alpha_2$. 
By the proof above,  
$[\alpha_1,\beta]$ and $[\alpha_2,\beta]$ are in $N$. 
It follows from Lemma \ref{normalize2} that 
\[
[\alpha,\beta]=[\alpha_1\alpha_2,\beta]
=\alpha_1[\alpha_2,\beta]\alpha_1^{-1}[\alpha_1,\beta]
\]
belongs to $N$. 

Finally, let us show $[\alpha,\beta]\in N$. 
We can find $\beta_1,\beta_2\in[[G]]$ such that 
$\supp(\beta_1)\neq G^{(0)}$, $\supp(\beta_2)\neq G^{(0)}$ 
and $\beta=\beta_1\beta_2$. 
By the proof above,  
$[\alpha,\beta_1]$ and $[\alpha,\beta_2]$ are in $N$. 
It follows from Lemma \ref{normalize2} that 
\[
[\alpha,\beta]=[\alpha,\beta_1\beta_2]
=[\alpha,\beta_1]\beta_1[\alpha,\beta_2]\beta_1^{-1}
\]
belongs to $N$, as desired. 
\end{proof}

\begin{remark}
In view of Theorem \ref{simple2}, 
it is a natural problem to determine the abelianization $[[G]]/D([[G]])$. 
In Theorem \ref{surjective}, it will be shown that 
the index map $I:[[G]]\to H_1(G)$ is surjective when $G$ is purely infinite. 
Thus $[[G]]/D([[G]])$ at least has $H_1(G)\cong[[G]]/[[G]]_0$ as its quotient. 
In general, however, 
we do not know what the quotient group $[[G]]_0/D([[G]])$ is. 
When $G$ arises from a one-sided irreducible shift of finite type, 
we will prove that 
$[[G]]/D([[G]])$ is isomorphic to $(H_0(G)\otimes\Z_2)\oplus H_1(G)$ 
(Corollary \ref{abel} (1)). 
\end{remark}

%%%%%%%%%%%%%%%%%%%%%%%%%%%%%%%%%%%%%%%%%%%%%%%%%%%%%%%%%%%%
\section{Surjectivity of the index map}

Throughout this section, 
we let $G$ be an essentially principal \'etale groupoid 
whose unit space is a Cantor set. 
In this section we will show that 
the index map $I:[[G]]\to H_1(G)$ is surjective 
when $G$ is purely infinite (Theorem \ref{surjective}). 

\begin{lemma}\label{dsj}
Suppose that $G$ is purely infinite. 
For any clopen sets $A_1,A_2,\dots,A_n\subset G^{(0)}$, 
there exist compact open $G$-sets $U_1,U_2,\dots,U_n$ such that 
$s(U_i)=A_i$ for all $i$ and the sets $r(U_i)$ are mutually disjoint. 
\end{lemma}
\begin{proof}
The proof is by induction on $n$. 
Suppose that the lemma is known for $n-1$. 
Let $A_1,A_2,\dots,A_n\subset G^{(0)}$ be clopen sets. 
There exist compact open $G$-sets $U_1,U_2,\dots,U_{n-1}$ such that 
$s(U_i)=A_i$ for all $i$ and the sets $r(U_i)$ are mutually disjoint. 
Set $B_i=r(U_i)\cap A_n$ for $i=1,2,\dots,n{-}1$. 
By pure infiniteness of $G$, 
we can find compact open $G$-sets $V_i$ and $W_i$ 
such that $s(V_i)=s(W_i)=B_i$, $r(V_i)\cup r(W_i)\subset B_i$ 
and $r(V_i)\cap r(W_i)=\emptyset$. 
Define compact open $G$-sets $U'_1,U'_2,\dots,U'_{n-1}$ by 
\[
U'_i=((r(U_i)\setminus B_i)\cup V_iB_i)U_i. 
\]
Let 
\[
U_n=W_1B_1\cup W_2B_2\cup\dots\cup W_{n-1}B_{n-1}
\cup(A_n\setminus(B_1\cup B_2\cup\dots\cup B_{n-1})). 
\]
Then $U'_1,U'_2,\dots,U'_{n-1},U_n$ meet the requirement. 
\end{proof}

\begin{theorem}\label{surjective}
Suppose that $G$ is purely infinite. 
The index map $I:[[G]]\to H_1(G)$ is surjective. 
\end{theorem}
\begin{proof}
Let $\delta_1:C_c(G,\Z)\to C(G^{(0)},\Z)$ and 
$\delta_2:C_c(G^{(2)},\Z)\to C_c(G,\Z)$ be as in \cite[Section 3]{M12PLMS}. 
The homology group $H_1(G)$ is $\Ker\delta_1/\Ima\delta_2$. 
For $f\in\Ker\delta_1$, we denote its equivalence class in $H_1(G)$ by $[f]$. 
For a subset $C\subset G$, 
we let $1_C$ denote the characteristic function of $C$. 

Let $f\in\Ker\delta_1$. 
We would like to show that there exists a compact open $G$-set $V$ 
such that $s(V)=r(V)=G^{(0)}$ and $[f]=[1_V]$. 
By \cite[Lemma 7.3 (4)]{M12PLMS}, 
we may assume that there exist compact open $G$-sets $C_1,C_2,\dots,C_n$ 
such that $f=1_{C_1}+1_{C_2}+\dots+1_{C_n}$. 
Since $f$ is in $\Ker\delta_1$, one has 
\[
\sum_{i=1}^n1_{r(C_i)}=\sum_{j=1}^n1_{s(C_j)}. 
\]
Hence there exist clopen sets $A_{ij}\subset G^{(0)}$ for $i,j=1,2,\dots,n$ 
satisfying 
\[
\sum_{i=1}^n1_{A_{ij}}=1_{r(C_j)}\quad\text{and}\quad 
\sum_{j=1}^n1_{A_{ij}}=1_{s(C_i)}.
\]
By the lemma above, 
we can find compact open $G$-sets $U_1,U_2,\dots,U_n$ such that 
$s(U_i)=r(C_i)$ for all $i$ and the sets $r(U_i)$ are mutually disjoint. 
Put 
\[
V_{ij}=U_iC_iA_{ij}U_j^{-1}. 
\]
Then $V_{ij}$ is a compact open $G$-set which satisfies 
\[
r(V_{ij})=U_iC_iA_{ij}U_j^{-1}U_jA_{ij}C_iU_i^{-1}
=U_iC_iA_{ij}r(C_j)A_{ij}C_iU_i^{-1}=r(U_iC_iA_{ij})
\]
and
\[
s(V_{ij})=U_jA_{ij}C_i^{-1}U_i^{-1}U_iC_iA_{ij}U_j^{-1}
=U_jA_{ij}C_i^{-1}r(C_i)C_iA_{ij}U_j^{-1}=s(A_{ij}U_j^{-1}). 
\]
Therefore $V=\bigcup_{i,j}V_{ij}$ is a compact open $G$-set satisfying 
\[
r(V)=\bigcup_{i=1}^nr(U_iC_i)=\bigcup_{i=1}^nr(U_i)
=\bigcup_{j=1}^ns(U_j^{-1})=s(V). 
\]
By using \cite[Lemma 7.3 (1)]{M12PLMS} repeatedly, we obtain 
\begin{align*}
[1_V]
&=\left[\sum_{i,j}1_{V_{ij}}\right]
=\left[\sum_{i,j}1_{U_iC_iA_{ij}U_j^{-1}}\right]\\
&=\left[\sum_{i,j}(1_{U_iC_iA_{ij}}+1_{A_{ij}U_j^{-1}})\right]\\
&=\left[\sum_i1_{U_iC_i}+\sum_j1_{U_j^{-1}}\right]\\
&=\left[\sum_i(1_{U_i}+1_{C_i})+\sum_j1_{U_j^{-1}}\right]
\end{align*}
in $H_1(G)$. 
This, together with \cite[Lemma 7.3 (4)]{M12PLMS}, 
implies $[1_V]=[f]$ in $H_1(G)$. 
Hence $\tilde V=V\cup(G^{(0)}\setminus s(V))$ is a desired $G$-set. 
\end{proof}

We will use the following lemma in Section 6.5. 
For $f\in C(G^{(0)},\Z)$, 
we denote its equivalence class in $H_0(G)$ by $[f]_G$. 

\begin{lemma}\label{pi>clopen}
Suppose that $G$ is purely infinite. 
For any $h\in H_0(G)$, 
there exists a non-empty clopen set $A$ such that $[1_A]_G=h$. 
\end{lemma}
\begin{proof}
There exists $f\in C(G^{(0)},\Z)$ such that $[f]_G=h$. 
We can find non-empty clopen sets $C_1,C_2,\dots,C_n$ and $D_1,D_2,\dots,D_m$ 
such that 
\[
f=\sum_{i=1}^n1_{C_i}-\sum_{j=1}^m1_{D_j}. 
\]
From Lemma \ref{dsj}, 
there exist compact open $G$-sets $U_1,U_2,\dots,U_n$ and $V_1,V_2,\dots,V_m$ 
such that $s(U_i)=C_i$, $s(V_j)=D_j$ for all $i,j$ and 
the sets $r(U_i)$, $r(V_j)$ are mutually disjoint. 
Let 
\[
C=\bigcup_{i=1}^nr(U_i)\quad\text{and}\quad D=\bigcup_{j=1}^mr(V_j). 
\]
Then $[f]_G=[1_C]_G-[1_D]_G$. 
Since $G$ is purely infinite, 
there exist compact open $G$-sets $U,V\subset G$ 
such that $s(U)=s(V)=D$, $r(U)\cup r(V)\subset D$ 
and $r(U)\cap r(V)=\emptyset$. 
Put $E=D\setminus(r(U)\cup r(V))$. 
One has 
\[
[1_E]_G=[1_D]_G-[1_{r(U)}]_G-[1_{r(V)}]_G=[1_D]_G-[1_D]_G-[1_D]_G=-[1_D]_G, 
\]
which implies 
\[
h=[f]_G=[1_C]_G-[1_D]_G=[1_C]_G+[1_E]_G=[1_{C\cup E}]_G, 
\]
as desired. 
\end{proof}

%%%%%%%%%%%%%%%%%%%%%%%%%%%%%%%%%%%%%%%%%%%%%%%%%%%%%%%%%%%%
\section{Groupoids from shifts of finite type}

In this section 
we study topological full groups of \'etale groupoids 
arising from shifts of finite type. 
We refer the reader to \cite{CK80Invent,LM} 
for background material on symbolic dynamical systems.

%%%%%%%%%%%%%%%%%%%%%%%%%%%%%%%%%%%%%%%%%%%%%%%%%%%%%%%%%%%%
\subsection{Preliminaries}

Let $(\V,\E)$ be a finite directed graph, 
where $\V$ is a finite set of vertices and $\E$ is a finite set of edges. 
For $e\in\E$, $i(e)$ denotes the initial vertex of $e$ and 
$t(e)$ denotes the terminal vertex of $e$. 
Let $M=(M(\xi,\eta))_{\xi,\eta\in\V}$ be the adjacency matrix of $(\V,\E)$, 
that is, 
\[
M(\xi,\eta)=\#\{e\in\E\mid i(e)=\xi,\ t(e)=\eta\}. 
\]
We assume that $M$ is irreducible 
(i.e. for all $\xi,\eta\in\V$ 
there exists $n\in\N$ such that $M^n(\xi,\eta)>0$) 
and that $M$ is not a permutation matrix. 
Define 
\[
X=\{(x_k)_{k\in\N}\in\E^\N\mid t(x_k)=i(x_{k+1})\quad\forall k\in\N\}. 
\]
With the product topology, $X$ is a Cantor set 
(see condition (I) defined in \cite{CK80Invent}). 
Define a surjective continuous map $\sigma:X\to X$ by 
\[
\sigma(x)_k=x_{k+1}\quad k\in\N,\ x=(x_k)_k\in X. 
\]
In other words, $\sigma$ is the (one-sided) shift on $X$.  
It is easy to see that $\sigma$ is a local homeomorphism. 
The dynamical system $(X,\sigma)$ is called 
the one-sided irreducible shift of finite type 
associated with the graph $(\V,\E)$. 

The \'etale groupoid $G$ for $(X,\sigma)$ is given by 
\[
G=\{(x,n,y)\in X\times\Z\times X\mid
\exists k,l\in\N,\ n=k{-}l,\ \sigma^k(x)=\sigma^l(y)\}. 
\]
The topology of $G$ is generated by the sets 
$\{(x,k{-}l,y)\in G\mid x\in A,\ y\in B,\ \sigma^k(x)=\sigma^l(y)\}$, 
where $A,B\subset X$ are open and $k,l\in\N$. 
Two elements $(x,n,y)$ and $(x',n',y')$ in $G$ are composable 
if and only if $y=x'$, and the multiplication and the inverse are 
\[
(x,n,y)\cdot(y,n',y')=(x,n{+}n',y'),\quad (x,n,y)^{-1}=(y,-n,x). 
\]
We identify $X$ with the unit space $G^{(0)}$ via $x\mapsto(x,0,x)$. 

A multiindex $\mu=(e_1,e_2,\dots,e_k)$ with $e_j\in\E$ is called a word. 
We denote by $\lvert\mu\rvert$ the length $k$ of $\mu$ and 
write $i(\mu)=i(e_1)$, $t(\mu)=t(e_k)$. 
Every $e\in\E$ is regarded as a word of length one. 
When $\mu=(e_1,e_2,\dots,e_k)$ and $\nu=(f_1,f_2,\dots,f_l)$ are words, 
we write $\mu\nu$ for the word 
$\mu\nu=(e_1,\dots,e_k,f_1,\dots,f_l)$ of length $k+l$. 
A word $\mu=(e_1,e_2,\dots,e_k)$ is said to be admissible 
if $t(e_j)=i(e_{j+1})$ for every $j=1,2,\dots,k-1$. 
For a word $\mu=(e_1,e_2,\dots,e_k)$, 
\[
C_\mu=\{(x_n)_{n\in\N}\in X\mid x_j=e_j\quad\forall j=1,2,\dots,k\}
\]
is a clopen subset of $X$ and is called a cylinder set. 
Clearly, $C_\mu$ is non-empty if and only if $\mu$ is admissible. 
For any two cylinder sets $C_\mu$ and $C_\nu$, we have 
$C_\mu\cap C_\nu=\emptyset$, $C_\mu\subset C_\nu$ or $C_\mu\supset C_\nu$. 
For $\xi\in\V$, we set 
\[
D_\xi=\{(x_n)_{n\in\N}\in X\mid i(x_1)=\xi\}. 
\]
For words $\mu$ and $\nu$ with $t(\mu)=t(\nu)$, 
we define a compact open $G$-set $U_{\mu,\nu}$ by 
\[
U_{\mu,\nu}=\{(x,\lvert\mu\rvert-\lvert\nu\rvert,y)\in G\mid
\sigma^{\lvert\mu\rvert}(x)=\sigma^{\lvert\nu\rvert}(y),\ 
x\in C_\mu,\ y\in C_\nu\}. 
\]
The subsets $U_{\mu,\nu}$ form a base for the topology of $G$. 

\begin{lemma}\label{SFT>pim}
For any non-empty clopen set $Y\subset X$, 
the \'etale groupoid $G|Y$ is purely infinite and minimal. 
\end{lemma}
\begin{proof}
By the remark following Definition \ref{pi}, we may assume $Y=X$. 
We check condition (2) of Proposition \ref{pim}. 
Let $A$ and $B$ be non-empty clopen subsets of $X$. 
Write $A$ and $B$ as disjoint unions 
$A=\bigcup_{\mu\in I}C_\mu$ and $B=\bigcup_{\nu\in J}C_\nu$ 
of non-empty cylinder sets, respectively. 
By dividing up the cylinder sets $C_\nu$ if necessary, 
we may assume that $\#J$ is not less than $\#I$. 
Let $f:I\to J$ be an injection. 
Since the graph $(\V,\E)$ is irreducible, for each $\mu\in I$, 
we can find an admissible word $\mu'$ such that 
$C_{\mu'}\subset C_{f(\mu)}$ and $t(\mu')=t(\mu)$. 
Set 
\[
U=\bigcup_{\mu\in I}U_{\mu',\mu}. 
\]
Then $U$ is a compact open $G$-set satisfying $s(U)=A$ and $r(U)\subset B$. 
\end{proof}

Let us recall the homology groups $H_n(G)$ from \cite[Section 4]{M12PLMS}. 
The \'etale groupoid $G$ has an open subgroupoid 
\[
K=\{(x,0,y)\in G\}. 
\]
It is well-known that $K$ is an AF (approximately finite) groupoid 
(see \cite[Defintion III.1.1]{Rtext} or \cite[Definition 2.2]{M12PLMS} 
for the definition of AF groupoids) and that 
$H_0(K)$ is isomorphic to the inductive limit 
\[
\lim_{\longrightarrow}(\Z^\V\stackrel{M^t}{\longrightarrow}
\Z^\V\stackrel{M^t}{\longrightarrow}\dots), 
\]
where $M^t$ is the transpose of $M$ and 
is thought of as a homomorphism from $\Z^\V$ to $\Z^\V$. 
This group is called the dimension group of the subshift $(X,\sigma)$. 
Notice that $K$ is not necessary minimal. 
Indeed, $K$ is minimal if and only if 
the subshift $(X,\sigma)$ is topologically mixing. 
For $f\in C(X,\Z)$, 
we denote its equivalence class in $H_0(K)$ by $[f]_K$. 
There exists an automorphism $\delta:H_0(K)\to H_0(K)$ such that 
\[
[1_{r(U)}]_K=\delta^{-n}\left([1_{s(U)}]_K\right)
\]
holds for any $n\in\Z$ and 
any compact open $G$-set $U$ contained in $\{(x,n,y)\in G\}$. 
Then, by virtue of \cite[Theorem 4.14]{M12PLMS}, we have 
\[
H_0(G)\cong\Coker(\id-\delta)\cong\Coker(\id-M^t), 
\]
\[
H_1(G)\cong\Ker(\id-\delta)\cong\Ker(\id-M^t)
\]
and $H_n(G)=0$ for $n\geq2$. 
In particular, $H_0(G)$ is a finitely generated abelian group and 
$H_1(G)$ is the torsion-free part of $H_0(G)$. 
We remark that $\Coker(\id-M^t)$ is called the Bowen-Franks group of $M$ 
(\cite[Definition 7.4.15]{LM}). 
For $f\in C(X,\Z)$, 
we denote its equivalence class in $H_0(G)$ by $[f]_G$. 

K. Matsumoto \cite{Matsu10PJM, MatsuProcAMS} obtained 
the following classification theorem (see also \cite[Theorem 5.1]{M12PLMS}). 
Actually, 
only the case of $Y_i=X_i$ was considered in \cite{Matsu10PJM, MatsuProcAMS}, 
but the proof works in the general case. 

\begin{theorem}
[{\cite[Theorem 1.1]{Matsu10PJM}, \cite[Theorem 1.1]{MatsuProcAMS}}]
\label{classify}
For $i=1,2$, let $(\V_i,\E_i)$ be a finite directed graph and 
suppose that the adjacency matrix $M_i$ is irreducible and 
is not a permutation matrix. 
Let $(X_i,\sigma_i)$ be the one-sided shift associated with $(\V_i,\E_i)$ 
and let $G_i$ be the \'etale groupoid for $(X_i,\sigma_i)$. 
Let $Y_i\subset X_i$ be a non-empty clopen subset. 
If there exists an isomorphism $\phi:H_0(G_1)\to H_0(G_2)$ 
such that $\phi([1_{Y_1}]_{G_1})=[1_{Y_2}]_{G_2}$ and 
$\det(\id-M_1^t)=\det(\id-M_2^t)$, then 
$G_1|Y_1$ is isomorphic to $G_2|Y_2$ as an \'etale groupoid. 
\end{theorem}

In the theorem above, we do not know whether or not 
the hypothesis $\det(\id-M_1^t)=\det(\id-M_2^t)$ is really necessary. 
If the torsion-free part of $H_0(G_i)$ is non-zero 
(or equivalently $H_1(G_i)\neq0$), then $\det(\id-M_i^t)=0$. 
Thus the hypothesis $\det(\id-M_1^t)=\det(\id-M_2^t)$ automatically holds. 
When $H_0(G_i)$ is a torsion group (or equivalently $H_1(G_i)=0$), 
we always have $\lvert\det(\id-M_1^t)\rvert=\lvert\det(\id-M_2^t)\rvert$, 
and so the only issue is the sign of $\det(\id-M_i^t)$. 

We also remark that 
the pair $\Coker(\id-M^t)$ and $\det(\id-M^t)$ is a complete invariant 
of flow equivalence for two-sided irreducible shifts of finite type 
(\cite{F84ETDS}). 

\begin{remark}\label{HigThomp}
We would like to note a close connection 
between the topological full group $[[G]]$ and Thompson's groups. 
I wish to thank R. Grigorchuk, who drew my attention to this connection. 
For background on Thompson's groups, 
see the survey article \cite{CFP} and the references therein. 
In 1965 Richard Thompson gave the first example of 
a finitely presented infinite simple group. 
G. Higman \cite{Hig} and K. S. Brown \cite{Br87JPAA} later introduced 
infinite families $F_{n,r}\subset T_{n,r}\subset V_{n,r}$ 
of finitely presented infinite simple groups generalizing Thompson's example. 
Notice that $V_{n,r}$ is written $G_{n,r}$ in \cite{Br87JPAA}. 
Here we describe $F_{n,r}$ and $V_{n,r}$ 
as subgroups of PL (piecewise linear) maps on an interval. 

Let $n\in\N\setminus\{1\}$ and $r\in\N$. 
Let $F_{n,r}$ be the set of all PL homeomorphisms $f:[0,r]\to[0,r]$ 
with finitely many singularities 
such that all singularities of $f$ are in $\Z[1/n]$ and 
the derivative of $f$ at any non-singular point is $n^k$ for some $k\in\Z$. 
It is easy to see that 
$F_{n,r}$ is closed under composition and forms a group. 
It is known that $F_{n,r}$ is independent of $r$ up to isomorphism, 
the commutator subgroup $D(F_{n,r})$ is simple, 
the abelianization $F_{n,r}/D(F_{n,r})$ is $\Z^n$, and 
$F_{n,r}$ is finitely presented (\cite{Br87JPAA}). 
It is an outstanding open problem whether the group $F_{n,r}$ is amenable. 
%(Note that J. T. Moore \cite{Mo} has recently announced 
%a proof of the amenability of $F_{2,1}$.) 
The group $V_{n,r}$ is defined to be the group of 
all right continuous PL bijections $f:[0,r)\to[0,r)$ 
with finitely many singularities 
such that all singularities of $f$ are in $\Z[1/n]$, 
the derivative of $f$ at any non-singular point is $n^k$ for some $k\in\Z$ 
and $f$ maps $\Z[1/n]\cap[0,r)$ to itself. 
We call $V_{n,r}$ the Higman-Thompson groups. 
It is known that $V_{n,r}$ is finitely presented, 
the commutator subgroup $D(V_{n,r})$ is simple, and 
the abelianization $V_{n,r}/D(V_{n,r})$ is trivial when $n$ is even 
and is $\Z_2$ when $n$ is odd (\cite{Br87JPAA}). 
Also, $V_{n,r}$ contains non-abelian free groups. 
For $n_1,n_2,r_1,r_2$, 
the groups $V_{n_1,r_1}$ and $V_{n_2,r_2}$ are isomorphic if and only if 
$n_1=n_2$ and $\gcd(n_1{-}1,r_1)=\gcd(n_2{-}1,r_2)$ (\cite{P11JA}). 
It is also possible to describe $F_{n,r}\subset V_{n,r}$ 
as groups of homeomorphisms of the Cantor set. 

Let $(X,\sigma)$ be the full shift over $n$ symbols and 
let $G$ be the \'etale groupoid associated to $(X,\sigma)$. 
The $C^*$-algebra $C^*_r(G)$ is the Cuntz algebra $\mathcal{O}_n$. 
V. V. Nekrashevych \cite{Ne04JOP} proved that 
the Higman-Thompson group $V_{n,1}$ is identified with a certain subgroup 
of the unitary group $U(\mathcal{O}_n)$ (\cite[Proposition 9.6]{Ne04JOP}). 
This identification yields an isomorphism between 
the Higman-Thompson group $V_{n,1}$ and the topological full group $[[G]]$. 
Therefore, when $(X,\sigma)$ is a shift of finite type 
which is not necessarily a full shift, 
$[[G]]$ can be thought of as a generalization of $V_{n,1}$. 
We will prove that $[[G]]$ is of type F$_\infty$ in later subsections. 
In Section 6.7.1, we revisit the \'etale groupoids of full shifts. 
\end{remark}

%%%%%%%%%%%%%%%%%%%%%%%%%%%%%%%%%%%%%%%%%%%%%%%%%%%%%%%%%%%%
\subsection{Equivalence between clopen subsets}

We use the notation of the preceding subsection. 
As mentioned above, $H_0(G)$ is isomorphic to $\Coker(\id-M^t)$. 
This isomorphism is described as follows. 
For $\xi\in\V$, 
we consider the clopen set $D_\xi=\{x\in X\mid i(x_1)=\xi\}$. 
The map sending $\xi\mapsto[1_{D_\xi}]_G$ gives rise to 
a homomorphism from $\Z^\V$ to $H_0(G)$. 
It is well-known that this homomorphism is surjective and 
its kernel is $(\id-M^t)\Z^\V$ 
(see \cite[Proposition 3.1]{C81Invent}, \cite[Theorem 4.14]{M12PLMS}). 
Thus, it induces $H_0(G)\cong\Coker(\id-M^t)$. 

\begin{theorem}\label{Hopf}
For non-empty clopen sets $A,B\subset X$, 
the following conditions are equivalent. 
\begin{enumerate}
\item $[1_A]_G=[1_B]_G$ in $H_0(G)$. 
\item There exists a compact open $G$-set $U$ 
such that $s(U)=A$ and $r(U)=B$. 
\end{enumerate}
\end{theorem}
\begin{proof}
The implication from (2) to (1) is immediate from the definition of $H_0(G)$. 
We show that (1) implies (2). 
Write $A$ and $B$ as disjoint unions 
$A=\bigcup_{\mu\in I}C_\mu$ and $B=\bigcup_{\nu\in J}C_\nu$ 
of non-empty cylinder sets $C_\mu,C_\nu$, respectively. 
By dividing up a cylinder set if necessary, 
we may assume that there exist $\mu_0\in I$ and $\nu_0\in J$ 
such that $t(\mu_0)=\zeta=t(\nu_0)$. 
We define $a,b\in\Z^\V$ by 
\[
a(\xi)=\#\{\mu\in I\mid t(\mu)=\xi\},\quad 
b(\xi)=\#\{\nu\in J\mid t(\nu)=\xi\}\quad\forall \xi\in\V. 
\]
Since $[1_A]_G$ equals $[1_B]_G$, $a-b\in\Z^\V$ is in the image of $\id-M^t$. 
Therefore, 
there exist $c,d\in\N^\V$ such that $a+(M^t-\id)c=b+(M^t-\id)d$. 
By dividing the clopen set $D_\zeta$ into cylinder sets, 
we can find a finite set $L$ of admissible words such that 
$D_\zeta$ equals the disjoint union $\bigcup_{\lambda\in L}C_\lambda$ and 
\[
\#\{\lambda\in L\mid t(\lambda)=\xi\}\geq\max\{c(\xi),d(\xi)\}
\quad\forall \xi\in\V. 
\]
Define $I'$ and $J'$ by 
\[
I'=(I\setminus\{\mu_0\})\cup\{\mu_0\lambda\mid\lambda\in L\}
\quad\text{and}\quad 
J'=(J\setminus\{\nu_0\})\cup\{\nu_0\lambda\mid\lambda\in L\}. 
\]
Then we have 
$A=\bigcup_{\mu\in I'}C_\mu$ and $B=\bigcup_{\nu\in J'}C_\nu$. 
Also, 
\[
a'(\xi)=\#\{\mu\in I'\mid t(\mu)=\xi\}\quad\text{and}\quad 
b'(\xi)=\#\{\nu\in J'\mid t(\nu)=\xi\}
\]
still satisfy $a'+(M^t-\id)c=b'+(M^t-\id)d$, because $a'-a=b'-b$. 
For each $\xi\in\V$, 
we can find a finite subset $I'_\xi\subset I'$ such that 
$t(\mu)=\xi$ for all $\mu\in I'_\xi$ and $\#I'_\xi=c(\xi)$. 
Define $I''$ by 
\[
I''=\left(I'\setminus\bigcup_{\xi\in\V}I'_\xi\right)
\cup
\bigcup_{\xi\in\V}\{\mu e\mid\mu\in I'_\xi,\ e\in\E,\ i(e)=\xi\}. 
\]
Then we have $A=\bigcup_{\mu\in I''}C_\mu$ and 
\[
a''(\xi)=\#\{\mu\in I''\mid t(\mu)=\xi\}
=a'(\xi)-c(\xi)+\sum_{\eta}M(\eta,\xi)c(\eta), 
\]
that is, $a''=a'+(M^t-\id)c$. 
Similarly, we can construct a set $J''$ of admissible words such that 
$B=\bigcup_{\nu\in J''}C_\nu$ and 
$b''=b'+(M^t-\id)d$, where 
\[
b''(\xi)=\#\{\nu\in J''\mid t(\nu)=\xi\}. 
\]
It follows from $a'+(M^t-\id)c=b'+(M^t-\id)d$ that 
there exists a bijection $f:I''\to J''$ 
such that $t(\mu)=t(f(\mu))$ for every $\mu\in I''$. 
Then 
\[
U=\bigcup_{\mu\in I''}U_{f(\mu),\mu}
\]
is a compact open $G$-set satisfying $r(U)=B$ and $s(U)=A$. 
\end{proof}

%%%%%%%%%%%%%%%%%%%%%%%%%%%%%%%%%%%%%%%%%%%%%%%%%%%%%%%%%%%%
\subsection{The Haagerup property}

As in Section 6.1, 
we let $G$ be the \'etale groupoid 
arising from a one-sided irreducible shift of finite type $(X,\sigma)$. 
We use the notation of Section 6.1. 
We have already seen that $[[G]]$ is not amenable 
(see Proposition \ref{pi>Z2*Z3} and Proposition \ref{SFT>pim}). 
In this subsection, 
we would like to see that $[[G]]$ has the Haagerup property, 
i.e. it is a-T-menable in the sense of M. Gromov. 
Indeed, this is a corollary of B. Hughes's theorem \cite[Theorem 1.1]{H09GGD}, 
which states that any locally finitely determined group of local similarities 
on a compact ultrametric space has the Haagerup property. 
One can regard $X$ as a compact ultrametric space and 
introduce a suitable finite similarity structure for (cylinder sets of) $X$. 
Then \cite[Theorem 1.1]{H09GGD} applies to show that 
$[[G]]$ has the Haagerup property. 
But, instead of proceeding in this way, 
we would like to provide a self-contained succinct proof of it 
for the convenience of the reader. 
I wish to thank R. Grigorchuk, 
who drew my attention to this problem and to the reference \cite{H09GGD}. 

We recall the notion of zipper actions from \cite{H09GGD}. 

\begin{definition}[{\cite[Defintion 5.1]{H09GGD}}]\label{zipper}
An action $\phi:\Gamma\curvearrowright\Omega$ of 
a discrete group $\Gamma$ on a set $\Omega$ is called a zipper action 
if there exists a subset $Z\subset\Omega$ such that the following hold. 
\begin{enumerate}
\item For every $\alpha\in\Gamma$, 
the symmetric difference $\phi_\alpha(Z)\bigtriangleup Z$ is finite. 
\item For every $r>0$, 
$\{\alpha\in\Gamma\mid\#(\phi_\alpha(Z)\bigtriangleup Z)\leq r\}$ 
is finite. 
\end{enumerate}
\end{definition}

\begin{theorem}[{\cite[Theorem 5.3]{H09GGD}}]
If a discrete group $\Gamma$ admits a zipper action, 
then $\Gamma$ has the Haagerup property. 
\end{theorem}

In order to apply the theorem above to $\Gamma=[[G]]$, 
we have to construct a zipper action of $[[G]]$. 
Let $\Omega$ be the set of equivalence classes of 
non-empty compact open $G$-sets $V$ such that $s(V)$ is a cylinder set. 
Two such $V_1$ and $V_2$ are equivalent 
provided $V_2=V_1U_{\mu,\nu}$, where $s(V_1)=C_\mu$ and $s(V_2)=C_\nu$. 
We denote the equivalence class of $V$ in $\Omega$ by $[V]$. 
Define an action $\phi:[[G]]\curvearrowright\Omega$ 
by $\phi_\alpha([V])=[WV]$, 
where $W$ is the compact open $G$-set such that $\alpha=\pi_W$. 

\begin{theorem}\label{Haagerup}
The action $\phi:[[G]]\curvearrowright\Omega$ is a zipper action. 
In particular, $[[G]]$ has the Haagerup property. 
\end{theorem}
\begin{proof}
Define a subset $Z\subset\Omega$ by 
\[
Z=\{[C_\mu]\in\Omega\mid
\text{$C_\mu$ is a non-empty cylinder set}\}. 
\]
Let $\alpha=\pi_W\in[[G]]$. 
There exist admissible words 
$\mu_1,\mu_2,\dots,\mu_k$, $\nu_1,\nu_2,\dots,\nu_k$ such that 
$t(\mu_i)=t(\nu_i)$ for all $i=1,2,\dots,k$ and 
$W$ equals the disjoint union $\bigcup_{i=1}^kU_{\mu_i,\nu_i}$. 
Suppose that there exists $L\subset\{1,2,\dots,k\}$ such that 
$\bigcup_{i\in L}U_{\mu_i,\nu_i}$ equals $U_{\mu,\nu}$ 
for some admissible words $\mu,\nu$ with $t(\mu)=t(\nu)$. 
Then we may replace $\bigcup_{i\in L}U_{\mu_i,\nu_i}$ with $U_{\mu,\nu}$. 
By repeating this procedure if necessary, 
we may assume that no such $L$ exists. 
Let $m(\alpha)=\max\{\lvert\mu_i\rvert,\lvert\nu_i\rvert\mid i=1,2,\dots,k\}$. 

We would like to estimate $\#(\phi_\alpha(Z)\bigtriangleup Z)$. 
Let $[C_\lambda]\in Z$. 
Then $\phi_\alpha([C_\lambda])=[WC_\lambda]$ is in $Z$ if and only if 
there exists $i$ such that $C_\lambda$ is contained in $C_{\nu_i}$. 
Hence, if $[C_\lambda]$ is in $Z\setminus\phi_\alpha^{-1}(Z)$, 
then $C_\lambda$ is not contained in any $C_{\nu_i}$, and so in particular 
\[
\lvert\lambda\rvert<\max\{\lvert\nu_i\rvert\mid i=1,2,\dots,k\}\leq m(\alpha). 
\] 
Therefore $\#(Z\setminus\phi_\alpha^{-1}(Z))=\#(\phi_\alpha(Z)\setminus Z)$ 
does not exceed the number of admissible words 
of length less than $m(\alpha)$, and hence is finite. 
In the same way, 
we can show $\#(Z\setminus\phi_\alpha(Z))$ is finite, 
which implies condition (1) of Definition \ref{zipper}. 
Next, let us estimate $\#(\phi_\alpha(Z)\bigtriangleup Z)$ from below. 
Without loss of generality, we may assume $m(\alpha)=\lvert\nu_1\rvert$. 
Let $\nu_1=(e_1,e_2,\dots,e_{m(\alpha)})$. 
Put $\lambda_j=(e_1,e_2,\dots,e_j)$ for $j=1,2,\dots,m(\alpha){-}1$. 
Then $C_{\lambda_j}$ is not contained in any $C_{\nu_i}$. 
It follows that $\phi_\alpha([C_{\lambda_j}])$ is not in $Z$, 
and so we get 
\[
\#(\phi_\alpha(Z)\bigtriangleup Z)
\geq\#(\phi_\alpha(Z)\setminus Z)\geq m(\alpha)-1. 
\]

Hence, in order to prove condition (2) of Definition \ref{zipper}, 
for given $r>0$, it suffices to show that 
the set of $\alpha\in[[G]]$ such that $m(\alpha)<r$ is finite. 
Clearly there exist finitely many partitions of $X$ into cylinder sets $C_\mu$ 
satisfying $\lvert\mu\rvert<r$. 
Hence the number of $\alpha=\pi_W\in[[G]]$ 
for which there exist $\mu_1,\mu_2,\dots,\mu_k$ and $\nu_1,\nu_2,\dots,\nu_k$ 
such that $\lvert\mu_i\rvert<r$, $\lvert\nu_i\rvert<r$ and 
$W=\bigcup_{i=1}^kU_{\mu_i,\nu_i}$ is at most finite. 
The proof is complete. 
\end{proof}

For any clopen set $Y\subset X$, $[[G|Y]]$ is a subgroup of $[[G]]$. 
It follows from the theorem above that 
$[[G|Y]]$ also has the Haagerup property.

%%%%%%%%%%%%%%%%%%%%%%%%%%%%%%%%%%%%%%%%%%%%%%%%%%%%%%%%%%%%
\subsection{Kernel of the index map}

As in Section 6.1, 
we let $G$ be the \'etale groupoid 
arising from a one-sided irreducible shift of finite type $(X,\sigma)$. 
We use the notation of Section 6.1. 
Fix a non-empty clopen subset $Y\subset X$. 
In this subsection, we would like to show that 
$[[G|Y]]_0$ is finitely generated (Theorem \ref{fg}). 

First we need to write down the index map $I:[[G|Y]]\to H_1(G|Y)$ explicitly. 
Any $\alpha\in[[G|Y]]$ extends to an element of $[[G]]$ 
by letting $\alpha(x)=x$ for $x\in X\setminus Y$, 
and so we may regard $[[G|Y]]$ as a subgroup of $[[G]]$. 
Since $G$ is minimal, 
the inclusion $G|Y\hookrightarrow G$ induces 
isomorphisms between $H_n(G|Y)$ and $H_n(G)$ 
(see Proposition 3.5 and Theorem 3.6 of \cite{M12PLMS}). 
These observations yield the diagram 
\[
\begin{CD}[[G|Y]]@>I>>H_1(G|Y)\\
@VVV@VV\cong V\\
[[G]]@>>I>H_1(G), \end{CD}
\]
which becomes commutative by the definition of the index map. 
From now on, we identify $H_1(G|Y)$ with $H_1(G)$. 
As mentioned in Section 6.1, 
$H_1(G)$ is isomorphic to $\Ker(\id-\delta)$, which is a subgroup of $H_0(K)$. 
For $\alpha\in[[G|Y]]$, 
we would like to describe $I(\alpha)\in H_1(G|Y)\cong H_1(G)$ 
as an element of $\Ker(\id-\delta)$. 
Take a compact open $G|Y$-set $U$ such that $\alpha=\pi_U$. 
For $n\in\Z$, we put 
\[
S(\alpha,n)=\{x\in Y\mid(\alpha(x),n,x)\in U\}. 
\]
Note that $S(\alpha,n)$ is empty except for finitely many $n$. 
We have 
\[
[1_Y]_K=\sum_{n\in\Z}[1_{S(\alpha,n)}]_K
\]
in $H_0(K)$, 
because $\{S(\alpha,n)\mid n\in\Z\}$ is a clopen partition of $Y$. 
Since $\{\alpha(S(\alpha,n))\mid n\in\Z\}$ is also a clopen partition of $Y$, 
one obtains 
\[
[1_Y]_K=\sum_{n\in\Z}[1_{\alpha(S(\alpha,n))}]_K
=\sum_{n\in\Z}\delta^{-n}\left([1_{S(\alpha,n)}]_K\right)
\]
in $H_0(K)$. 
Therefore, we get 
\begin{equation}
\sum_{n\in\Z}(\id-\delta^{-n})\left([1_{S(\alpha,n)}]_K\right)=0. 
\label{mouse}
\end{equation}
For $n\in\Z$, we define homomorphisms $\delta^{(n)}:H_0(K)\to H_0(K)$ by 
\[
\delta^{(n)}=\begin{cases}\id+\delta+\dots+\delta^{n-1}&n>0\\
0&n=0\\
-(\delta^{-1}+\delta^{-2}+\dots+\delta^n)&n<0, \end{cases}
\]
so that $(\id-\delta)\delta^{(n)}=\id-\delta^n$ hold. 
It follows from \eqref{mouse} that 
\[
\sum_{n\in\Z}(\id-\delta)
\left(\delta^{(-n)}\left([1_{S(\alpha,n)}]_K\right)\right)=0, 
\]
which means that $\sum_{n\in\Z}\delta^{(-n)}([1_{S(\alpha,n)}]_K)$ 
is in $\Ker(\id-\delta)$. 
The proof of \cite[Theorem 4.14]{M12PLMS} immediately implies the following. 

\begin{lemma}\label{index}
In the setting above, $I(\alpha)\in H_1(G|Y)\cong H_1(G)$ is identified with 
\[
\sum_{n\in\Z}\delta^{(-n)}\left([1_{S(\alpha,n)}]_K\right)\in\Ker(\id-\delta). 
\]
\end{lemma}

In what follows, we regard $[[K|Y]]$ as a subgroup of $[[G|Y]]$. 
As mentioned in Section 6.1, $K$ is an AF groupoid, and so is $K|Y$. 
Hence, by \cite[Proposition 3.2]{M06IJM}, $[[K|Y]]$ is a locally finite group. 
More precisely, $[[K|Y]]$ is written as 
an increasing union of finite direct sums of symmetric groups. 
In particular, $[[K|Y]]$ is contained in $[[G|Y]]_0=\Ker I$. 
For two clopen sets $A,B\subset X$, it is well-known that 
$[1_A]_K=[1_B]_K$ in $H_0(K)$ if and only if 
there exists a compact open $K$-set $U$ such that $s(U)=A$ and $r(U)=B$. 
For $\alpha\in[[G|Y]]$, it is also easy to see that 
$\alpha$ belongs to $[[K|Y]]$ if and only if $S(\alpha,0)=Y$. 

We introduce the set of transpositions. 
We call $\alpha\in[[G|Y]]$ a transposition 
if there exists a clopen set $A\subset Y$ such that 
$A\cap\alpha(A)=\emptyset$, $\supp(\alpha)=A\cup\alpha(A)$ and $\alpha^2=1$. 
By Theorem \ref{Hopf}, 
if non-empty clopen sets $A,B\subset Y$ satisfy 
$A\cap B=\emptyset$ and $[1_A]_G=[1_B]_G$, then 
there exists a transposition $\alpha\in[[G|Y]]$ such that 
$\alpha(A)=B$ and $\supp(\alpha)=A\cup B$. 
We let $\mathcal{T}$ denote the set of all transpositions. 
By \cite[Lemma 7.3]{M12PLMS}, 
$\mathcal{T}$ is contained in $[[G|Y]]_0=\Ker I$. 
It is also easy to see that 
the group $[[K|Y]]$ is generated by $\mathcal{T}\cap[[K|Y]]$. 
Below we will see that $[[G|Y]]_0$ is generated by $\mathcal{T}$ 
(Lemma \ref{Tgenerate} and Theorem \ref{fg}). 

\begin{lemma}\label{transp}
Let $A,B\subset Y$ be clopen sets and let $n\in\Z$. 
If $\delta^n([1_A]_K)=[1_B]_K$ in $H_0(K)$, then 
there exists a compact open $G$-set $U$ such that 
$r(U)=A$, $s(U)=B$ and $U\subset\{(x,n,y)\in G\}$. 
Furthermore, if $A$ and $B$ are disjoint, then 
there exists $\tau\in\mathcal{T}$ such that 
$S(\tau,-n)=A$, $S(\tau,n)=B$, $\tau(A)=B$ and $\supp(\tau)=A\cup B$. 
\end{lemma}
\begin{proof}
The case $n=0$ is trivial as mentioned above. 
We may assume without loss of generality that $n$ is positive. 
Write $B$ as a disjoint union 
$B=\bigcup_{\mu\in I}C_\mu$ of non-empty cylinder sets. 
For each $\mu\in I$, 
there exists an admissible word $\mu'$ such that 
$t(\mu')=i(\mu)$ and $\lvert\mu'\rvert=n$. 
Then 
\[
\sum_{\mu\in I}[1_{C_{\mu'\mu}}]_K
=\sum_{\mu\in I}\delta^{-n}\left([1_{C_\mu}]_K\right)
=\delta^{-n}([1_B]_K)=[1_A]_K. 
\]
in $H_0(K)$. 
Since $K$ is an AF groupoid, 
we can find compact open $K$-sets $V_\mu$ such that 
$s(V_\mu)=C_{\mu'\mu}$ and 
$A$ equals the disjoint union $\bigcup_\mu r(V_\mu)$. 
Then 
\[
U=\bigcup_{\mu\in I}V_\mu U_{\mu'\mu,\mu}
\]
is a desired compact open $G$-set. 

Assume further that $A$ and $B$ are disjoint. 
Clearly 
\[
V=U\cup U^{-1}\cup(Y\setminus(A\cup B))
\]
is a compact open $G$-set satisfying $r(V)=s(V)=Y$, 
and $\tau=\pi_V$ satisfies the requirement. 
\end{proof}

\begin{lemma}\label{Tgenerate}
Suppose that $(X,\sigma)$ is topologically mixing. 
Then $[[G|Y]]_0$ is generated by the set $\mathcal{T}$ of transpositions. 
\end{lemma}
\begin{proof}
We remark that $(X,\sigma)$ is topologically mixing if and only if 
the matrix $M=(M(\xi,\eta))_{\xi,\eta\in\V}$ is primitive, 
that is, there exists $n\in\N$ such that 
$M^n(\xi,\eta)>0$ for all $\xi,\eta\in\V$ (\cite[Proposition 4.5.10]{LM}). 
Moreover, in this case the AF groupoid $K$ is minimal and 
has a unique $K$-invariant probability measure on $X$, say $\omega$. 
We can define a homomorphism $\hat\omega:H_0(K)\to\R$ by 
\[
\hat\omega([f])=\int f\,d\omega
\]
for $f\in C(X,\Z)$. 
Let $r>1$ be the Perron eigenvalue of $M$. 
Then it is well-known that $\hat\omega\circ\delta=r\hat\omega$ holds. 

Let $\langle\mathcal{T}\rangle$ be the subgroup generated by $\mathcal{T}$. 
Suppose that $\alpha\in[[G|Y]]_0\setminus\{1\}$ is given. 
There exists a non-empty clopen set $C\subset Y$ 
such that $C\cap\alpha(C)=\emptyset$. 
We can construct $\tau\in\mathcal{T}$ 
such that $\tau(x)=\alpha(x)$ for all $x\in C$. 
Hence, by replacing $\alpha$ with $\tau\alpha$, 
we may assume that $\supp(\alpha)$ is not equal to $Y$. 
Let $A=\supp(\alpha)\varsubsetneq Y$. 
Since $K$ is minimal, one has $\omega(Y)>0$ and $\omega(Y\setminus A)>0$. 

By Lemma \ref{index}, we have 
\[
\sum_{n\in\Z}\delta^{(-n)}\left([1_{S(\alpha,n)}]_K\right)=0. 
\]
Set $P=\{n\in\N\mid S(\alpha,n)\neq\emptyset\}$ and 
$Q=\{n\in\Z\mid n<0,\ S(\alpha,n)\neq\emptyset\}$. 
Then 
\begin{equation}
-\sum_{n\in P}\delta^{(-n)}\left([1_{S(\alpha,n)}]_K\right)=
\sum_{n\in Q}\delta^{(-n)}\left([1_{S(\alpha,n)}]_K\right). 
\label{cow}
\end{equation}
Put 
\[
z=[1_A]_K+\delta\left(
\sum_{n\in Q}\delta^{(-n)}\left([1_{S(\alpha,n)}]_K\right)\right)\in H_0(K). 
\]
We have 
\[
\hat\omega(z)=\omega(A)
+\sum_{n\in Q}(r+r^2+\dots+r^{-n})\omega(S(\alpha,n)). 
\]
Choose $m\in\N$ so that 
\[
r^{-m}\omega(A)
<\omega(Y\setminus A)\quad\text{and}\quad r^{-m}\hat\omega(z)<\omega(Y)
\]
hold. 
Then we can find a clopen set $B\subset Y\setminus A$ 
such that $[1_B]_K=\delta^{-m}([1_A]_K)$. 
It follows from Lemma \ref{transp} that 
there exists $\tau_0\in\mathcal{T}$ such that 
$S(\tau_0,m)=A$, $S(\tau_0,-m)=B$, $\tau_0(A)=B$ and $\supp(\tau_0)=A\cup B$. 
Set $\beta=\tau_0\alpha\tau_0$. 
It suffices to show that $\beta$ belongs to $\langle\mathcal{T}\rangle$. 
Notice that $\supp(\beta)=\tau_0(A)=B$, $S(\beta,n)=\tau_0(S(\alpha,n))$ 
and 
\[
[1_{S(\beta,n)}]_K=\delta^{-m}\left([1_{S(\alpha,n)}]_K\right)
\]
for every $n\in\Z$. 
Hence, by \eqref{cow}, we get 
\begin{equation}
-\sum_{n\in P}\delta^{(-n)}\left([1_{S(\beta,n)}]_K\right)=
\sum_{n\in Q}\delta^{(-n)}\left([1_{S(\beta,n)}]_K\right). 
\label{tiger}
\end{equation}

We would like to construct $\gamma\in\langle\mathcal{T}\rangle$ 
such that $S(\gamma,n)=S(\beta,n)$ for all $n\in\Z$. 
Because 
\begin{align*}
&\omega(B)+\sum_{n\in Q}(r+r^2+\dots+r^{-n})\omega(S(\beta,n))\\
&=r^{-m}\left(\omega(A)
+\sum_{n\in Q}(r+r^2+\dots+r^{-n})\omega(S(\alpha,n))\right)\\
&=r^{-m}\hat\omega(z)<\omega(Y), 
\end{align*}
there exist mutually disjoint clopen sets $C_{n,i}\subset Y\setminus B$ 
for $n\in Q$ and $i=1,2,\dots,-n$ such that 
\[
\delta\left([1_{S(\beta,n)}]_K\right)=[1_{C_{n,1}}]_K
\quad\text{and}\quad 
\delta\left([1_{C_{n,i}}]_K\right)=[1_{C_{n,i+1}}]_K
\]
for each $n\in Q$ and $i=1,2,\dots,-n{-}1$. 
By Lemma \ref{transp}, for $n\in Q$ and $i=1,2,\dots,-n$, 
we can find $\tau_{n,i}\in\mathcal{T}$ such that 
\[
S(\tau_{n,1},-1)=S(\beta,n),\quad S(\tau_{n,1},1)=C_{n,1}, 
\]
\[
\tau_{n,1}(S(\beta,n))=C_{n,1},\quad 
\supp(\tau_{n,1})=S(\beta,n)\cup C_{n,1}
\]
and 
\[
S(\tau_{n,i+1},-1)=C_{n,i},\quad S(\tau_{n,i+1},1)=C_{n,i+1}, 
\]
\[
\tau_{n,i+1}(C_{n,i})=C_{n,i+1},\quad 
\supp(\tau_{n,i+1})=C_{n,i}\cup C_{n,i+1}. 
\]
Define $\tau_-\in\langle\mathcal{T}\rangle$ by 
\[
\tau_-=\prod_{n\in Q}\tau_{n,-n}\dots\tau_{n,2}\tau_{n,1}. 
\]
Then we can verify 
\[
S(\tau_-,n)=S(\beta,n)\quad\forall n\in Q,\quad 
S(\tau_-,1)=\bigcup_{n\in Q}\bigcup_{i=1}^{-n}C_{n,i}
\]
and 
\[
\tau_-(S(\beta,n))=C_{n,-n}\quad\forall n\in Q,\quad 
\supp(\tau_-)
=\bigcup_{n\in Q}\left(S(\beta,n)\cup\bigcup_{i=1}^{-n}C_{n,i}\right). 
\]
Next, we would like to consider $S(\beta,n)$ for $n\in P$. 
By using \eqref{tiger}, we have 
\begin{align*}
&\sum_{n\in P}[1_{S(\beta,n)}]_K
+\sum_{n\in P}(\delta^{-1}+\delta^{-2}+\dots+\delta^{-n})
\left([1_{S(\beta,n)}]_K\right)\\
&=\sum_{n\in P}[1_{S(\beta,n)}]_K
-\sum_{n\in P}\delta^{(-n)}\left([1_{S(\beta,n)}]_K\right)\\
&=\sum_{n\in P}[1_{S(\beta,n)}]_K
+\sum_{n\in Q}\delta^{(-n)}\left([1_{S(\beta,n)}]_K\right)\\
&=\sum_{n\in P}[1_{S(\beta,n)}]_K+\sum_{n\in Q}[1_{S(\beta,n)}]_K
+\sum_{n\in Q}(\delta+\delta^2+\dots+\delta^{-n-1})
\left([1_{S(\beta,n)}]_K\right)\\
&=\sum_{n\in P}[1_{S(\beta,n)}]_K+\sum_{n\in Q}[1_{S(\beta,n)}]_K
+\sum_{n\in Q}\sum_{i=1}^{-n-1}[1_{C_{n,i}}]_K. 
\end{align*}
Hence we can find mutually disjoint clopen sets $D_{n,j}$ 
for $n\in P$ and $j=1,2,\dots,n$ such that 
\[
\delta^{-1}\left([1_{S(\beta,n)}]_K\right)=[1_{D_{n,1}}]_K,\quad 
\delta^{-1}\left([1_{D_{n,j}}]_K\right)=[1_{D_{n,j+1}}]_K
\]
and 
\[
\bigcup_{n\in P}\bigcup_{j=1}^nD_{n,j}
=\bigcup_{n\in Q}\left(S(\beta,n)\cup\bigcup_{i=1}^{-n-1}C_{n,i}\right). 
\]
Therefore, in the same way as above, 
we can construct $\tau_+\in\langle\mathcal{T}\rangle$ such that 
\[
S(\tau_+,n)=S(\beta,n)\quad\forall n\in P,\quad 
S(\tau_+,-1)
=\bigcup_{n\in Q}\left(S(\beta,n)\cup\bigcup_{i=1}^{-n-1}C_{n,i}\right)
\]
and 
\[
\tau_+(S(\beta,n))=D_{n,n}\quad\forall n\in P,\quad 
\supp(\tau_+)
=\bigcup_{n\in P}S(\beta,n)\cup
\bigcup_{n\in Q}\left(S(\beta,n)\cup\bigcup_{i=1}^{-n-1}C_{n,i}\right). 
\]
Let $\gamma=\tau_+\tau_-\in\langle\mathcal{T}\rangle$. 
It is not so hard to see 
\[
S(\gamma,n)=S(\beta,n)\quad\forall n\in P\cup Q,\quad 
\bigcup_{n\in P\cup Q}S(\gamma,n)\cup S(\gamma,0)=Y, 
\]
and so $S(\gamma,n)=S(\beta,n)$ for all $n\in\Z$. 
Then we obtain $S(\beta\gamma^{-1},0)=Y$, 
which implies $\beta\gamma^{-1}\in[[K|Y]]$ 
and hence $\beta\gamma^{-1}\in\langle\mathcal{T}\rangle$. 
It follows that $\beta$ belongs to $\langle\mathcal{T}\rangle$, as desired. 
\end{proof}

We are now ready to prove the main theorem of this subsection. 

\begin{theorem}\label{fg}
The group $[[G|Y]]_0$ is generated by finitely many elements of $\mathcal{T}$. 
\end{theorem}
\begin{proof}
Let $p\in\N$ be the period of the shift $(X,\sigma)$ 
or, equivalently, of the irreducible matrix $M$ 
(\cite[Definition 4.5.4]{LM}). 
It follows from \cite[Section 4.5]{LM} that 
there exists a non-empty $\sigma^p$-invariant clopen subset $Z\subset X$ 
such that $(Z,\sigma^p|Z)$ is (conjugate to) 
a topologically mixing shift of finite type. 
Clearly the groupoid $G|Z$ is identified 
with the \'etale groupoid arising from the shift $(Z,\sigma^p|Z)$. 
Since $G$ is purely infinite and minimal, 
there exists a compact open $G$-set $U$ 
satisfying $s(U)=Y$ and $r(U)\subset Z$ (see Lemma \ref{SFT>pim}). 
Hence $G|s(U)=G|Y$ is isomorphic to $G|r(U)$ as an \'etale groupoid. 
The groupoid $G|r(U)$ is equal to the reduction of $G|Z$ to $r(U)$. 
Therefore, by replacing $G$ with $G|Z$, we may assume that 
$G$ arises from a topologically mixing shift of finite type. 

Choose two distinct edges $p,q\in\E$ such that $t(p)=i(q)$. 
It is not so hard to see $C_{pq}\cup C_{q}\neq X$. 
By Proposition \ref{pim} (3), there exists a compact open $G$-set $U$ 
such that $Y=s(U)$ and $C_{pq}\cup C_{q}\subset r(U)$. 
It follows that $G|Y=G|s(U)$ is isomorphic to $G|r(U)$. 
By replacing $Y$ with $r(U)$, 
we may assume that $C_{pq}\cup C_{q}$ is contained in $Y$. 
Let $J$ be the set of admissible words $\mu$ such that $C_\mu\subset Y$. 
Then $pq,q\in J$. 

For a word $\mu=e_1e_2\dots e_k$ and $1\leq i\leq j\leq k$, 
we write $\mu_{[i,j]}=e_ie_{i+1}\dots e_j$. 
The clopen set $Y\subset X$ is written 
as a disjoint union of finitely many cylinder sets, 
and so there exists $l\in\N$ such that 
if $\mu\in J$ and $\lvert\mu\rvert\geq l$, then $\mu_{[1,l]}\in J$. 
Because we have assumed that the adjacency matrix $M$ is primitive, 
there exists $m\in\N$ such that 
every entry of the matrix $M^m$ is greater than $2$. 
We assume $m\geq l$. 

Let $\mu,\nu\in J$ be such that $t(\mu)=t(\nu)$. 
In Section 6.1, we introduced the compact open $G$-set $U_{\mu,\nu}$. 
Suppose that $C_\mu$ and $C_\nu$ are disjoint. 
Then 
\[
V=U_{\mu,\nu}\cup U_{\mu,\nu}^{-1}\cup(Y\setminus(C_\mu\cup C_\nu))
\]
is a compact open $G|Y$-set satisfying $r(V)=s(V)=Y$. 
We write $\gamma(\mu,\nu)=\pi_V$, which is an element of $\mathcal{T}$. 
Let $\mathcal{T}_0$ denote the set of all such elements $\gamma(\mu,\nu)$. 
Since the sets $U_{\mu,\nu}$ form a base for the topology of $G|Y$, 
any element of $\mathcal{T}$ is 
a product of finitely many elements of $\mathcal{T}_0$. 
Note also that $\gamma(\mu,\nu)$ equals $\gamma(\nu,\mu)$. 
Define the finite set $F$ by 
\[
F=\{\gamma(\mu,\nu)\in\mathcal{T}_0\mid
\lvert\mu\rvert=\lvert\nu\rvert\leq m+2\}
\cup\{\gamma(\mu,\nu)\in\mathcal{T}_0\mid
\lvert\mu\rvert+1=\lvert\nu\rvert\leq m+2\}. 
\]
Let $\langle F\rangle$ be the subgroup generated by $F$. 
By Lemma \ref{Tgenerate}, $[[G|Y]]_0$ is generated by $\mathcal{T}$. 
Hence, it suffices to prove that $\langle F\rangle$ contains $\mathcal{T}_0$. 

First, we claim that $\langle F\rangle$ contains 
$\{\gamma(\mu,\nu)\in\mathcal{T}_0\mid\lvert\mu\rvert=\lvert\nu\rvert\}$. 
The proof is by induction on the length of the words. 
Assume that the claim is known for length $n$, where $n\geq m+2$. 
Let $\mu,\nu\in J$ be distinct admissible words 
of length $n{+}1$ such that $t(\mu)=t(\nu)$. 
We would like to show $\gamma(\mu,\nu)\in\langle F\rangle$. 
Since each entry of $M^m$ is not less than $3$, 
we can find an admissible word $\lambda$ of length $n{+}1$ such that 
\[
\lambda_{[1,2]}=pq,\quad t(\lambda)=t(\mu),\quad 
\lambda_{[3,m+2]}\neq\mu_{[3,m+2]},\quad \lambda_{[3,m+2]}\neq\nu_{[3,m+2]}. 
\]
We will show that $\gamma(\mu,\lambda)$ is in $\langle F\rangle$. 
By the choice of $m$, 
there exists an admissible word $a$ of length $n$ such that 
\[
a_{[1,1]}=q,\quad a_{[m+2,n]}=\mu_{[m+3,n+1]},\quad 
a_{[2,m+1]}\neq\lambda_{[3,m+2]},\quad a_{[2,m+1]}\neq\mu_{[3,m+2]}. 
\]
Set $b=\lambda_{[2,n+1]}$. 
Then $a,b\in J$, $\lvert a\rvert=\lvert b\rvert=n$ and $a\neq b$. 
By the induction hypothesis, $\gamma(a,b)$ belongs to $\langle F\rangle$. 
It follows that 
\[
\gamma(q,pq)\gamma(a,b)\gamma(q,pq)
=\gamma(pa,pb)=\gamma(pa,\lambda)
\]
also belongs to $\langle F\rangle$. 
Set $c=a_{[1,n-1]}$ and $d=\mu_{[1,n]}$. 
We have $pc,d\in J$, $\lvert pc\rvert=\lvert d\rvert=n$ and $pc\neq d$. 
By the induction hypothesis, $\gamma(pc,d)$ is in $\langle F\rangle$. 
Also $C_\lambda\cap C_{pc}=C_\lambda\cap C_d=\emptyset$ and 
$C_{pa}\subset C_{pc}$. 
Therefore 
\[
\gamma(pc,d)\gamma(pa,\lambda)\gamma(pc,d)=\gamma(\mu,\lambda)
\]
is in $\langle F\rangle$, too. 
In the same way, we can show $\gamma(\nu,\lambda)\in\langle F\rangle$. 
Then 
\[
\gamma(\mu,\lambda)\gamma(\nu,\lambda)\gamma(\mu,\lambda)=\gamma(\mu,\nu)
\in\langle F\rangle, 
\]
which completes the induction, and the proof of the claim. 

Next, we claim that $\langle F\rangle$ contains $\gamma(\mu,\nu)$ 
for $\mu,\nu$ with $\lvert\mu\rvert+1=\lvert\nu\rvert$. 
The proof is by induction on the length $\lvert\nu\rvert$. 
Assume that the claim is known for length $n$, where $n\geq m+2$. 
Let $\mu,\nu\in J$ be such that $t(\mu)=t(\nu)$, 
$\lvert\mu\rvert=n$, $\lvert\nu\rvert=n{+}1$ and $C_\mu\cap C_\nu=\emptyset$. 
Since each entry of the matrix $M^m$ is greater than $2$, 
we can find an admissible word $a$ of length $n$ such that 
\[
a_{[1,1]}=q,\quad a_{[m+2,n]}=\nu_{[m+3,n+1]},\quad 
a_{[2,m+1]}\neq\mu_{[2,m+1]},\quad a_{[2,m+1]}\neq\nu_{[1,m]}. 
\]
Let $b=a_{[1,n-1]}$ and $c=\nu_{[1,n]}$. 
Then $a,b,c\in J$, $\lvert b\rvert=n{-}1$, $\lvert c\rvert=n$ and 
$C_b\cap C_c=\emptyset$. 
By the induction hypothesis, $\gamma(b,c)$ is in $\langle F\rangle$. 
Note also that $C_\mu\cap C_b=C_\mu\cap C_c=\emptyset$ and $C_a\subset C_b$. 
We have already shown that $\gamma(\mu,a)$ is in $\langle F\rangle$. 
It follows that 
\[
\gamma(b,c)\gamma(\mu,a)\gamma(b,c)=\gamma(\mu,\nu)
\]
is in $\langle F\rangle$, 
which completes the induction, and the proof of the claim. 

Finally, we claim that $\langle F\rangle$ contains $\gamma(\mu,\nu)$ 
for $\mu,\nu$ with $\lvert\mu\rvert<\lvert\nu\rvert$. 
Clearly we may assume $\lvert\mu\rvert\geq m+1$. 
The proof is by induction on $n=\lvert\nu\rvert-\lvert\mu\rvert$. 
The case $n=1$ is done. 
Assume that the claim is known for $n{-}1$. 
Let $\mu,\nu\in J$ be such that $t(\mu)=t(\nu)$, $\lvert\mu\rvert\geq m+1$, 
$\lvert\nu\rvert-\lvert\mu\rvert=n\geq2$ and $C_\mu\cap C_\nu=\emptyset$.
In the same way as the argument above, 
we can choose an admissible word $\lambda$ of length $\lvert\mu\rvert{+}1$ 
such that 
\[
\lambda_{[1,1]}=q,\quad t(\lambda)=t(\mu),\quad 
\lambda_{[2,m+1]}\neq\mu_{[2,m+1]},\quad \lambda_{[2,m+1]}\neq\nu_{[2,m+1]}. 
\]
Thus, $\lambda$ is in $J$, and 
$C_\mu$, $C_\lambda$ and $C_\nu$ are mutually disjoint. 
By the induction hypothesis, 
$\gamma(\mu,\lambda)$ and $\gamma(\nu,\lambda)$ are in $\langle F\rangle$. 
Hence we get 
\[
\gamma(\mu,\lambda)\gamma(\nu,\lambda)\gamma(\mu,\lambda)
=\gamma(\mu,\nu)\in\langle F\rangle, 
\]
thereby proving the theorem. 
\end{proof}

%%%%%%%%%%%%%%%%%%%%%%%%%%%%%%%%%%%%%%%%%%%%%%%%%%%%%%%%%%%%
\subsection{Finiteness properties}

As in Section 6.1, 
we let $G$ be the \'etale groupoid 
arising from a one-sided irreducible shift of finite type $(X,\sigma)$. 
We use the notation of Section 6.1. 
Fix a non-empty clopen subset $Y\subset X$. 
In this subsection, 
we would like to show that $[[G|Y]]$ is of type F$_\infty$. 
As mentioned in Section 1, 
this is equivalent to saying that 
$[[G|Y]]$ is of type FP$_\infty$ and is finitely presented. 

Such finiteness properties for Thompson's group $F_{2,1}$ were studied by 
K. S. Brown and R. Geoghegan \cite{BrG84Invent}. 
Brown \cite{Br87JPAA,Br89MSRI} later extended this to 
the infinite families $F_{n,r}\subset T_{n,r}\subset V_{n,r}$ 
(see Remark \ref{HigThomp}). 
M. Stein \cite{St92TAMS} generalized these families further. 
Here we follow the approach of \cite{Br87JPAA,Br89MSRI,St92TAMS}. 
The following theorem, due to Brown, is the key tool. 
We refer the reader to \cite[Section I.4]{BrGTM} 
for the definition of a $\Gamma$-complex. 

\begin{theorem}[{\cite[Corollary 3.3]{Br87JPAA}}]\label{F_infty}
Let $\Gamma$ be a group. 
Suppose that $\Gamma$ admits a contractible $\Gamma$-complex $Z$ 
such that the stabilizer of every cell is of type F$_\infty$. 
Let $\{Z_q\}_{q\in\N}$ be a filtration of $Z$ 
such that each $Z_q$ is finite mod $\Gamma$. 
Suppose that the connectivity of the pair $(Z_{q+1},Z_q)$ tends to $\infty$ 
as $q$ tends to $\infty$. 
Then $\Gamma$ is of type F$_\infty$. 
\end{theorem}

First, we need to modify the finite directed graph $(\V,\E)$ 
so that its adjacency matrix $M$ has a certain special property. 
Let $N=(N(i,j))_{1\leq i,j\leq n}$ be a matrix with entries in $\N$. 
We say that $N$ is in canonical form 
if the following conditions are satisfied. 
\begin{itemize}
\item For any $i\neq j$, $N(i,j)=1$. 
\item For any $i$, $N(i,i)\geq2$. 
\item There exists $j$ such that $N(j,j)=2$. 
\end{itemize}
Note that $N$ is primitive and $N=N^t$. 
Take $j$ such that $N(j,j)=2$. 
For $i\neq j$, let $d_i=N(i,i)-2$. 
It is not so hard to see 
\[
\Coker(\id-N^t)\cong\bigoplus_{i\neq j}\Z_{d_i}\quad\text{and}\quad 
\det(\id-N^t)=(-1)^n\prod_{i\neq j}d_i, 
\]
where $\Z_1$ is understood as $0$ and $\Z_0$ is understood as $\Z$. 
Hence, for the given matrix $M$, 
we can construct a matrix $N$ in canonical form satisfying 
\[
\Coker(\id-M^t)\cong\Coker(\id-N^t)\quad\text{and}\quad 
\det(\id-M^t)=\det(\id-N^t). 
\]
Let $(\tilde X,\tilde\sigma)$ be 
the one-sided shift of finite type arising from the matrix $N$ and 
let $\tilde G$ be the \'etale groupoid of $(\tilde X,\tilde\sigma)$. 
Then there exists an isomorphism $\phi:H_0(G)\to H_0(\tilde G)$. 
By Lemma \ref{pi>clopen}, 
there exists a non-empty clopen set $\tilde Y\subset\tilde X$ 
such that $[1_{\tilde Y}]_{\tilde G}=\phi([1_Y]_G)$. 
It follows from Theorem \ref{classify} that 
$G|Y$ is isomorphic to $\tilde G|\tilde Y$. 
Therefore, by replacing $G|Y$ with $\tilde G|\tilde Y$, we may assume that 
the adjacency matrix $M$ of the graph $(\V,\E)$ is in canonical form. 
Set $\V_0=\{\zeta\in\V\mid M^t(\zeta,\zeta)=2\}$. 
Then 
\begin{equation}
a\in\Ker(\id-M^t)\iff 
\sum_{\zeta\in\V_0}a(\zeta)=0\quad\text{and}\quad 
a(\xi)=0\quad\forall \xi\in\V\setminus\V_0. 
\label{rabbit}
\end{equation}
This property will play an important role in the argument below. 

We write $\Z_+=\N\cup\{0\}$. 
For an admissible word $\mu$, we define a compact open $G$-set $U_\mu$ by 
\[
U_\mu=\{(x,\lvert\mu\rvert,y)\in G\mid
\sigma^{\lvert\mu\rvert}(x)=y,\ x\in C_\mu\}. 
\]
In other words, $U_\mu$ equals $U_{\mu,\emptyset}$, 
where $\emptyset$ denotes the empty word of length zero. 
For $\xi\in\V$, we let $i^{-1}(\xi)=\{e\in\E\mid i(e)=\xi\}$. 
Define $\hat\xi\in\Z_+^\V$ by 
\[
\hat\xi(\eta)=\begin{cases}1&\eta=\xi\\0&\eta\neq\xi. \end{cases}
\]
Recall that $D_\xi=\{(x_n)_n\in X\mid i(x_1)=\xi\}$. 

Now we would like to construct a poset (partially ordered set) $\B$. 
Let $\B_0$ denote the set of all compact open $G$-sets $U$ 
such that $r(U)\subset Y$ and $s(U)=D_\xi$ for some $\xi\in\V$. 
For a finite subset $u$ of $\B_0$, 
we define $\rank(u)\in\Z_+^\V$ by 
\[
\rank(u)(\xi)=\#\{U\in u\mid s(U)=D_\xi\}\quad\forall\xi\in\V. 
\]
We let $\B$ denote 
the set of all finite subsets $u$ of $\B$ 
such that $\{r(U)\mid U\in u\}$ is a clopen partition of $Y$. 
We equip $\B$ with a partial order and 
view it as a poset in the following way. 
Given $u\in\B$ and an element $U\in u$ with $s(U)=D_\xi$, 
we get $v\in\B$ 
by replacing $U$ with the sets $UU_e$ 
for $e\in\E$ with $i(e)=\xi$. 
Thus, 
\[
v=(u\setminus\{U\})\cup\{UU_e\in\B_0\mid e\in i^{-1}(\xi)\}. 
\]
One says that $v$ is a simple expansion of $u$. 
Notice that $\rank(v)$ is equal to $\rank(u)+(M^t-\id)\hat\xi$. 
Conversely, if we start with $v\in\B$, $\xi\in\V$ and 
an injection $f:i^{-1}(\xi)\to v$ such that $s(f(e))=D_{t(e)}$, 
then we obtain $u\in\B$ by replacing the elements $f(e)$ with 
\[
\bigcup_{i(e)=\xi}f(e)U_e^{-1}\in\B_0. 
\]
The element $u$ is called a simple contraction of $v$. 
Again we have $\rank(u)=\rank(v)+(\id-M^t)\hat\xi$. 
Given two elements $u,v\in\B$, 
$v$ is a simple expansion of $u$ if and only if 
$u$ is a simple contraction of $v$. 
We now iterate these construction. 
We say that $v\in\B$ is an expansion of $u\in\B$ 
and that $u$ is a contraction of $v$ 
if there exists a sequence $u{=}u_0,u_1,\dots,u_d{=}v$ ($d\geq0$) 
with $u_i$ a simple expansion of $u_{i-1}$ for $i=1,2,\dots,d$. 
Then we equip $\B$ with a partial order 
by saying $u\leq v$ if $v$ is an expansion of $u$. 

\begin{lemma}
The poset $\B$ is directed, 
i.e. for any $u,v\in\B$ there exists $w\in\B$ 
such that $u\leq w$ and $v\leq w$. 
\end{lemma}
\begin{proof}
Let $u\in\B$. 
We first claim that there exists an expansion $w$ of $u$ such that 
any $W\in w$ is of the form $U_\mu$ for some admissible word $\mu$. 
Take $U\in u$ which is not of the form $U_\mu$. 
Let $s(U)=D_\xi$. 
There exist admissible words 
$\mu_1,\mu_2,\dots,\mu_k$, $\nu_1,\nu_2,\dots,\nu_k$ such that 
$U$ equals the disjoint union $\bigcup_{i=1}^kU_{\mu_i,\nu_i}$. 
Consider a simple expansion 
\[
u'=(u\setminus\{U\})\cup\{UU_e\in\B_0\mid e\in i^{-1}(\xi)\}
\]
of $u$. 
Then for any $e\in i^{-1}(\xi)$, 
\[
UU_e=\bigcup U_{\mu_i,\tilde\nu_i}, 
\]
where the union runs over all $i$ such that 
the first letter of $\nu_i$ is $e$, and 
$\tilde\nu_i$ is the (possibly empty) word 
obtained by removing the first letter of $\nu_i$. 
In particular, $\lvert\tilde\nu_i\rvert=\lvert\nu_i\rvert-1$. 
By repeating this argument, 
we finally obtain an expansion $w$ of $u$ 
for which every $\nu_i$ is of length zero. 
The proof of the claim is complete. 

By the claim above, we may assume that there exist admissible words 
$\mu_1,\mu_2,\dots,\mu_k$, $\nu_1,\nu_2,\dots,\nu_l$ such that 
\[
u=\{U_{\mu_i}\mid i=1,2,\dots,k\}\quad\text{and}\quad 
v=\{U_{\nu_j}\mid j=1,2,\dots,l\}. 
\]
Let $m=\max\{\lvert\mu_i\rvert,\lvert\nu_j\rvert
\mid i=1,\dots,k,\ j=1,\dots,l\}$. 
Suppose that $\lvert\mu_i\rvert$ is less than $m$. 
Then 
\[
u'=(u\setminus\{U_{\mu_i}\})
\cup\{U_{\mu_ie}\in\B_0\mid e\in\E,\ t(\mu_i)=i(e)\}
\]
is a simple expansion of $u$ and $\lvert\mu_ie\rvert=\lvert\mu_i\rvert+1$. 
By repeating this argument, we can find an expansion $\tilde u$ of $u$ 
such that any element of $\tilde u$ is 
of the form $U_\mu$ with $\lvert\mu\rvert=m$. 
In the same way, there exists an expansion $\tilde v$ of $v$ 
such that any element of $\tilde v$ is 
of the form $U_\mu$ with $\lvert\mu\rvert=m$. 
This means that $\tilde u$ equals $\tilde v$, 
which completes the proof. 
\end{proof}

We denote by $\lvert\B\rvert$ the simplicial complex 
whose simplices are the non-empty finite linearly ordered subsets of $\B$. 
For a subset $\mathcal{C}\subset\B$, 
$\lvert\mathcal{C}\rvert$ denotes 
the subcomplex of $\lvert\B\rvert$ spanned by $\mathcal{C}$. 
We identify 
the set of vertices of $\lvert\B\rvert$ with $\B$. 

We now define an action of $[[G|Y]]$ on $\lvert\B\rvert$. 
For $\alpha\in[[G|Y]]$, 
take a compact open $G|Y$-set $W$ such that $\alpha=\pi_W$. 
Note that $r(W)=Y$ and $s(W)=Y$. 
For $u\in\B$, we define 
\[
\alpha u=\{WU\in\B_0\mid U\in u\}. 
\]
This gives a well-defined action of $[[G|Y]]$ on $\B$. 
It is easy to see that $u\mapsto\alpha u$ is an order preserving map. 
Hence this action gives rise to 
an action of $[[G|Y]]$ on the simplicial complex $\lvert\B\rvert$. 

\begin{lemma}\label{stabilizer}
\begin{enumerate}
\item For $u,v\in\B$, 
there exists $\alpha\in[[G|Y]]$ such that $\alpha u=v$ 
if and only if $\rank(u)=\rank(v)$. 
\item Every vertex stabilizer is a finite group. 
More precisely, for any $u\in\B$, 
$\{\alpha\in[[G|Y]]\mid\alpha u=u\}$ is isomorphic to 
$\bigoplus_{\xi\in\V}\Sigma_{\rank(u)(\xi)}$, 
where $\Sigma_n$ stands for the symmetric group on $n$ letters. 
\end{enumerate}
\end{lemma}
\begin{proof}
(1)
The `only if' part is obvious. 
To prove the `if ' part, assume $\rank(u)=\rank(v)$. 
There exists a bijection $f:u\to v$ 
such that $s(f(U))=s(U)$ for every $U\in u$. 
Then 
\[
W=\bigcup_{U\in u}f(U)U^{-1}
\]
is a compact open $G|Y$-set satisfying $r(W)=s(W)=Y$ and $\pi_Wu=v$. 

(2)
Assume $\alpha u=u$. 
Take a compact open $G|Y$-set $W$ such that $\alpha=\pi_W$. 
The homeomorphism $\alpha$ induces a permutation of the finite set $u$. 
Suppose that the induced permutation is the identity. 
For every $U\in u$, we have $WU=U$, and so 
\[
W=WY=\bigcup_{U\in u}WUU^{-1}=\bigcup_{U\in u}UU^{-1}=Y. 
\]
Thus $\alpha$ is the identity. 
It is clear that the induced permutations form a group 
isomorphic to $\bigoplus_\xi\Sigma_{\rank(u)(\xi)}$. 
\end{proof}

Let $\omega\in\B$. 
By replacing $\omega$ with its any simple expansion, 
we may assume that $\rank(\omega)(\xi)$ is positive for any $\xi\in\V$, 
because each entry of $M^t-\id$ is positive. 
We fix such $\omega\in\B$ and set 
\[
\B_\omega=\{u\in\B\mid
\alpha\omega\leq u\text{ for some }\alpha\in[[G|Y]]\}. 
\]
Since each entry of $\rank(\omega)$ and $M^t-\id$ is positive, 
from Lemma \ref{stabilizer} (1) one can see that $B_\omega$ is equal to 
\[
\left\{u\in\B\mid
\rank(u)=\rank(\omega)+(M^t-\id)a\text{ for some }a\in\Z_+^\V\right\}. 
\]
As $\B$ is directed, so is $\B_\omega$. 
This readily implies the following 
(see \cite[Proposition 9.3.14]{Geotext} for instance). 

\begin{lemma}\label{Bcontract}
The simplicial complex $\lvert\B_\omega\rvert$ is contractible. 
\end{lemma}

For $u\in\B_\omega$, choose $a\in\Z_+^\V$ so that 
\[
\rank(u)=\rank(\omega)+(M^t-\id)a. 
\]
Put 
\[
\height(u)=\sum_{\xi\in\V}a(\xi). 
\]
Thanks to \eqref{rabbit}, 
$\height(u)$ does not depend on the choice of $a$. 
If $v$ is a simple expansion of $u\in\B_\omega$, 
then $\height(v)=\height(u)+1$. 
Clearly $\height(u)$ equals the largest integer $n$ such that 
there exists a linearly ordered $(n{+}1)$-tuple 
$u_0<u_1<\dots<u_n=u$ in $\B_\omega$. 
Because each entry of $M^t-\id$ is positive, 
one has $\rank(u)(\xi)\geq\height(u)$ for every $\xi\in\V$. 
We also note that the $[[G|Y]]$-action preserves $\height$, 
i.e. $\height(\alpha u)=\height(u)$. 

Next, we would like to construct 
a contractible $[[G|Y]]$-invariant subcomplex 
$Z\subset\lvert\B_\omega\rvert$. 
Given an element $u\in\B$, an elementary expansion of $u$ is 
an element $v\in\B$ obtained by choosing a subset $u'\subset u$ 
and taking simple expansions of $u$ at each of $U\in u'$, that is, 
\[
v=(u\setminus u')
\cup\{UU_e\in\B_0
\mid U\in u',\ s(U)=D_\xi,\ e\in i^{-1}(\xi)\}. 
\]
We call a simplex $u_0<u_1<\dots<u_n$ of $\lvert\B_\omega\rvert$ elementary 
if $u_n$ is an elementary expansion of $u_0$. 
This implies that $u_j$ is an elementary expansion of $u_i$ for any $i\leq j$. 
Hence any face of an elementary simplex is elementary. 
Clearly elementary simplices are preserved by the action of $[[G|Y]]$. 
It follows that the union $Z$ of all elementary simplices is 
a $[[G|Y]]$-invariant subcomplex of $\lvert\B_\omega\rvert$. 

In what follows, we will use the standard notation for intervals in a poset. 
For example, the open interval $(u,v)$ is defined 
by $(u,v)=\{w\in\B_\omega\mid u<w<v\}$. 
The closed interval $[u,v]$ and 
the half-open intervals $[u,v)$ and $(u,v]$ are defined similarly. 
The following is exactly the same as the lemma of \cite[Section 4]{Br89MSRI}. 

\begin{lemma}
Let $v\in\B_\omega$ be an expansion of $u\in\B_\omega$. 
If $v$ is not an elementary expansion of $u$, 
then the subcomplex $\lvert(u,v)\rvert$ is contractible. 
\end{lemma}

The following is exactly the same as \cite[Theorem 1]{Br89MSRI}. 
We include the argument here for completeness. 

\begin{lemma}\label{Zcontract}
The complex $Z$ is contractible. 
\end{lemma}
\begin{proof}
Let $W_n$ be the complex obtained by adjoining to $Z$ the subcomplexes 
of the form $\lvert[u,v]\rvert$ 
with $v$ a non-elementary expansion of $u\in\B_\omega$ and 
$\height(v)-\height(u)\leq n$. 
Since $\lvert\B_\omega\rvert$ is contractible by Lemma \ref{Bcontract} and 
$\lvert\B_\omega\rvert$ equals the union $\bigcup_n W_n$, 
it suffices to show that 
the inclusion $Z\hookrightarrow W_n$ is a homotopy equivalence. 
Suppose by induction that 
$Z\hookrightarrow W_{n-1}$ is a homotopy equivalence. 
Let $v$ be a non-elementary expansion of $u\in\B_\omega$ 
such that $\height(v)-\height(u)=n$. 
The intersection of $\lvert[u,v]\rvert$ with $W_{n-1}$ is 
$\lvert[u,v)\cup(u,v]\rvert$, which is the suspension of $\lvert(u,v)\rvert$ 
and hence is contractible by the lemma above. 
Since $\lvert[u,v]\rvert$ is also contractible for trivial reasons, 
the adjunction of $\lvert[u,v]\rvert$ does not change 
the homotopy type of $W_{n-1}$. 
Two such subcomplexes $\lvert[u,v]\rvert$ intersect only in $W_{n-1}$, 
and so the inclusion $W_{n-1}\hookrightarrow W_n$ is a homotopy equivalence. 
The proof is complete. 
\end{proof}

We remark that the argument above goes through without change 
if we fix $p\in\Z_+$ and replace $\B_\omega$ 
with $\{u\in\B_\omega\mid\height(u)\geq p\}$. 
Therefore the subcomplex of $Z$ 
spanned by the vertices $u$ with $\height(u)\geq p$ is also contractible. 

For every $u\in\B_\omega$, 
we introduce the (finite) simplicial complex $K_u$ as follows. 
Consider a pair $(\xi,f)$ of $\xi\in\V$ and an injection $f:i^{-1}(\xi)\to u$ 
such that $s(f(e))=D_{t(e)}$. 
A simplex of $K_u$ is 
a finite collection $\{(\xi_0,f_0),\dots,(\xi_m,f_m)\}$ of such pairs 
satisfying the following conditions. 
\begin{itemize}
\item The images of the maps $f_j$ are pairwise disjoint. 
\item There exists $a\in\Z_+^\V$ such that 
\[
\rank(u)-(M^t-\id)(\hat\xi_0+\dots+\hat\xi_m)=\rank(\omega)+(M^t-\id)a. 
\]
\end{itemize}
The first condition says that 
by contracting $\Ima f_j$ we can obtain $v$ 
such that $u$ is an elementary expansion of $v$. 
The second condition says that $v$ is still in $\B_\omega$, 
because the left-hand side of the equality is the rank of $v$. 

\begin{lemma}\label{Kuconnected}
For any integer $n\geq0$, there exists $r\in\N$ such that 
the simplicial complex $K_u$ is $n$-connected provided $\height(u)>r$. 
\end{lemma}
\begin{proof}
Let $l=\max\{M(\xi,\eta)\mid\xi,\eta\in\V\}$. 
We claim that for any $m\in\Z_+$ and $k\in\N$ the following holds: 
if 
\[
\height(u)>\max\{(m{+}1)\#\V,\ k(m{+}1)l+l\}
\]
and 
$P_1,\dots,P_k$ are $m$-simplices of $K_u$, then 
we can find a vertex $(\eta,g)$ of $K_u$ such that 
$P_j\cup\{(\eta,g)\}$ is an $(m{+}1)$-simplex of $K_u$ for any $j$. 
Choose $a\in\Z_+^\V$ so that $\rank(u)=\rank(\omega)+(M^t-\id)a$. 
From 
\[
\sum_{\xi\in\V}a(\xi)=\height(u)>(m{+}1)\#\V, 
\]
we can find $\eta\in\V$ such that $a(\eta)>m{+}1$. 
Let 
\[
u'=\bigcup_{j=1}^k\bigcup_{(\xi,f)\in P_j}\Ima f. 
\]
Then for any $\xi\in\V$ we have 
\[
\#\{U\in u\setminus u'\mid s(U)=D_\xi\}
\geq\rank(u)(\xi)-k(m{+}1)l\geq\height(u)-k(m{+}1)l>l. 
\]
Hence we can find an injection $g:i^{-1}(\eta)\to u$ 
such that $s(g(e))=D_{t(e)}$ for all $e\in i^{-1}(\eta)$ and 
$u'\cap\Ima g=\emptyset$. 
Now for each $j=1,2,\dots,k$, let 
\[
b_j=a-\sum_{(\xi,f)\in P_j}\hat\xi
\]
so that 
\[
\rank(u)-\sum_{(\xi,f)\in P_j}(M^t-\id)\hat\xi
=\rank(\omega)+(M^t-\id)b_j. 
\]
Since $P_j$ is a simplex of $K_u$, by \eqref{rabbit}, one has 
\[
\sum_{\zeta\in\V_0}b_j(\zeta)\geq0\quad\text{and}\quad 
b_j(\xi)\geq0\quad\forall\xi\in\V\setminus\V_0. 
\]
From the choice of $\eta$, we can verify 
\[
\sum_{\zeta\in\V_0}(b_j-\hat\eta)(\zeta)\geq0\quad\text{and}\quad 
(b_j-\hat\eta)(\xi)\geq0\quad\forall\xi\in\V\setminus\V_0. 
\]
From \eqref{rabbit}, we can conclude that 
$P_j\cup\{(\eta,g)\}$ is an $(m{+}1)$-simplex of $K_u$ for any $j$. 

Once this claim is established, the rest of the proof proceeds 
in the same manner as \cite[Lemma 4.20]{Br89MSRI}. 
We omit the details. 
\end{proof}

Notice that the natural number $r$ obtained in the lemma above 
depends only on the adjacency matrix $M$ and 
does not depend on the choice of $\omega\in\B$. 

Given integers $p,q$ with $0\leq p\leq q$, 
we let $Z_{p,q}$ be the subcomplex of $Z$ 
spanned by the vertices $u$ with $p\leq\height(u)\leq q$. 
Note that $Z_{p,q}$ is $[[G|Y]]$-invariant. 
Note also that the dimension of $Z_{p,q}$ is at most $q-p$. 
The following is almost the same as \cite[Theorem 2]{Br89MSRI}. 
We include the proof for completeness. 

\begin{lemma}\label{Zpqconnected}
For any integer $n\geq0$, there exists $r\in\N$ such that 
the inclusion $Z_{p,q}\hookrightarrow Z_{p,q+1}$ induces 
isomorphisms $\pi_i(Z_{p,q})\to\pi_i(Z_{p,q+1})$ for all $i\leq n$, 
provided $q\geq r$ and $q\geq p+n+1$. 
In particular, for such $p,q$. $Z_{p,q}$ is $n$-connected. 
\end{lemma}
\begin{proof}
For given $n\geq0$, let $r\in\N$ be as in the lemma above. 
The inclusion $Z_{p,q}\hookrightarrow Z_{p,q+1}$ is obtained 
by adjoining, for each $u\in\B_\omega$ with $\height(u)=q+1$, 
a cone over the link $L$ of $u$ in $Z_{p,q}$. 
The vertices of this link $L$ are $v\in\B_\omega$ 
such that $u$ is an elementary expansion of $v$ 
and $p\leq\height(v)\leq q$, 
and the simplices of $L$ are linearly ordered tuples $v_0<v_1<\dots<v_m$. 

We can describe a vertex $v$ of $L$ 
by specifying which elements of $u$ are contracted to get $v$. 
More precisely, each vertex $v$ of $L$ corresponds to 
a set $P=\{(\xi_0,f_0),\dots,(\xi_m,f_m)\}$ consisting of 
pairs of $\xi_j\in\V$ and injections $f_j:i^{-1}(\xi_j)\to u$ 
such that the following conditions are satisfied. 
\begin{itemize}
\item $s(f_j(e))=D_{t(e)}$ for every $e\in i^{-1}(\xi_j)$. 
\item $\Ima f_j$ are pairwise disjoint. 
\item There exists $a\in\Z_+^\V$ such that 
\[
\rank(u)-(M^t-\id)(\hat\xi_0+\dots+\hat\xi_m)=\rank(\omega)+(M^t-\id)a. 
\]
\item $q-m\geq p$. 
\end{itemize}
The simplices of $L$ correspond to 
chains $P_0\subset\dots\subset P_l$ of such sets $P$. 
Now such $P$ are the same as simplices of the complex $K_u$ defined above 
that have dimension at most $q-p$. 
It follows that the complex $L$ is 
the barycentric subdivision of the $(q{-}p)$-skeleton of $K_u$. 
The lemma above implies that $K_u$ is $n$-connected, 
because $\height(u)=q+1>r$. 
Hence $L$ is $n$-connected, too. 
Attaching a cone over such a complex $L\subset Z_{p,q}$ does not affect 
$\pi_i$ for $i\leq n$. 
Consequently, 
$\pi_i(Z_{p,q})\to\pi_i(Z_{p,q+1})$ is an isomorphism as required. 

Consider the sequence of inclusions 
$Z_{p,q}\subset Z_{p,q+1}\subset Z_{p,q+2}\subset\dots$. 
The union is the subcomplex of $Z$ 
spanned by the vertices $u$ with $\height(u)\geq p$. 
This union is contractible 
by Lemma \ref{Zcontract} and the remark following it. 
Therefore $Z_{p,q}$ is $n$-connected. 
\end{proof}

\begin{lemma}\label{finitemodG}
For any integers $p,q$ with $0\leq p\leq q$, 
the simplicial complex $Z_{p,q}$ is finite mod $[[G|Y]]$. 
\end{lemma}
\begin{proof}
Let $u_0<u_1<\dots<u_m$ be an elementary simplex. 
For every $j=1,2,\dots,m$, 
let $\tilde u_j\subset u_0$ be the subset such that 
$u_j$ is obtained by taking simple expansions of $u_0$ 
at each of $U\in u_0\setminus\tilde u_j$. 
Then we have 
\[
u_0\varsupsetneq\tilde u_1\varsupsetneq\tilde u_2\varsupsetneq
\dots\varsupsetneq\tilde u_m. 
\]
Let $v_0<v_1<\dots<v_m$ be another elementary simplex. 
Define $\tilde v_1,\tilde v_2,\dots,\tilde v_m$ in the same way. 
We claim that if $\rank(u_0)=\rank(v_0)$ and 
$\rank(\tilde u_j)=\rank(\tilde v_j)$ for every $j$, 
then there exists $\alpha\in[[G|Y]]$ such that 
$\alpha u_j=v_j$ for all $j=0,1,\dots,m$. 
From $\rank(\tilde u_m)=\rank(\tilde v_m)$, 
we can find a bijection $f:\tilde u_m\to\tilde v_m$ 
satisfying $s(U)=s(f(U))$ for all $U\in\tilde u_m$. 
Using $\rank(\tilde u_m\setminus\tilde u_{m-1})
=\rank(\tilde v_m\setminus\tilde v_{m-1})$, 
we can extend $f$ to a bijection from $\tilde u_{m-1}$ to $\tilde v_{m-1}$ 
satisfying $s(U)=s(f(U))$ for all $U\in\tilde u_{m-1}$. 
Repeating this argument, 
we finally obtain a bijection $f:u_0\to v_0$ such that 
$f(\tilde u_j)=\tilde v_j$ for any $j$ and $s(U)=s(f(U))$ for any $U\in u_0$. 
Define a compact open $G|Y$-set $W$ by 
\[
W=\bigcup_{U\in u_0}f(U)U^{-1}. 
\]
Then $\alpha=\pi_W\in[[G|Y]]$ satisfies 
$\alpha u_j=v_j$ for all $j=0,1,\dots,m$. 

If $u_0<u_1<\dots<u_m$ is a simplex of $Z_{p,q}$, 
then $\height(u_0)$ is at most $q-m$. 
Evidently $\{\rank(u)\in\Z^\V\mid\height(u)\leq q{-}m\}$ is finite. 
Hence the number of $(m{+}1)$-tuples 
$(\rank(u_0),\rank(\tilde u_1),\dots,\rank(\tilde u_m))$ 
for $m$-simplices of $Z_{p,q}$ is finite. 
It follows from the claim above that 
$m$-simplices of $Z_{p,q}$ is finite mod $[[G|Y]]$. 
\end{proof}

We are now in a position to apply Brown's theorem. 

\begin{theorem}\label{finite}
The group $[[G|Y]]$ is of type F$_\infty$. 
In other words, 
the group $[[G|Y]]$ is finitely presented and is of type FP$_\infty$. 
\end{theorem}
\begin{proof}
This immediately follows from 
Lemma \ref{stabilizer} (2), Lemma \ref{Zpqconnected}, Lemma \ref{finitemodG} 
and Theorem \ref{F_infty}. 
\end{proof}

%%%%%%%%%%%%%%%%%%%%%%%%%%%%%%%%%%%%%%%%%%%%%%%%%%%%%%%%%%%%
\subsection{A presentation for $[[G|Y]]$}

As in Section 6.1, 
we let $G$ be the \'etale groupoid 
arising from a one-sided irreducible shift of finite type $(X,\sigma)$. 
We use the notation of Section 6.1. 
Fix a non-empty clopen subset $Y\subset X$. 
In this subsection, 
we would like to describe generators and relations for $[[G|Y]]$. 
This enables us to calculate the abelianization of $[[G|Y]]$ 
(Corollary \ref{abel}). 
Consequently, 
the commutator subgroup $D([[G|Y]])$ is shown to be type F$_\infty$ 
when $H_0(G)$ is a finite abelian group (Corollary \ref{Dfinite}). 

As in the last subsection, we assume that 
the adjacency matrix $M$ is in canonical form, 
and put $\V_0=\{\eta\in\V\mid M^t(\eta,\eta)=2\}$. 
Fix an element $\zeta\in\V_0$. 
Set $\V_1=(\V\setminus\V_0)\cup\{\zeta\}$. 
Notice that the rank of the free abelian group $H_1(G)$ is equal to 
$\#(\V\setminus\V_1)=\#(\V_0\setminus\{\zeta\})$. 
We equip $\V\setminus\V_1$ with a linear order. 
We regard $\Z^{\V_1}$ as a subgroup of $\Z^\V$. 
By \eqref{rabbit}, $\id-M^t$ is injective on $\Z^{\V_1}$, 
i.e. for $a\in\Z^{\V_1}$, $(\id-M^t)a=0$ implies $a=0$. 
For $a\in\Z_+^\V$, we define $\lVert a\rVert=\sum_{\xi\in\V}a(\xi)$. 

In the previous subsection, 
we constructed contractible $[[G|Y]]$-invariant complexes 
$Z\subset\lvert\B_\omega\rvert$. 
For $p<q$, $Z_{p,q}$ denotes the subcomplex of $Z$ 
spanned by vertices $u$ with $p\leq\height(u)\leq q$. 
In Lemma \ref{Zpqconnected}, we proved that $Z_{p,q}$ is $n$-connected 
if $p$ is large enough and $q-p>n$. 
In particular, there exists $p\in\N$ such that 
the $2$-dimensional simplicial complex $Z_{p,p+2}$ is $1$-connected. 
We would like to examine the $[[G|Y]]$-action on $Z_{p,p+2}$ and 
write down generators and relations of $[[G|Y]]$, 
by using \cite[Theorem 1]{Br84JPAA}. 

As pointed out in the remark following Lemma \ref{Kuconnected}, 
the natural number $p$ depends only on the matrix $M$ and 
does not depend on the choice of $\omega\in\Z_+^\V$. 
Hence, by replacing $\omega$ with its certain expansion if necessary, 
we may assume $\rank(\omega)(\xi)\geq p+2$ for any $\xi\in\V$. 
Let 
\[
f:\V\times\{1,2,\dots,p{+}2\}\to\omega
\]
be an injection satisfying $s(f(\xi,j))=D_\xi$ for every $(\xi,j)$. 
For each $a\in\Z_+^\V$ with $\lVert a\rVert\leq p+2$, we let 
\[
u=\{f(\xi,j)\in\omega\mid\xi\in\V,\ 1\leq j\leq a(\xi)\}
\]
and define a vertex $\omega(a)$ of $Z$ by 
\[
\omega(a)=(\omega\setminus u)
\cup\{UU_e\mid U\in u,\ s(U)=D_\xi,\ e\in i^{-1}(\xi)\}. 
\]
Thus, $\omega(a)$ is the elementary expansion of $\omega$ 
with respect to the subset $u\subset\omega$. 
Clearly $\rank(\omega(a))=\rank(\omega)+(M^t-\id)a$, 
and $\height(\omega(a))=\lVert a\rVert$. 
Note also that $\omega(a')$ is an elementary expansion of $\omega(a)$ 
if $a(\xi)\leq a'(\xi)$ for every $\xi\in\V$. 

For every $\eta\in\V\setminus\V_1$, 
we associate an element $\theta_\eta\in[[G|Y]]$ as follows. 
First, for each $\eta\in\V_0$, 
we choose an edge $e_\eta\in\E$ such that $i(e_\eta)=t(e_\eta)=\eta$. 
Let $\eta\in\V\setminus\V_1$. 
Since the matrix $M$ is in canonical form, 
there exists a unique bijection 
\[
h:i^{-1}(\zeta)\setminus\{e_\zeta\}\to i^{-1}(\eta)\setminus\{e_\eta\}
\]
satisfying $t(h(e))=t(e)$. 
Let $U=f(\zeta,1)$ and $V=f(\eta,1)$. 
Define a compact open $G|Y$-set $W$ by 
\[
W=U(UU_{e_\zeta})^{-1}\cup(VU_{e_\eta})V^{-1}\cup\bigcup_eVU_{h(e),e}U^{-1}. 
\]
One has $r(W)=s(W)=r(U)\cup r(V)$. 
Let $\tilde W=W\cup(Y\setminus r(W))$ and 
define $\theta_\eta\in[[G|Y]]$ by $\theta_\eta=\pi_{\tilde W}$. 
Then we have $\theta_\eta\omega(\hat\zeta)=\omega(\hat\eta)$. 
For every $\eta\in\V_0$ and $1\leq k\leq p{+}2$, 
we define a compact open $G|Y$-set $U$ by 
\[
U=f(\eta,1)f(\eta,k)^{-1}\cup f(\eta,k)f(\eta,1)^{-1}
\]
and put $\beta_{\eta,k}=\pi_{\tilde U}\in[[G|Y]]$, 
where $\tilde U=U\cup(Y\setminus r(U))$. 
It is easy to see that $\beta_{\eta,k}\omega=\omega$ and $\beta_{\eta,k}^2=1$. 
Moreover, for $\eta\in\V\setminus\V_1$ we have 
\[
(\beta_{\zeta,k}\beta_{\eta,l}\theta_\eta\beta_{\eta,l}\beta_{\zeta,k})
\omega(a+\hat\zeta)=\omega(a+\hat\eta)
\]
if $a(\zeta)=k-1$ and $a(\eta)=l-1$. 

For each $i=0,1,2$, we would like to find 
a set of representatives for the $i$-simplices of $Z_{p,p+2}$ mod $[[G|Y]]$. 
For $p\leq k\leq p{+}2$, set 
\[
A_k=\left\{a\in\Z_+^{\V_1}\mid\lVert a\rVert=k\right\}. 
\]
Then, by Lemma \ref{stabilizer} (1), 
\[
C_0=\{\omega(a)\in Z_{p,p+2}\mid a\in A_p\cup A_{p+1}\cup A_{p+2}\}
\]
is a set of representatives for the vertices of $Z_{p,p+2}$ mod $[[G|Y]]$, 
that is, for any vertex $u$ of $Z_{p,p+2}$ 
there exists a unique vertex $v\in C_0$ 
such that $\alpha u=v$ for some $\alpha\in[[G|Y]]$. 

We let $C_1$ be the set of all $1$-simplices of 
the simplicial complex $\lvert C_0\rvert$. 
In other words, 
\begin{align*}
C_1=&\{\omega(a)<\omega(a+\hat\xi)\mid a\in A_p\cup A_{p+1},\ \xi\in\V_1\}\\
&\cup\{\omega(a)<\omega(a+\hat\xi_1+\hat\xi_2)
\mid a\in A_p,\ \xi_1,\xi_2\in\V_1\}. 
\end{align*}
In order to obtain a set of representatives 
for the $1$-simplices mod $[[G|Y]]$, we consider the following $1$-simplices. 
\begin{enumerate}
\item[(i)] $\omega(a)<\omega(a+\hat\eta)$ 
for $a\in A_p$ and $\eta\in\V\setminus\V_1$. 
\item[(ii)] $\omega(a)<\omega(a+\hat\eta)$ 
for $a\in A_{p+1}$ and $\eta\in\V\setminus\V_1$. 
\item[(iii)] $\omega(a)<\omega(a+\hat\xi+\hat\eta)$ 
for $a\in A_p$, $\xi\in\V_1$ and $\eta\in\V\setminus\V_1$. 
\item[(iv)] $\omega(a)<\omega(a+\hat\eta_1+\hat\eta_2)$ 
for $a\in A_p$ and $\eta_1,\eta_2\in\V\setminus\V_1$. 
\end{enumerate}
Let $C_1'$ be the set of these $1$-simplices. 
Then $C_1\cup C_1'$ is a set of representatives 
for the $1$-simplices of $Z_{p,p+2}$ mod $[[G|Y]]$, 

Likewise, we let $C_2$ be the set of all $2$-simplices of $\lvert C_0\rvert$, 
i.e. 
\[
C_2=\{\omega(a)<\omega(a+\hat\xi_1)<\omega(a+\hat\xi_1+\hat\xi_2)
\mid a\in A_p,\ \xi_1,\xi_2\in\V_1\}. 
\]
In order to obtain a set of representatives 
for the $2$-simplices mod $[[G|Y]]$, we consider the following $2$-simplices. 
\begin{enumerate}
\item[(v)] $\omega(a)<\omega(a+\hat\xi)<\omega(a+\hat\xi+\hat\eta)$ 
for $a\in A_p$, $\xi\in\V_1$ and $\eta\in\V\setminus\V_1$. 
\item[(vi)] $\omega(a)<\omega(a+\hat\eta)<\omega(a+\hat\xi+\hat\eta)$ 
for $a\in A_p$, $\xi\in\V_1$ and $\eta\in\V\setminus\V_1$. 
\item[(vii)] $\omega(a)<\omega(a+\hat\eta_1)
<\omega(a+\hat\eta_1+\hat\eta_2)$ 
for $a\in A_p$ and $\eta_1,\eta_2\in\V\setminus\V_1$. 
\end{enumerate}
Let $C_2'$ be the set of these $2$-simplices. 
Then $C_2\cup C_2'$ is a set of representatives 
for the $2$-simplices of $Z_{p,p+2}$ mod $[[G|Y]]$, 

For each vertex $u$ of $Z_{p,p+2}$, 
let $\Sigma(u)$ denote the stabilizer of $u$. 
By Lemma \ref{stabilizer} (2), 
$\Sigma(u)$ is naturally identified with 
a subgroup of the symmetric group on $u$ and 
is isomorphic to $\bigoplus_{\xi\in\V}\Sigma_{\rank(u)(\xi)}$. 
Similarly, for each $1$-simplex $e$ of $Z_{p,p+2}$, 
we denote the stabilizer of $e$ by $\Sigma(e)$. 
We define an abstract group $\Gamma$ to be the free product of 
the groups $\Sigma(u)$ for $u\in C_0$ and 
a free group generated by elements $g_e$ for $e\in C_1\cup C_1'$. 
It is clear that the group $\Gamma$ is finitely presented. 

Next, 
we would like to introduce a surjective homomorphism $\pi:\Gamma\to[[G|Y]]$. 
To this end, for each $e=\{u{<}v\}\in C_1\cup C_1'$, 
we choose $\gamma_e\in[[G|Y]]$ such that $\gamma_e^{-1}(v)\in C_0$ as follows. 
First, for $e\in C_1$, we put $\gamma_e=1$. 
Let us consider $\gamma_e$ for $e\in C_1'$. 
For $e=\{\omega(a){<}\omega(a{+}\hat\eta)\}$ as in (i) and (ii), 
we put 
\[
\gamma_e=\beta_{\zeta,k}\theta_\eta\beta_{\zeta,k}, 
\]
where $k=a(\zeta)+1$. 
For $e=\{\omega(a){<}\omega(a{+}\hat\xi{+}\hat\eta)\}$ as in (iii), 
we put 
\[
\gamma_e=\begin{cases}\beta_{\zeta,k}\theta_\eta\beta_{\zeta,k}&\xi\neq\zeta\\
\beta_{\zeta,k+1}\theta_\eta\beta_{\zeta,k+1}&\xi=\zeta, \end{cases}
\]
where $k=a(\zeta)+1$. 
For $e=\{\omega(a){<}\omega(a{+}\hat\eta_1{+}\hat\eta_2)\}$ as in (iv), 
we put 
\[
\gamma_e=\begin{cases}
(\beta_{\zeta,k}\theta_{\eta_1}\beta_{\zeta,k})
(\beta_{\zeta,k+1}\theta_{\eta_2}\beta_{\zeta,k+1})&\eta_1<\eta_2\\
(\beta_{\zeta,k}\theta_{\eta_1}\beta_{\zeta,k})
(\beta_{\zeta,k+1}\beta_{\eta_2,2}\theta_{\eta_2}
\beta_{\eta_2.2}\beta_{\zeta,k+1})&\eta_1=\eta_2, \end{cases}
\]
where $k=a(\zeta)+1$. 
Now we set $\pi(g_e)=\gamma_e$ for $e\in C_1\cup C_1'$. 
This, together with the canonical inclusions $\Sigma(u)\to[[G|Y]]$, 
gives rise to a homomorphism $\pi:\Gamma\to[[G|Y]]$. 

The following theorem says that 
the kernel of $\pi:\Gamma\to[[G|Y]]$ is finitely generated 
as a normal subgroup of $\Gamma$. 
In order to describe generators of the kernel, 
we need to introduce a little more notations. 
The $1$-skeleton of $\lvert C_0\rvert$ is $C_0\cup C_1$, 
which can be regarded as a graph. 
Consider a maximal tree of this graph. 
Every vertex of $C_0$ is contained in the tree. 
Let us identify this tree with the corresponding subset $T\subset C_1$. 
For $e=\{u{<}v\}\in C_1\cup C_1'$, 
we let $i_e:\Sigma(e)\to\Sigma(u)$ be the canonical inclusion and 
let $c_e:\Sigma(e)\to\Sigma(\tilde v)$ be the homomorphism 
given by $\alpha\mapsto\gamma_e^{-1}\alpha\gamma_e$, 
where $\tilde v=\gamma_e^{-1}v\in C_0$. 
For each $2$-simplex $\tau\in C_2\cup C_2'$, 
as discussed in \cite[Section 1]{Br84JPAA}, 
we can associate a `relator' $r_\tau\in\Gamma$ 
by thinking of the boundary of $\tau$ as a closed path. 

\begin{theorem}[{\cite[Theorem 1]{Br84JPAA}}]
The homomorphism $\pi:\Gamma\to[[G|Y]]$ is surjective and 
its kernel is generated by the following elements as a normal subgroup. 
\begin{enumerate}
\item $g_e$ for $e\in T$. 
\item $g_e^{-1}i_e(\alpha)g_ec_e(\alpha^{-1})$ 
for $e\in C_1\cup C_1'$ and $\alpha\in\Sigma(e)$. 
\item $r_\tau$ for $\tau\in C_2\cup C_2'$. 
\end{enumerate}
\end{theorem}

The presentation given in the theorem above contains much redundancy. 
We would like to obtain a `smaller' presentation of $[[G|Y]]$. 
First, let us look at $e\in C_1\setminus T$. 
Since $T$ is a maximal tree of the $1$-skeleton of $\lvert C_0\rvert$, 
there exists a loop in $T\cup\{e\}$ containing $e$. 
Then, from $\pi(r_\tau)=1$ for all $\tau\in C_2$, we obtain 
\begin{enumerate}
\item[(4)] $\pi(g_e)=1$ for all $e\in C_1\setminus T$. 
\end{enumerate}
Let us now look at the elements $r_\tau\in\Gamma$ 
associated with $\tau\in C_2'$. 
From the $2$-simplices 
$\tau=\{\omega(a){<}\omega(a{+}\hat\xi){<}\omega(a{+}\hat\xi{+}\hat\eta)\}$ 
as in (v), we get 
\begin{enumerate}
\item[(5)] $\pi(g_e)=\pi(g_{e'})$ 
for all $e=\{\omega(a{+}\hat\xi){<}\omega(a{+}\hat\xi{+}\hat\eta)\}$ and 
$e'=\{\omega(a){<}\omega(a{+}\hat\xi{+}\hat\eta)\}$ 
with $a\in A_p$, $\xi\in\V_1$ and $\eta\in\V\setminus\V_1$. 
\end{enumerate}
Next, let $\tau
=\{\omega(a){<}\omega(a{+}\hat\eta){<}\omega(a{+}\hat\xi{+}\hat\eta)\}$ 
be a $2$-simplex as in (vi). 
In case $\xi\neq\zeta$, we obtain 
\begin{enumerate}
\item[(6)] $\pi(g_e)=\pi(g_{e'})$ 
for all $e=\{\omega(a){<}\omega(a{+}\hat\eta)\}$ and 
$e'=\{\omega(a){<}\omega(a{+}\hat\xi{+}\hat\eta)\}$ 
with $a\in A_p$, $\xi\in\V_1\setminus\{\zeta\}$ and $\eta\in\V\setminus\V_1$. 
\end{enumerate}
In case $\xi=\zeta$, we have 
\begin{enumerate}
\item[(7)] $\pi(\beta g_e\beta)=\pi(g_{e'})$ 
for all $e=\{\omega(a){<}\omega(a{+}\hat\eta)\}$ and 
$e'=\{\omega(a){<}\omega(a{+}\hat\zeta{+}\hat\eta)\}$ 
with $a\in A_p$ and $\eta\in\V\setminus\V_1$, 
where $\beta\in\Sigma(\omega(a))$ denotes the simple transposition 
on $\{f(\zeta,a(\zeta){+}1),\ f(\zeta,a(\zeta){+}2)\}$. 
\end{enumerate}
Finally, let us look at $2$-simplices 
$\tau
=\{\omega(a){<}\omega(a{+}\hat\eta_1){<}\omega(a{+}\hat\eta_1{+}\hat\eta_2)\}$ 
as in (vii). 
Recall that we have fixed a linear order on $\V\setminus\V_1$. 
In case $\eta_1<\eta_2$, one obtains 
\begin{enumerate}
\item[(8)] $\pi(g_eg_{e'})=\pi(g_{e''})$ 
for all $e=\{\omega(a){<}\omega(a{+}\hat\eta_1)\}$, 
$e'=\{\omega(a{+}\hat\zeta){<}\omega(a{+}\hat\zeta{+}\hat\eta_2)\}$ and 
$e''=\{\omega(a){<}\omega(a{+}\hat\eta_1{+}\hat\eta_2)\}$ 
with $a\in A_p$, $\eta_1,\eta_2\in\V\setminus\V_1$ and $\eta_1<\eta_2$. 
\end{enumerate}
In case $\eta_1>\eta_2$, one obtains 
\begin{enumerate}
\item[(9)] $\pi(\beta g_eg_{e'}\beta)=\pi(g_{e''})$ 
for all $e=\{\omega(a){<}\omega(a{+}\hat\eta_1)\}$, 
$e'=\{\omega(a{+}\hat\zeta){<}\omega(a{+}\hat\zeta{+}\hat\eta_2)\}$ and 
$e''=\{\omega(a){<}\omega(a{+}\hat\eta_1{+}\hat\eta_2)\}$ 
with $a\in A_p$, $\eta_1,\eta_2\in\V\setminus\V_1$ and $\eta_1>\eta_2$, 
where $\beta\in\Sigma(\omega(a))$ is the simple transposition 
on $\{f(\zeta,a(\zeta){+}1),\ f(\zeta,a(\zeta){+}2)\}$. 
\end{enumerate}
In case $\eta_1=\eta_2=\eta$, one has 
\begin{enumerate}
\item[(10)] $\pi(g_e\sigma g_{e'}\sigma)=\pi(g_{e''})$ 
for all $e=\{\omega(a){<}\omega(a{+}\hat\eta)\}$, 
$e'=\{\omega(a{+}\hat\zeta){<}\omega(a{+}\hat\zeta{+}\hat\eta)\}$ and 
$e''=\{\omega(a){<}\omega(a{+}2\hat\eta)\}$ 
with $a\in A_p$, $\eta\in\V\setminus\V_1$, 
where $\sigma\in\Sigma(\omega(a))$ is the simple transposition 
on $\{f(\eta,1),\ f(\eta,2)\}$. 
\end{enumerate}

For each $\eta\in\V\setminus\V_1$, consider the $1$-simplex 
\[
e=\{\omega((p{+}1)\hat\zeta){<}\omega((p{+}1)\hat\zeta+\hat\eta)\}\in C_1'
\]
and set $g_\eta=g_e$. 
Let $\Gamma_0$ be the free product of 
the groups $\Sigma(u)$ for $u\in C_0$ and 
the free group generated by $g_\eta$ for $\eta\in\V\setminus\V_1$. 
We regard $\Gamma_0$ as a subgroup of $\Gamma$. 
As a consequence of the discussion above, the following theorem is obtained. 

\begin{theorem}
The homomorphism $\pi:\Gamma_0\to[[G|Y]]$ is surjective and 
its kernel is generated by the following elements as a normal subgroup. 
\begin{itemize}
\item $i_e(\alpha)c_e(\alpha^{-1})$ 
for any $e\in C_1$ and $\alpha\in\Sigma(e)$. 
\item $(\beta g_\eta\beta)^{-1}i_e(\alpha)(\beta g_\eta\beta)c_e(\alpha^{-1})$ 
for any $e=\{\omega(a){<}\omega(a+\hat\eta)\}\in C_1'$ and 
$\alpha\in\Sigma(e)$, 
where $a$ is in $A_p\cup A_{p+1}$, $\eta$ is in $\V\setminus\V_1$ and 
$\beta\in\Sigma(\omega(a))$ is the simple transposition 
on $\{f(\zeta,a(\zeta){+}1),\ f(\zeta,p{+}2)\}$. 
\item $[g_{\eta_1},\beta g_{\eta_2}\beta]$ 
for any distinct $\eta_1,\eta_2\in\V\setminus\V_1$, 
where $\beta\in\Sigma(\omega(p\hat\zeta))$ is the simple transposition 
on $\{f(\zeta,p{+}1),\ f(\zeta,p{+}2)\}$. 
\item $(\beta g_\eta\beta\sigma g_\eta\sigma)^{-1}i_e(\alpha)
(\beta g_\eta\beta\sigma g_\eta\sigma)c_e(\alpha^{-1})$ 
for any $e=\{\omega(p\hat\zeta){<}\omega(p\hat\zeta+2\hat\eta)\}\in C_1'$ 
and $\alpha\in\Sigma(e)$, 
where $\eta$ is in $\V\setminus\V_1$, 
$\beta\in\Sigma(\omega(p\hat\zeta))$ denotes the simple transposition 
on $\{f(\zeta,p{+}1),\ f(\zeta,p{+}2)\}$ and 
$\sigma\in\Sigma(\omega(p\hat\zeta))$ denotes the simple transposition 
on $\{f(\eta,1),\ f(\eta,2)\}$. 
\end{itemize}
\end{theorem}

As an application of the theorem above, we obtain the following. 

\begin{corollary}\label{abel}
\begin{enumerate}
\item The abelianization $[[G|Y]]/D([[G|Y]])$ of 
the topological full group $[[G|Y]]$ is isomorphic to 
$(H_0(G)\otimes\Z_2)\oplus H_1(G)$. 
\item The group $[[G|Y]]_0/D([[G|Y]])$ is isomorphic to 
$H_0(G)\otimes\Z_2$. 
\item The group $[[G|Y]]$ is simple if and only if $H_0(G)$ is $2$-divisible. 
\end{enumerate}
\end{corollary}
\begin{proof}
(1)
Since $\Gamma_0$ is the free product of 
the groups $\Sigma(u)$ for $u\in C_0$ and 
the free group generated by $g_\eta$ for $\eta\in\V\setminus\V_1$, 
its abelianization $\Gamma_0/D(\Gamma_0)$ is isomorphic to 
\[
\bigoplus_{u\in C_0}(\Z_2)^\V\oplus\Z^{\V\setminus\V_1}. 
\]
By the theorem above, we have $\pi(i_e(\alpha))=\pi(c_e(\alpha))$ 
for any $e=\{\omega(a){<}\omega(a+\hat\xi)\}$ and $\alpha\in\Sigma(e)$, 
where $a$ is in $A_p\cup A_{p+1}$ and $\xi$ is in $\V_1$. 
These equations collapse the group $\Gamma_0/D(\Gamma_0)$ to 
\[
(\Z_2)^\V\oplus\Z^{\V\setminus\V_1}. 
\]
By the theorem above, we also have $\pi(i_e(\alpha))=\pi(c_e(\alpha))$ 
for any $e=\{\omega(a){<}\omega(a+2\hat\xi)\}$ and $\alpha\in\Sigma(e)$, 
where $a$ is in $A_p$ and $\xi$ is in $\V_1$. 
These equations collapse the group above to 
\[
((\Z_2)^\V/\Ima(\id-M^t))\oplus\Z^{\V\setminus\V_1}, 
\]
which is isomorphic to $(H_0(G)\otimes\Z_2)\oplus H_1(G)$. 
The theorem above tells us that no further collapsing occurs. 

(2) 
This follows from (1) and Theorem \ref{surjective}. 

(3) 
As mentioned in Section 6.1, 
$H_0(G)$ is a finitely generated abelian group and 
$H_1(G)$ is isomorphic to the torsion-free part of $H_0(G)$. 
Therefore if $H_0(G)$ is $2$-divisible, 
then $(H_0(G)\otimes\Z_2)\oplus H_1(G)=0$, and so $[[G|Y]]=D([[G|Y]])$. 
By Lemma \ref{SFT>pim} and Theorem \ref{simple2}, $D([[G|Y]])$ is simple. 
Hence $[[G|Y]]$ is simple. 
The other implication is shown in a similar way. 
\end{proof}

The corollary above, together with Theorem \ref{fg} and Theorem \ref{finite}, 
implies the following. 

\begin{corollary}\label{Dfinite}
\begin{enumerate}
\item The group $D([[G|Y]])$ is finitely generated. 
\item If $H_0(G)$ is a finite group, 
then the groups $[[G|Y]]_0$ and $D([[G|Y]])$ are of type F$_\infty$. 
In particular, they are finitely presented and are of type FP$_\infty$. 
\end{enumerate}
\end{corollary}

In general, we do not know 
if $[[G|Y]]_0$ and $D([[G|Y]])$ are finitely presented or not.

%%%%%%%%%%%%%%%%%%%%%%%%%%%%%%%%%%%%%%%%%%%%%%%%%%%%%%%%%%%%
\subsection{Examples}

%%%%%%%%%%%%%%%%%%%%%%%%%%%%%%%%%%%%%%%%%%%%%%%%%%%%%%%%%%%%
\subsubsection{Full shifts}

We consider \'etale groupoids arising from full shifts. 
Let $n\in\N\setminus\{1\}$ and let $r\in\N$. 
Let $(\V,\E)$ be a finite directed graph such that 
$\#\V=r$ and its adjacency matrix is 
\[
M=\begin{bmatrix}0&0&\ldots&0&n\\1&0&\ldots&0&0\\0&1&\ldots&0&0\\
\vdots&\vdots&\ddots&\vdots&\vdots\\0&0&\ldots&1&0\end{bmatrix}. 
\]
Note that $\det(\id-M^t)=1{-}n$. 
Let $(X_{n,r},\sigma_{n,r})$ be the shift of finite type 
associated with $(\V,\E)$. 
When $r=1$, the shift $(X_{n,1},\sigma_{n,1})$ is 
the full shift over $n$ symbols. 
Let $G_{n,r}$ be the \'etale groupoid of the shift $(X_{n,r},\sigma_{n,r})$ 
(see Section 6.1). 
One has 
\[
H_0(G_{n,r})\cong\Coker(\id-M^t)\cong\Z_{n-1}\quad\text{and}\quad 
H_1(G_{n,r})\cong\Ker(\id-M^t)=0,
\]
where $\id-M^t$ is thought of as a homomorphism from $\Z^r$ to $\Z^r$. 
The equivalence class of $1_{X_{n,r}}$ in $H_0(G_{n,r})$ 
corresponds to $r\in\Z_{n-1}$. 
As mentioned in Remark \ref{HigThomp}, 
the topological full group $[[G_{n,r}]]$ is 
naturally isomorphic to the Higman-Thompson group $V_{n,r}$. 
From $H_1(G_{n,r})=0$, we obtain $[[G_{n,r}]]=[[G_{n,r}]]_0$. 
By Lemma \ref{SFT>pim} and Theorem \ref{simple2}, $D([[G_{n,r}]])$ is simple. 
Moreover, Corollary \ref{abel} says that 
\[
[[G_{n,r}]]/D([[G_{n,r}]])\cong\Z_{n-1}\otimes\Z_2
=\begin{cases}0&\text{$n$ is even}\\\Z_2&\text{$n$ is odd. }\end{cases}
\]
By Theorem \ref{iso} and Theorem \ref{classify}, 
$[[G_{n,r}]]$ (or $D([[G_{n,r}]])$) is isomorphic to 
$[[G_{m,s}]]$ (or $D([[G_{m,s}]])$) if and only if 
there exists an isomorphism $\phi:\Z_{n-1}\to\Z_{m-1}$ 
such that $\phi(r)=s$. 
This reproves the main result of \cite{P11JA}. 
By Theorem \ref{finite} and Corollary \ref{Dfinite}, 
$[[G_{n,r}]]$ and $D([[G_{n,r}]])$ are of type F$_\infty$.

%%%%%%%%%%%%%%%%%%%%%%%%%%%%%%%%%%%%%%%%%%%%%%%%%%%%%%%%%%%%
\subsubsection{The golden mean shift}

Let $X$ be the set of sequences of $\{0,1\}$ 
in which `11' does not appear, that is, 
\[
X=\{(x_n)_n\in\{0,1\}^\N\mid\text{ if $x_n=1$, then $x_{n+1}=0$}\}. 
\]
Let $\sigma:X\to X$ be the one-sided shift. 
The dynamical system $(X,\sigma)$ is called the golden mean shift, 
and the corresponding matrix is 
\[
M=\begin{bmatrix}1&1\\1&0\end{bmatrix}. 
\]
The Perron eigenvalue of $M$ is the golden number $(1+\sqrt{5})/2$, 
and $\det(\id-M^t)=-1$. 
Consider the \'etale groupoid $G$ of $(X,\sigma)$. 
One has 
\[
H_0(G)\cong\Coker(\id-M^t)=0\quad\text{and}\quad 
H_1(G)\cong\Ker(\id-M^t)=0. 
\]
It follows from Theorem \ref{classify} that 
$G$ is isomorphic to $G_{2,1}$ discussed above. 
Hence $[[G]]$ is isomorphic to the Higman-Thompson group $V_{2,1}$. 

In \cite{KMW98ETDS}, for an arbitrary real number $\beta>1$, 
the $\beta$-shifts and 
the associated $C^*$-algebras $\mathcal{O}_\beta$ were studied. 
It is known that the $\beta$-shift is a shift of finite type 
if and only if the $\beta$-expansion of $1$ is finite 
(see \cite[Proposition 3.8]{KMW98ETDS} for instance). 
The golden mean shift $(X,\sigma)$ is equal to 
the $\beta$-shift with $\beta=(1+\sqrt{5})/2$. 
Note that the ring $\Z[\beta,\beta^{-1}]$ equals $\Z+\beta\Z$, 
because $\beta^2-\beta-1=0$. 
We would like to observe that 
$[[G]]$ is identified with a group of PL bijections on the unit interval. 
As in Remark \ref{HigThomp}, 
let $F_{\beta,1}$ be the set of all PL homeomorphisms $f:[0,1]\to[0,1]$ 
with finitely many singularities 
such that all singularities of $f$ are in $\Z[\beta,\beta^{-1}]$ and 
the derivative of $f$ at any non-singular point 
is $\beta^k$ for some $k\in\Z$. 
The group $F_{\beta,1}$ was shown to be finitely presented and 
is of type FP$_\infty$ in \cite{C00Illi}. 
In the same fashion as Remark \ref{HigThomp}, 
one can define the group $V_{\beta,1}$ containing $F_{\beta,1}$. 
Define a continuous map $\rho:X\to[0,1]$ by 
\[
\rho((x_n)_n)=\sum_{n=1}^\infty\frac{x_n}{\beta^n}
\quad\forall (x_n)_n\in X. 
\]
Then for every $t\in[0,1]$, 
\[
\#\rho^{-1}(t)=\begin{cases}2&t\in\Z[\beta,\beta^{-1}]\cap[0,1)\\
1&\text{otherwise. }\end{cases}
\]
Any cylinder set of $X$ is mapped to 
an interval $[a,b]$ with $a,b\in\Z[\beta,\beta^{-1}]$. 
It is not so hard to see that 
$[[G]]$ is naturally isomorphic to a subgroup of $V_{\beta,1}$ 
via the map $\rho$. 
Thus $V_{2,1}$ is embeddable into $V_{\beta,1}$.

%%%%%%%%%%%%%%%%%%%%%%%%%%%%%%%%%%%%%%%%%%%%%%%%%%%%%%%%%%%%
\subsubsection{A subshift with non-trivial $H_1$}

Let $(X,\sigma)$ be the irreducible shift of finite type 
arising from the matrix 
\[
M=\begin{bmatrix}2&1\\1&2\end{bmatrix}, 
\]
and let $G$ be the \'etale groupoid of $(X,\sigma)$. 
One has 
\[
H_0(G)\cong\Coker(\id-M^t)\cong\Z\quad\text{and}\quad 
H_1(G)\cong\Ker(\id-M^t)=\Z. 
\]
The equivalence class of $1_X$ in $H_0(G)$ corresponds to $0\in\Z$. 
By Lemma \ref{SFT>pim} and Theorem \ref{simple2}, $D([[G]])$ is simple. 
Theorem \ref{surjective} says that $[[G]]/[[G]]_0\cong\Z$. 
By Corollary \ref{abel}, $[[G]]_0/D([[G]])\cong\Z_2$. 
By Theorem \ref{finite}, $[[G]]$ is of type F$_\infty$. 
By Theorem \ref{fg} and Corollary \ref{Dfinite}, 
$[[G]]_0$ and $D([[G]])$ are finitely generated. 

We would like to give a finite presentation for $[[G]]$. 
The notation is borrowed from Section 6.5 and Section 6.6. 
Let $(\V,\E)$ be a finite directed graph whose adjacency matrix is $M$. 
Let $\V=\{\zeta,\eta\}$. 
Then $\omega=\{D_\zeta,D_\eta\}$ is in $\B$ and $\rank(\omega)=(1,1)$. 
It is easy to see 
$\rank(u)=(n,n)$ for $u\in\B_\omega$ with $\height(u)=n{-}1$. 
By Lemma \ref{Kuconnected}, 
there exists $n$ such that $K_u$ is $1$-connected 
for any $u\in\B_\omega$ satisfying $\rank(u)=(n,n)$. 
In view of the remark given in \cite[Section 5]{Br89MSRI} 
(this argument is due to K. Vogtmann), one can take $n=9$ in this case. 
Thus, 
the simplicial complex $Z_{5,7}$ is $1$-connected. 
Let $u_6\in\B_\omega$ be such that $\rank(u_6)=(6,6)$. 
Choose two elements $V_1,V_2\in u_6$ such that $s(V_1)=s(V_2)=D_\zeta$. 
Let $u_7$ be the elementary expansion of $u_6$ with respect to $V_1$ and 
let $u_8$ be the elementary expansion of $u_7$ with respect to $V_2$. 
Then $\rank(u_k)=(k,k)$, $\{u_6{<}u_7{<}u_8\}$ is a $2$-simplex of $Z_{5,7}$ 
and $\{u_6,u_7,u_8\}$ is a set of representatives 
for the vertices of $Z_{5,7}$ mod $[[G]]$. 
The stabilizer of $u_k$ is isomorphic to $\Sigma_k\times\Sigma_k$ 
for each $k=6,7,8$. 
To simplify notation, we denote products of symmetric groups 
by $\Sigma_{k,l}=\Sigma_k\times\Sigma_l$, etc. 
Let $\Gamma_0$ be the free product of 
$\Sigma_{6,6}$, $\Sigma_{7,7}$, $\Sigma_{8,8}$ 
and $\Z$ generated by $g=g_\eta$. 
As discussed in Section 6.6, 
there exists a natural surjection $\pi:\Gamma_0\to[[G]]$. 

Consider 
the direct limit (in the sense of J-P. Serre) of the following diagram. 
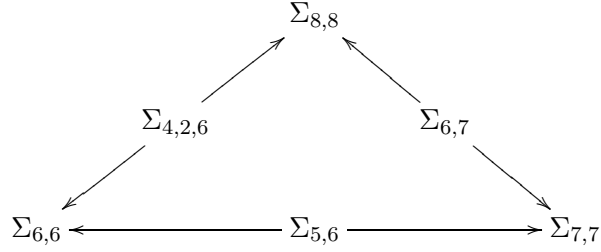
\[
\xymatrix{&&\Sigma_{8,8}&&\\
&\Sigma_{4,2,6}\ar[ur]\ar[dl]&&
\Sigma_{6,7}\ar[ul]\ar[dr]&\\
\Sigma_{6,6}&&\Sigma_{5,6}\ar[ll]\ar[rr]&&
\Sigma_{7,7}}
\]
Regard $\Sigma_{k,k}$ as the group of permutations of 
letters $\{1_k,2_k,\dots,k_k\}\cup\{1'_k,2'_k,\dots,k'_k\}$ for $k=6,7,8$. 
The maps $\Sigma_{5,6}\hookrightarrow\Sigma_{k,k}$ for $k=6,7$ are 
obtained by letting $\Sigma_5$ permute $\{1_k,2_k,3_k,4_k,k_k\}$ 
and letting $\Sigma_6$ permute $\{1'_k,2'_k,\dots,6'_k\}$. 
Similarly, the maps $\Sigma_{6,7}\hookrightarrow\Sigma_{k,k}$ for $k=7,8$ 
are obtained by letting $\Sigma_6$ permute $\{1_k,\dots,6_k\}$ 
and letting $\Sigma_7$ permute $\{1'_k,\dots,7'_k\}$. 
The map $\Sigma_{4,2,6}\hookrightarrow\Sigma_{6,6}$ is obtained 
by letting $\Sigma_{4,2,6}$ 
permute $\{1_6,\dots,4_6\}$, $\{5_6,6_6\}$ and $\{1'_6,\dots,6'_6\}$. 
The embedding $\Sigma_{4,2,6}\hookrightarrow\Sigma_{8,8}$ is not 
of the standard type. 
The subgroup $\Sigma_4\times\Sigma_6$ permutes 
$\{1_8,\dots,4_8\}$ and $\{1'_8,\dots,6'_8\}$, 
but the non-trivial element of $\Sigma_2$ maps to 
the product $(5_8 \ 7_8)(6_8 \ 8_8)(7'_8 \ 8'_8)$ of three transpositions. 
Let $\Gamma_1$ be the direct limit of this triangle of groups. 
Then the restriction of the homomorphism $\pi:\Gamma_0\to[[G]]$ 
to the subgroup $\Sigma_{6,6}*\Sigma_{7,7}*\Sigma_{8,8}$ 
factors through $\Gamma_1$. 

Let us now turn to discussion on the relations involving $g=g_\eta$. 
We consider the following diagram. 
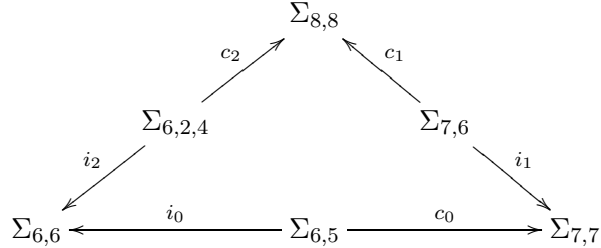
\[
\xymatrix{&&\Sigma_{8,8}&&\\
&\Sigma_{6,2,4}\ar[ur]^{c_2}\ar[dl]_{i_2}&&
\Sigma_{7,6}\ar[ul]_{c_1}\ar[dr]^{i_1}&\\
\Sigma_{6,6}&&\Sigma_{6,5}\ar[ll]_{i_0}\ar[rr]^{c_0}&&
\Sigma_{7,7}}
\]
The maps $i_0$ and $c_0$ are obtained by 
letting $\Sigma_6$ permute $\{1_k,\dots,5_k,k_k\}$ and 
letting $\Sigma_5$ permute $\{2'_k,\dots,6'_k\}$ for $k=6,7$. 
The maps $i_1$ and $c_1$ are obtained by 
letting $\Sigma_7$ permute $\{1_k,\dots,7_k\}$ and 
letting $\Sigma_6$ permute $\{2'_k,\dots,7'_k\}$ for $k=7,8$. 
The map $i_2$ is obtained by letting $\Sigma_{6,2,4}$ permute 
$\{1_6,\dots,6_6\}$, $\{1'_6,2'_6\}$ and $\{3'_6,\dots,6'_6\}$. 
The map $c_2$ is not of the standard type. 
The subgroup $\Sigma_6\times\Sigma_4$ permutes 
$\{1_8,\dots,5_8,7_8\}$ and $\{3'_8,\dots,6'_8\}$, 
but the non-trivial element of $\Sigma_2$ maps to 
the product $(1'_8 \ 2'_8)(7'_8 \ 8'_8)(6_8 \ 8_8)$ of three transpositions. 
Now the relations involving $g$ are given by 
\[
(\beta_0g\beta_0)^{-1}i_0(\sigma)
(\beta_0g\beta_0)c_0(\sigma)\quad \forall\sigma\in\Sigma_{6,5},\quad 
g^{-1}i_1(\sigma)
gc_1(\sigma)\quad \forall\sigma\in\Sigma_{7,6}, 
\]
and 
\[
(\beta_0g\beta_0\beta_1g\beta_1)^{-1}i_2(\sigma)
(\beta_0g\beta_0\beta_1g\beta_1)c_2(\sigma)\quad 
\forall\sigma\in\Sigma_{6,2,4}, 
\]
where $\beta_0=(5_6 \ 6_6)$ and $\beta_1=(1'_6 \ 2'_6)$. 
Hence $[[G]]$ is isomorphic to the quotient of 
the free product of $\Gamma_1$ and $\Z=\langle g\rangle$ 
by the normal subgroup generated by these relations.

%%%%%%%%%%%%%%%%%%%%%%%%%%%%%%%%%%%%%%%%%%%%%%%%%%%%%%%%%%%%
\subsubsection{Boundary actions I}

We consider the boundary actions of free products of finite groups. 
These actions and related $C^*$-algebras were studied 
in \cite{S91IJM,O02RIMS}. 
Let $P$ and $Q$ be non-trivial finite groups 
(we exclude the case $P=Q=\Z_2$). 
Put $p=\#P$ and $q=\#Q$. 
Let $\E$ be the disjoint union of $P\setminus\{1\}$ and $Q\setminus\{1\}$. 
We define 
\[
X=\{(x_n)_n\in\E^\N\mid\text{if $x_n\in P$ (resp. $x_n\in Q$), 
then $x_{n+1}\in Q$ (resp. $x_{n+1}\in P$)}\}. 
\]
The topological space $X$ is naturally identified with 
the hyperbolic boundary of the free product $P*Q$, 
and admits a natural action $\phi:P*Q\curvearrowright X$ by homeomorphisms, 
which is called the boundary action 
(see \cite[Proposition 5.5]{O02RIMS}, for example). 
Let $\sigma:X\to X$ be the one-sided shift. 
Then $(X,\sigma)$ is the shift of finite type associated with the matrix 
\[
M=\begin{bmatrix}0&p{-}1\\ q{-}1&0\end{bmatrix}. 
\]
We have $\det(\id-M^t)=1-(p-1)(q-1)$. 
It is not so hard to see that 
the transformation groupoid $G_\phi$ is canonically identified with 
the \'etale groupoid arising from $(X,\sigma)$ 
(see \cite[Definition 2.1]{M12PLMS} for the definition of $G_\phi$). 
By Lemma \ref{SFT>pim}, 
the groupoid $G_\phi$ is purely infinite and minimal. 
One has 
\[
H_0(G_\phi)\cong\Coker(\id-M^t)\cong\Z_{n-1}\quad\text{and}\quad 
H_1(G_\phi)\cong\Ker(\id-M^t)=0,
\]
where $n=(p-1)(q-1)$. 
The equivalence class of $1_X$ in $H_0(G_\phi)$ corresponds to 
$p\in\Z_{n-1}$ (or $q\in\Z_{n-1}$). 
Hence Theorem \ref{classify} tells us that 
$G_\phi$ is isomorphic to $G_{n,p}$ (or $G_{n,q}$) discussed in Section 6.7.1.

%%%%%%%%%%%%%%%%%%%%%%%%%%%%%%%%%%%%%%%%%%%%%%%%%%%%%%%%%%%%
\subsubsection{Boundary actions II}

Let $k\in\N\setminus\{1\}$ and 
let $\F_k$ denote the free group on $k$ generators $s_1,s_2,\dots,s_k$. 
Let $\E=\{s_1,s_1^{-1},s_2,s_2^{-1},\dots,s_k,s_k^{-1}\}$. 
Set 
\[
X_k=\{(x_n)_n\in\E^\N\mid
\text{if $x_n=s_i$ (resp. $x_n=s_i^{-1}$), 
then $x_{n+1}\neq s_i^{-1}$ (resp. $x_{n+1}\neq s_i$)}\}. 
\]
In the same way as the preceding example, 
the space $X_k$ is naturally identified with 
the hyperbolic boundary of the free group $\F_k$. 
Let $\phi_k:\F_k\curvearrowright X_k$ be the boundary action. 
The one-sided shift $\sigma_k:X_k\to X_k$ is a shift of finite type 
associated with the $2k\times2k$ matrix 
\[
M=\begin{bmatrix}1&0&1&1&\ldots&\ldots&1&1\\0&1&1&1&\ldots&\ldots&1&1\\
1&1&1&0&\ldots&\ldots&1&1\\1&1&0&1&\ldots&\ldots&1&1\\
\vdots&\vdots&\vdots&\vdots&\ddots&&\vdots&\vdots\\
\vdots&\vdots&\vdots&\vdots&&\ddots&\vdots&\vdots\\
1&1&1&1&\ldots&\ldots&1&0\\1&1&1&1&\ldots&\ldots&0&1\end{bmatrix}. 
\]
In the same way as the preceding example, 
the transformation groupoid $G_{\phi_k}$ is canonically identified with 
the \'etale groupoid arising from $(X_k,\sigma_k)$. 
By Lemma \ref{SFT>pim}, 
the groupoid $G_{\phi_k}$ is purely infinite and minimal. 
We have 
\[
H_0(G_{\phi_k})\cong\Coker(\id-M^t)\cong\Z^k\oplus\Z_{k-1}
\]
and 
\[
H_1(G_{\phi_k})\cong\Ker(\id-M^t)\cong\Z^k. 
\]
The equivalence class of $1_X$ in $H_0(G_{\phi_k})$ 
corresponds to $(0,1)\in\Z^k\oplus\Z_{k-1}$. 
By Theorem \ref{iso}, if $k\neq l$, then 
$[[G_{\phi_k}]]$ (or $[[G_{\phi_k}]]_0$, $D([[G_{\phi_k}]])$) 
is not isomorphic to 
$[[G_{\phi_l}]]$ (or $[[G_{\phi_l}]]_0$, $D([[G_{\phi_l}]])$). 
It follows from Theorem \ref{surjective} that 
$[[G_{\phi_k}]]/[[G_{\phi_k}]]_0$ is isomorphic to $\Z^k$. 
Theorem \ref{simple2} tells us that $D([[G_{\phi_k}]])$ is simple. 
By Corollary \ref{abel}, 
\[
[[G_{\phi_k}]]_0/D([[G_{\phi_k}]])\cong(\Z^k\oplus\Z_{k-1})\otimes\Z_2
=\begin{cases}(\Z_2)^k&\text{$k$ is even}\\
(\Z_2)^{k+1}&\text{$k$ is odd. }\end{cases}
\]
By Theorem \ref{finite}, 
we know that $[[G_{\phi_k}]]$ is of type F$_\infty$. 
By Theorem \ref{fg} and Corollary \ref{Dfinite}, 
$[[G_{\phi_k}]]_0$ and $D([[G_{\phi_k}]])$ are finitely generated. 

\bigskip
\bigskip

\noindent
\textbf{Acknowledgments}

I am very much indebted to D. Tamaki 
for helpful discussions on homotopy theory and algebraic topology. 
I also thank R. Grigorchuk and K. Matsumoto 
for many valuable suggestions and comments.

\end{document}